\documentclass{amsart}
\usepackage{latexsym,amssymb,amsmath,amsthm,amscd,graphicx}
\usepackage{esint} 
\usepackage{color}
\usepackage{comment}
\numberwithin{equation}{section}
\newtheorem{theorem}{Theorem}[section]
\newtheorem{lemma}[theorem]{Lemma}
\newtheorem{corollary}[theorem]{Corollary}
\newtheorem{proposition}[theorem]{Proposition}

\theoremstyle{definition}

\newtheorem{remark}[theorem]{Remark}
\newtheorem{definition}[theorem]{Definition}
\newtheorem{example}[theorem]{Example}

\theoremstyle{remark}

\begin{document}

\title{Weighted local Hardy spaces with variable exponents}

\author{Mitsuo Izuki, Toru Nogayama, Takahiro Noi and Yoshihiro Sawano}

\maketitle

\begin{abstract}
This paper defines local weighted Hardy spaces with variable exponent.
Local Hardy spaces permit
atomic decomposition,
which is one of the main themes in this paper.
A consequence
is that
the atomic decomposition
is obtained
for the functions in the Lebesgue spaces
with exponentially decaying exponent.
As an application, 
we obtain the boundedness of singular integral operators, 
the Littlewood--Paley characterization and wavelet decomposition.
\end{abstract}


\section{Introduction}
Motivated by Bui \cite{Bui82} and Tang \cite{Tang12},
we define weighted Hardy spaces with variable exponents
and obtain some decomposition results
for functions in $L^{p(\cdot)}(w)$
 as an application.
A variable exponent
means a positive measurable function
$p(\cdot):{\mathbb R}^n \to (0,\infty)$.
Here and below
the space
$L^0({\mathbb R}^n)$
denotes
the linear space of all Lebesgue measurable functions
in ${\mathbb R}^n$
and
${\mathbb N}_0\equiv \{0,1,\ldots\}$.

We
begin with the definition of weighted Lebesgue spaces with variable exponents.
Let $w:{\mathbb R}^n \to [0,\infty)$ be a weight.
That is,
$w$ is
a locally integrable function that satisfies
$0<w(x)<\infty$ for almost all $x \in {\mathbb R}^n$.
As in
\cite{CF-book,CMP-book,INS2014,KR},
for a variable exponent 
$p(\cdot) : {\mathbb R}^n \to (0,\infty)$, 
the weighted Lebesgue space
$L^{p(\cdot)}(w)$ 
with a variable exponent 
is defined by
\[
L^{p(\cdot)}(w) \equiv 
\bigcup_{\lambda>0}
\{ f\in L^0({\mathbb R}^n)\, : \, 
 \rho^w_{p(\cdot)}(\lambda^{-1}f) < \infty \} ,
\]
where 
\[
\rho^w_{p(\cdot)}(f) \equiv \|\,|f|^{p(\cdot)}w\,\|_{L^1}.
\]
Moreover,
for $f \in L^{p(\cdot)}(w)$ the 
variable Lebesgue quasi-norm
$\| \cdot \|_{L^{p(\cdot)}(w)}$
is defined
by
\[ \| f \|_{L^{p(\cdot)}(w)}\equiv 
\inf\left( \left\{ \lambda>0 \,:\, \rho^w_{p(\cdot)}(\lambda^{-1}f) \le 1 \right\} \cup \{\infty\}\right).
\]
If $w \equiv 1$, we write $L^{p(\cdot)}(1)=L^{p(\cdot)}({\mathbb R}^n)$ and 
$\| \cdot \|_{L^{p(\cdot)}(1)}=\| \cdot \|_{L^{p(\cdot)}}$.

We write
$\displaystyle w(E)\equiv\int_E w(x) {\rm d}x$ 
and 
$\displaystyle m_E(w)=\dfrac{w(E)}{|E|}$
for a weight $w$ and a measurable set $E$.
We postulate on $w$ and $p(\cdot)$ the following conditions:
As for the weight $w$,
we assume that $w \in A^{\rm loc}_{\infty}$.
The weight $w$ is an $A^{\rm loc}_\infty${\it-weight},
if $0<w<\infty$ almost everywhere,
and
$\displaystyle
[w]_{A^{\rm loc}_\infty}
\equiv 
\sup_{Q \in {\mathcal Q}, |Q| \le 1}
m_Q(w)\exp(-m_Q(\log w))<\infty,
$
where in the sequel
${\mathcal Q}$ stands for the set of all compact cubes
whose edges are parallel to the coordinate axes.
The quantity
$[w]_{A^{\rm loc}_\infty}$ is referred to as
the $A^{\rm loc}_\infty${\it-constant}.
For the variable exponent $p(\cdot)$,
consider two classes:
The class $\mathcal{P}_0$ consists of all
$p(\cdot):\mathbb{R}^n \to (0,\infty)$
such that 
\begin{equation}\label{eq:210712-1}
0<p_-\equiv {\rm essinf}_{x \in {\mathbb R}^n}p(x) \le
p_+\equiv {\rm esssup}_{x \in {\mathbb R}^n}p(x)<\infty,
\end{equation}
while the subclass $\mathcal{P}$ of $\mathcal{P}_0$
collects all exponents $p(\cdot)$ satisfying
$p_->1$.
We consider the local $\log$-H\"{o}lder continuity condition
and 
the $\log$-H\"{o}lder-type decay condition at infinity.
Recall that
$p(\cdot)$ satisfies the local $\log$-H\"{o}lder continuity condition
(denoted by $p(\cdot) \in {\rm L H}_0$)
if
\begin{equation}\label{p-logHolder}
 |p(x)-p(y)|\le\frac{c_*}{\log(|x-y|^{-1})}
 \quad\text{for}\quad
 |x-y|\le\frac12,\ x,y\in{\mathbb R}^n,
\end{equation}
while
the exponent $p(\cdot)$ satisfies 
the $\log$-H\"{o}lder-type decay condition at infinity
(denoted by $p(\cdot) \in {\rm L H}_\infty$)
if
\begin{equation}\label{p-decay}
 |p(x)-p_\infty|\le\frac{c^*}{\log(e+|x|)}
 \quad\text{for}\quad
 x\in{\mathbb R}^n.
\end{equation}
Here $c_*$, $c^*$ and $p_\infty$ are positive constants independent of $x$ and $y$.

Based on the paper by Tang \cite{Tang12}, who followed the idea of Feffereman and Stein
\cite{FeSt72},
we define local weighted Hardy spaces with variable exponents
using grand maximal functions.
To this end,
we recall 
the definition of grand maximal functions by Tang.
Let $L \in {\mathbb N}_0$.
The set ${\mathcal P}_L({\mathbb R}^n)^\perp$
denotes the set of all $f \in L^0({\mathbb R}^n)$
for which
$\langle \cdot \rangle^L f \in L^1({\mathbb R}^n)$
and
$
\int\limits_{{\mathbb R}^n}x^\alpha f(x){\rm d}x=0
$
for all $\alpha \in {\mathbb N}_0^n$ with $|\alpha| \le L$.
By convention, we define
${\mathcal P}_{-1}({\mathbb R}^n)^\perp=L^1({\mathbb R}^n)$.
Such a function $f$
satisfies the 
{\it moment condition of order $L$}.
In this case,
we write $f \perp {\mathcal P}_L({\mathbb R}^n)$.
If $f \perp {\mathcal P}_L({\mathbb R}^n)$
for all $L \in {\mathbb N}_0$,
we write
$f \perp {\mathcal P}({\mathbb R}^n)$.

Let $N \in {\mathbb N}_0$,
which will be specified shortly.
Denote by $B(r)$
the open ball centered at the origin of radius $r>0$.
The set ${\mathcal D}({\mathbb R}^n)$ consists of all infinitely differentiable functions defined on ${\mathbb R}^n$ whose support is compact.
Following \cite[p.457]{Tang12},
we write
\begin{align*}
{\mathcal D}^0_N({\mathbb R}^n)
&\equiv 
\{\varphi \in {\mathcal D}({\mathbb R}^n)
\setminus {\mathcal P}_0({\mathbb R}^n)^\perp\,:\,
|\partial^\alpha\varphi| \le \chi_{B(1)},
|\alpha| \le N+1\},\\
{\mathcal D}_N({\mathbb R}^n)
&\equiv 
\{\varphi \in {\mathcal D}({\mathbb R}^n)
\setminus {\mathcal P}_0({\mathbb R}^n)^\perp\,:\,
|\partial^\alpha\varphi| \le \chi_{B(2^{3n+30})},
|\alpha| \le N+1\}.
\end{align*}
Let $f \in {\mathcal D}'({\mathbb R}^n)$
and $N$ be large enough.
For $\varphi \in {\mathcal D}({\mathbb R}^n)$ and $t>0$,
write $\varphi_t\equiv t^{-n}\varphi(t^{-1}\cdot)$.
Define
three local grand maximal operators by
\begin{align*}
{\mathcal M}^0_N f(x)
&\equiv 
\sup\{|\varphi_t*f(x)|\,:\,0<t<1, \varphi \in {\mathcal D}^0_N({\mathbb R}^n)\},\\
\overline{\mathcal M}^0_N f(x)
&\equiv 
\sup\{|\varphi_t*f(x)|\,:\,0<t<1, \varphi \in {\mathcal D}_N({\mathbb R}^n)\},\\
{\mathcal M}_N f(x)
&\equiv 
\sup\{|\varphi_t*f(z)|\,:\,|z-x|<t<1, \varphi \in {\mathcal D}_N({\mathbb R}^n)\}.
\end{align*}
It is obvious that
$
{\mathcal M}^0_N f
\le
\overline{\mathcal M}^0_N f
\le
{\mathcal M}_N f
$
for any $N \in {\mathbb N}_0$.

We recall 
the notion of the class $A^{\rm loc}_p$ of weights.
Let $1<p<\infty$.
A weight $w$ belongs to $A^{\rm loc}_p$ if 
$$[w]_{A^{\rm loc}_p} \equiv \sup\limits_{Q \in {\mathcal Q}, |Q| \le 1}m_Q(w)m_Q(w^{-\frac{1}{p-1}})^{p-1}<\infty.$$
The quantity
$[w]_{A^{\rm loc}_p}$ is referred to as
the $A^{\rm loc}_p${\it-constant}.
Using the technique employed by Rychkov
\cite{Rychkov01},
we obtain
$A^{\rm loc}_\infty=\bigcup\limits_{q > 1}A^{\rm loc}_{q}$.
We set
\[
q_w\equiv \inf\{p \in [1,\infty)\,:\,w \in A^{\rm loc}_p\}
\]
for $w \in A^{\rm loc}_\infty$.
Similar to Tang \cite[p. 458]{Tang12}, we set
\[
N_{p(\cdot),w}
\equiv 
2+\left[n\left(\frac{q_w}{\min(1,p_-)}-1\right)\right].
\]
Herein we assume
\begin{equation}\label{eq:210712-2}
N \ge N_{p(\cdot),w}.
\end{equation}
Here and below we use the following convention
on the notation $\lesssim$ and $\gtrsim$.
Let $A,B \ge 0$.
Then $A \lesssim B$ and $B \gtrsim A$ mean
that there exists a constant $C>0$
such that $A \le C B$,
where $C$ depends only on the parameters
of importance.
The symbol $A \sim B$ means
that $A \lesssim B$ and $B \lesssim A$
happen simultaneously,
while $A \simeq B$ means that there exists
a constant $C>0$ such that $A=C B$.

The following theorem is the starting point
of this paper.
\begin{theorem}\label{thm:210712-1}
Let
$w \in A^{\rm loc}_{\infty}$.
Also let
$p(\cdot) \in {\mathcal P}_0 \cap {\rm L H}_0 \cap {\rm L H}_\infty$ and $N \in {\mathbb N}$ satisfy
{\rm \eqref{eq:210712-2}}.
Then
\[
\|{\mathcal M}^0_N f\|_{L^{p(\cdot)}(w)}
\le
\|\overline{\mathcal M}^0_N f\|_{L^{p(\cdot)}(w)}
\le
\|{\mathcal M}_N f\|_{L^{p(\cdot)}(w)}
\lesssim 
\|{\mathcal M}^0_N f\|_{L^{p(\cdot)}(w)}
\]
for all $f \in {\mathcal D}'({\mathbb R}^n)$.
\end{theorem}
Note that
Rychkov \cite{Rychkov01} proved Theorem \ref{thm:210712-1}
 when $p(\cdot)$ is a constant exponent.
Based on Theorem \ref{thm:210712-1},
we define weighted Hardy spaces with variable exponents.
\begin{definition}
Let
$w \in A^{\rm loc}_{\infty}$.
Also let
$p(\cdot) \in {\mathcal P}_0 \cap {\rm L H}_0 \cap {\rm L H}_\infty$ and $N \in {\mathbb N}$ satisfy
{\rm \eqref{eq:210712-2}}.
Then
the weighted local Hardy spaces $h^{p(\cdot)}(w)=h^{p(\cdot),N}(w)$
with variable exponents
is the set of all 
$f \in {\mathcal D}'({\mathbb R}^n)$
for which the quasi-norm
\[
\|f\|_{h^{p(\cdot)}(w)}\equiv 
\|{\mathcal M}_N f\|_{L^{p(\cdot)}(w)}
\]
is finite.
\end{definition}

Note that the norm $\|\cdot\|_{h^{p(\cdot)}}$ is independent of $N$
in the sense that different choices of $N$ satisfying
(\ref{eq:210712-2}) yield equivalent norms.

The main purpose of this note is to investigate
equivalent norms of $h^{p(\cdot)}(w)$.
Among others, we are interested in
the characterization by means of atoms and their related norms.

\begin{definition}
Let
$w \in A^{\rm loc}_{\infty}$,
$q\in(0,\infty]$ and $L \in {\mathbb N}_0 \cup \{-1\}$.
Also let
$p(\cdot) \in {\mathcal P}_0 \cap {\rm L H}_0 \cap {\rm L H}_\infty$ and $N \in {\mathbb N}$ satisfy
{\rm \eqref{eq:210712-2}}.
\begin{enumerate}
\item
If
$q>\max(p_+,q_w)$ and 
$L \ge 
\left[n\left(\frac{q_w}{\min(1,p_-)}-1\right)\right]$,
then
a triplet
$(p(\cdot),q,L)$ is called admissible.
\item
Let $Q$ be a cube with $|Q|<1$.
A function $a \in L^q(w)$ is 
a $(p(\cdot),q,L)_w$-atom supported on $Q$
if
$a \in {\mathcal P}_L({\mathbb R}^n)^\perp$,
$a$ is supported on 
$Q$
and satisfies
$\|a\|_{L^q(w)} \le w(Q)^{\frac1q}$.
\item
Let $Q$ be a cube with $|Q|=1$.
A function $a \in L^q(w)$ is 
a $(p(\cdot),q,L)_w$-atom supported on $Q$
if
$a$ is supported on 
$Q$
and satisfies
$\|a\|_{L^q(w)} \le w(Q)^{\frac1q}$.
\item
Assume $w({\mathbb R}^n)<\infty$,
or equivalently, $1 \in L^1(w)$.
Then we say that
a function $a$ is a single $(p(\cdot),q)_w$-atom
if
$\|a\|_{L^q(w)} \le w({\mathbb R}^n)^{\frac1q}$. 
\end{enumerate}
If $w=1$, then subscript $w$ is omitted in
these notions.
\end{definition}
For example,
$\chi_Q$ is a $(p(\cdot),\infty,-1)_w$-atom
for any cube $Q$.

Unlike \cite{Tang12},
we assume that the volume of the cubes for $(p(\cdot),q,L)_w$-atoms 
is less than or equal to $1$.

We remark that a single $(p(\cdot),q)_w$-atom does not have to belong to ${\mathcal P}_L({\mathbb R}^n)^\perp$.
For example, $1$ is a single $(p(\cdot),q)_w$-atom.
\begin{definition}
Let
$w \in A^{\rm loc}_{\infty}$,
$p(\cdot) \in {\mathcal P}_0 \cap {\rm L H}_0 \cap {\rm L H}_\infty$,
$v \in (0,p_-) \cap [0,1]$
and
$q \in (0,\infty]$.
Let
$L \in {\mathbb Z}$
satisfy
$L \ge \left[n\left(\frac{q_w}{\min(1,p_-)}-1\right)\right]$.
Then the weighted atomic local Hardy space
$h^{p(\cdot),q,L;v}(w)$ is defined as the set of all
$f \in {\mathcal D}'({\mathbb R}^n)$
satisfying that
$f=\sum\limits_{j=0}^{\infty} \lambda_j a_j$,
where
$a_0$ is a single $(p(\cdot),q)_w$-atom,
each $a_j, j \in {\mathbb N}$ is 
a $(p(\cdot),q,L)_w$-atom supported on a cube $Q_j$,
$\{\lambda_j\}_{j=1}^{\infty} \subset {\mathbb C}$
and
\[
{\mathcal A}_{p(\cdot),w,v}(\{\lambda_j\}_{j=1}^{\infty};\{Q_j\}_{j=1}^{\infty})
\equiv 
\left\|\left(
\sum_{j=1}^{\infty} |\lambda_j|^{v}\chi_{Q_j}
\right)^{\frac{1}{v}}\right\|_{L^{p(\cdot)}(w)}
<\infty.
\]
The norm
$\|f\|_{h^{p(\cdot),q,L;v}(w)}$ is defined as
the infimum of
$
|\lambda_0|+
{\mathcal A}_{p(\cdot),w,v}(\{\lambda_j\}_{j=1}^{\infty};\{Q_j\}_{j=1}^{\infty})$
over all expressions of $f$
above.
\end{definition}

We will show that
$h^{p(\cdot)}(w)$ can be characterized
via atoms
by proving that,
with an equivalece of norms, the weighted atomic local Hardy space
$h^{p(\cdot),q,L;v}(w)$ coincides with 
$h^{p(\cdot)}(w)$
as long as $q>\max(q_w,p_+)$,
$v\in(0,p_-) \cap[0,1]$ and 
$L \ge \left[n\left(\frac{q_w}{v}-1\right)\right]$.
\begin{theorem}\label{thm:190406-2}
Let
$p(\cdot) \in {\mathcal P}_0 \cap {\rm L H}_0 \cap {\rm L H}_\infty$ 
and $N \in {\mathbb N}$ satisfy
{\rm \eqref{eq:210712-2}}.
Let $v\in(0,p_-) \cap[0,1]$.
Also let
$w \in A^{\rm loc}_{\infty}$
and
$L \in {\mathbb Z}$
satisfy
\begin{equation}\label{eq:211124-101}
N \ge L \ge \left[n\left(\frac{q_w}{v}-1\right)\right].
\end{equation}
Suppose
that a parameter $q$ satisfies
\begin{equation}\label{eq:211124-102}
\max(q_w,p_+)<q \le \infty.
\end{equation}
Then
$h^{p(\cdot),q,L;v}(w)\cong h^{p(\cdot)}(w)$
with an equivalence of norms.
\end{theorem}

Let $p(\cdot)$ be a variable exponent
with $1<p_- \le p_+<\infty$.
A locally integrable weight $w$ is an $A^{\rm loc}_{p(\cdot)}${\it-weight},
if $0<w<\infty$ almost everywhere
and
\begin{equation}\label{eq:210921-111}
[w]_{A^{\rm loc}_{p(\cdot)}}
\equiv 
\sup_{Q \in {\mathcal Q}, |Q| \le 1}
\frac{1}{|Q|}
\|\chi_Q\|_{L^{p(\cdot)}(w)}
\|\chi_Q\|_{L^{p'(\cdot)}(\sigma)}<\infty,
\end{equation}
where $\sigma\equiv w^{-\frac{1}{p(\cdot)-1}}$
and $p'(\cdot)$ is the dual exponent given by
$p'(\cdot)=\frac{p(\cdot)}{p(\cdot)-1}$.
If $w \in A^{\rm loc}_{p(\cdot)}$,
then weighted local Hardy spaces with
a variable exponent
and weighted Lebesgue spaces with variable exponent
coincide as given in the following theorem.
\begin{theorem}\label{thm:210712-2}
Let
$w \in A^{\rm loc}_{p(\cdot)}$.
Also let
$p(\cdot) \in {\mathcal P} \cap {\rm L H}_0 \cap {\rm L H}_\infty$ and $N \in {\mathbb N}$ satisfy
{\rm \eqref{eq:210712-2}}.
Then
$h^{p(\cdot)}(w)=L^{p(\cdot)}(w)$
with an equivalence of norms.
\end{theorem}
It may be interesting to compare Theorem \ref{thm:210712-2}
with \cite{Ho17}.

Note that Theorem \ref{thm:210712-2}
is proved
by Rychkov \cite{Rychkov01} when $p(\cdot)$ is a constant exponent.

Thanks to
Theorems 
\ref{thm:190406-2}
and
\ref{thm:210712-2},
the following decomposition results
on $L^{p(\cdot)}(w)$
is given.
Here, the moment condition is not necessary.
\begin{theorem}\label{cor:211009-1}
Let
$f \in L^0({\mathbb R}^n)$,
$w \in A^{\rm loc}_{p(\cdot)}$
and
and $L \in {\mathbb N}_0 \cup \{-1\}$.
Also let
$p(\cdot) \in {\mathcal P} \cap {\rm L H}_0 \cap {\rm L H}_\infty$.
Assume that
the parameters $q$ and $q_0$ satisfy
$q,q_0>p_+$ and $\sigma=w^{-\frac{1}{p(\cdot)-1}} \in A_{p'(\cdot)/q_0'}$.
\begin{enumerate}
\item
Suppose
$w \in L^1({\mathbb R}^n)$.
Then the following are equivalent:
\begin{enumerate}
\item[$({\rm I})$]
$f \in L^{p(\cdot)}(w)$.
\item[$({\rm II})$]
There exist
a single $(p(\cdot),q)_w$-atom $a_0$
and
a collection $\{a_j\}_{j=1}^{\infty} \subset L^0({\mathbb R}^n)$,
where
each $a_j, j \in {\mathbb N}$ is 
a $(p(\cdot),q,L)_w$-atom supported on a cube $Q_j$
with $|Q_j| \le 1$,
and a collection
$\{\lambda_j\}_{j=0}^{\infty}$
of complex constants such that
$
f=\sum\limits_{j=0}^{\infty} \lambda_j a_j
$
in $L^{p(\cdot)}(w)$
and that
\[
|\lambda_0|+
{\mathcal A}_{p(\cdot),w,1}(\{\lambda_j\}_{j=1}^{\infty};\{Q_j\}_{j=1}^{\infty})=
|\lambda_0|+
\left\|\sum_{j=1}^{\infty} |\lambda_j|\chi_{Q_j}
\right\|_{L^{p(\cdot)}(w)}
<\infty.
\]
\item[$({\rm III})$]
There exist
a single $(p(\cdot),q_0)$-atom $a_0$
and
a collection $\{a_j\}_{j=1}^{\infty} \subset L^0({\mathbb R}^n)$,
where
each $a_j, j \in {\mathbb N}$ is 
a $(p(\cdot),q_0,L)$-atom supported on a cube $Q_j$
with $|Q_j| \le 1$,
and a collection
$\{\lambda_j\}_{j=0}^{\infty}$
of complex constants such that
$
f=\sum\limits_{j=0}^{\infty} \lambda_j a_j
$
in $L^{p(\cdot)}(w)$
and that
\[
|\lambda_0|+
{\mathcal A}_{p(\cdot),w,1}(\{\lambda_j\}_{j=1}^{\infty};\{Q_j\}_{j=1}^{\infty})=
|\lambda_0|+
\left\|\sum_{j=1}^{\infty} |\lambda_j|\chi_{Q_j}
\right\|_{L^{p(\cdot)}(w)}
<\infty.
\]
\end{enumerate}
In this case,
\[
\|f\|_{L^{p(\cdot)}(w)}
\sim
\inf\left\{
|\lambda_0|+
{\mathcal A}_{p(\cdot),w,1}(\{\lambda_j\}_{j=1}^{\infty};\{Q_j\}_{j=1}^{\infty})
\right\},
\]
where each
$a_j$ and $\lambda_j$ 
move over
all elements in
$L^0({\mathbb R}^n)$
and
${\mathbb C}$
to
satisfy the conditions of $({\rm II})$ or $({\rm III})$.
\item
Suppose
$w \notin L^1({\mathbb R}^n)$.
Then the following are equivalent{\rm:}
\begin{enumerate}
\item[$({\rm I})$] 
$f \in L^{p(\cdot)}(w)$.
\item[$({\rm II})$]
There exist
a collection $\{a_j\}_{j=1}^{\infty} \subset L^0({\mathbb R}^n)$,
where
each $a_j, j \in {\mathbb N}$ is 
a $(p(\cdot),q,L)_w$-atom supported on a cube $Q_j$
with $|Q_j| \le 1$,
and a collection
$\{\lambda_j\}_{j=1}^{\infty}$
of complex constants such that
$
f=\sum\limits_{j=1}^{\infty} \lambda_j a_j
$
in $L^{p(\cdot)}(w)$
and that
\[
{\mathcal A}_{p(\cdot),w,1}(\{\lambda_j\}_{j=1}^{\infty};\{Q_j\}_{j=1}^{\infty})=
\left\|\sum_{j=1}^{\infty} |\lambda_j|\chi_{Q_j}
\right\|_{L^{p(\cdot)}(w)}
<\infty.
\]
\item[$({\rm III})$]
There exist
a collection $\{a_j\}_{j=1}^{\infty} \subset L^0({\mathbb R}^n)$,
where
each $a_j, j \in {\mathbb N}$ is 
a $(p(\cdot),q_0,L)$-atom supported on a cube $Q_j$
with $|Q_j| \le 1$,
and a collection
$\{\lambda_j\}_{j=1}^{\infty}$
of complex constants such that
$
f=\sum\limits_{j=1}^{\infty} \lambda_j a_j
$
in $L^{p(\cdot)}(w)$
and that
\[
{\mathcal A}_{p(\cdot),w,1}(\{\lambda_j\}_{j=1}^{\infty};\{Q_j\}_{j=1}^{\infty})=
\left\|\sum_{j=1}^{\infty} |\lambda_j|\chi_{Q_j}
\right\|_{L^{p(\cdot)}(w)}
<\infty.
\]
\end{enumerate}
In this case,
\[
\|f\|_{L^{p(\cdot)}(w)}
\sim
\inf
{\mathcal A}_{p(\cdot),w,1}
(\{\lambda_j\}_{j=1}^{\infty};\{Q_j\}_{j=1}^{\infty}),
\]
where each
$a_j$ and $\lambda_j$ 
move over
all elements in
$L^0({\mathbb R}^n)$
and
${\mathbb C}$ to
satisfy the conditions of $({\rm II})$ or $({\rm III})$.
\end{enumerate}
\end{theorem}
We verify that $q$ satisfying
all the requirements in 
Theorem \ref{cor:211009-1} exists
(see
Lemma \ref{lem:210921-112}).

The implication $({\rm I})\Longrightarrow({\rm II})/({\rm III})$ is included in
Theorems 
\ref{thm:190406-2}
and
\ref{thm:210712-2}.
We prove the implication $({\rm II})/({\rm III})\Longrightarrow({\rm I})$.

Note that Theorem \ref{cor:211009-1}
overlaps with the previous studies on
weighted Lebesgue and Hardy spaces 
($w \in A_p, p(\cdot)$ is constant)
\cite[Chapter IIIV]{StTo-text-89},
Lebesgue spaces with variable exponents
$(w=1)$,
\cite{NaSa12,Sawano13},
Musielak--Orlicz spaces with general Young functions
($\Phi(x,t)=t^{p(x)}w(x)$,
$w \in A_\infty$)
\cite[Theorem 3.7]{INS21}
and
mixed Lebesgue spaces
\cite[Theorem 3]{NOSS21}.

Here,
we briefly address
the development of Hardy spaces with variable exponents
along with some related results.
We also compare these works
with the results obtained in this paper.
Meyer
\cite{Meyer-book} established several equivalent
wavelet characterization of $H^1({\mathbb R}^n)$.
Liu
\cite{Liu92} developed
an equivalent wavelet
characterization of the weak Hardy space 
$H^{1,\infty}({\mathbb R}^n)$.
Wu
\cite[Theorem 3.2]{Wu92} 
reported
the wavelet characterization of the weighted Hardy space 
$H^p({\mathbb R}^n)$ for
any $p \in (0,1]$.
Later Garc\'{i}a-Cuerva and Martell 
\cite{GaMa01}
characterized
$H^p({\mathbb R}^n)$
for any 
$p \in (0, 1]$ in terms of wavelets without compact supports using
the vector-valued Calder\'{o}n--Zygmund theory.
See \cite{LLL09} for the wavelet characterization of dual spaces.
Our results are new in the sense that we do not assume
$p_+ \le 1$.
For example, in \cite{YaYa12}, D. Yang and S. Yang assumed that
$\varphi$ is of uniformly upper type $1$,
which corresponds to the assumption $p_+ \le 1$.
Note that the normed space of sequences
used in this paper differs from the ones
in
\cite{Tang12,YaYa12}.
In fact,
as we mentioned,
in
\cite{Tang12,YaYa12}
the authors assumed a condition corresponding to $p_+ \le 1$.
In this case,
arguing similarly to \cite[Theorem 3.12]{INS21},
we learn that the sequence norm
$
{\mathcal A}_{p(\cdot),w,v}(
\{\lambda_j\}_{j=1}^\infty,\{Q_j\}_{j=1}^\infty)
$
can be replaced by
\[
{\mathcal A}_{p(\cdot),w}^\dagger(
\{\lambda_j\}_{j=1}^\infty,\{Q_j\}_{j=1}^\infty)\equiv
\inf\left\{
\lambda>0\,:\,
\sum_{j=1}^\infty
\int_{Q_j}
\left(\frac{|\lambda_j|}{\lambda}\chi_{Q_j}(x)
\right)^{p(x)}w(x){\rm d}x \le 1
\right\}.
\]
Several approaches define Hardy spaces based
on general Banach lattices or characterize them
in terms of wavelets.
However, most  require
that the underlying Banach
lattices
be rearrangement invariant
or the Hardy--Littlewood maximal operator
be bounded there.
For example,
see \cite{Sawano-text-2018,HSYY17,Soardi97,WYYZ19,WYY20} for example.
We remark that our results do not fall under the scope
of \cite{WYY20} since F. Wang, D. Yang and S. Yang assumed
the boundedness property of the Hardy--Littlewood maximal operator.
Nevertheless,
the present paper is based
on the idea of \cite{Tang12,WYY20}.
Since the variable exponents and weights distort
the function spaces strongly,
we can not hope for such a situation.
Therefore, an alternative approach using
the local Hardy--Littlewood maximal operator,
given by (\ref{eq:local maximal operator}),
is necessary.

The rest of this paper is organized as follows.
Section \ref{s52}
collects preliminary facts.
Section \ref{s53}
discusses
the
fundamental properties of function spaces.
We prove 
Theorems \ref{thm:210712-1}
and \ref{thm:190406-2}
in
Sections \ref{s54} and \ref{s55},
respectively.
Section \ref{s60}
is oriented to the applications of
Theorems \ref{thm:210712-1}
and \ref{thm:190406-2}.
We mainly discuss the boundedness property of singular integral operators.
Section \ref{s60} has some commonality with \cite{Ho19}.
Section \ref{s77}
is devoted to the Littlewood--Paley characterization
of $h^{p(\cdot)}(w)$,
which is a further application of the results in 
Section \ref{s60}.
Further examples and the relations to other function spaces are provided in Section \ref{s9}. 
Section \ref{s80}
considers wavelet characterization.
In Section \ref{s88},
we compare the definition of weighted Lebesgue spaces
with variable exponents.
There are many attempts to extend the classical 
Muckenhoupt class to the setting of variable exponents
inspired by
the works
\cite{CFN2003,CFN2004,DieningMIA2004}.
For example,
see
\cite{CDH-2011,CF-book,CFN-2012,CrWa17,DiHapre}.
Here we consider the local counterpart of the work
\cite{DiHapre} and compare it with the results
in \cite{NS-local}.
We remark that
\cite{DiHapre} is a preprint.
So, we gave details
for the facts related \cite{DiHapre}.
However, our results related to \cite{DiHapre}
 are essentially
minor modifications of
\cite{DiHapre}.

\section{Preliminaries}
\label{s52}

Many tools are necessary
to establish our results.
First, we recall 
the notion of generalized dyadic grids
in Section \ref{subsection:Dyadic grids}.
Section \ref{subsection:Weighted variable Lebesgue spaces}
collects
norm inequalities.
We establish
some boundedness properties
of the weighted maximal operator
in Section
\ref{subsection:Maximal inequalities}.
Section \ref{s2.20} refines the openness property
obtained by Hyt\"{o}nen and P\'{e}rez \cite{HyPe13}.
Section
\ref{subsection:operator}
is oriented to the boundedness of operators
including their vector-valued boundedness.
To establish the theory of
the atomic decomposition,
we depend on the boundedness
properties of some operators,
which are adapted to our class of weights.
Thus, we will carefully collect the results
on the boundedness of operators.
To develop the atomic decomposition theory,
we consider the power
of the normed spaces.
This is necessary
since $p_- \le 1$.
In Section \ref{subsection:Powered local weighted maximal operator}
we transform our results obtained
in the previous sections
to consider
the case of $w \in A_\infty^{\rm loc}$.
Finally, keeping in mind
that our characterization
of $h^{p(\cdot)}(w)$
includes the one by the Littlewood--Paley operators,
in Section \ref{subsection:distribution}
we recall some important inequalities
obtained by Rychkov \cite{Rychkov01}.

\subsection{Generalized dyadic grids}
\label{subsection:Dyadic grids}
Let
\[
{\mathcal D}_{k,a}^0
\equiv
\{2^{-k}[m+a/3,m+a/3+1)\,:\,m \in {\mathbb Z}\}
\]
for
$k \in {\mathbb Z}$ and $a=0,1,2$.
Consider
\[
{\mathcal D}_{k,{\bf a}}
\equiv
\{Q_1 \times Q_2 \times \cdots \times Q_n\,:\,
Q_j \in {\mathcal D}_{k,a_j}^0, j=1,2,\ldots,n\}
\]
for
$k \in {\mathbb Z}$ and ${\bf a}=(a_1,a_2,\ldots,a_n) \in \{0,1,2\}^n$.
Herein,
a (generalized) dyadic grid is
in the family
$\mathcal{D}_{\bf a} \equiv \bigcup\limits_{k \in {\mathbb Z}}{\mathcal D}_{k,{\bf a}}$
for ${\bf a} \in \{0,1,2\}^n$.
It is noteworthy that for any cube $Q$
there exists
$R\in\bigcup\limits_{{\bf a} \in \{0,1,2\}^n}{\mathcal D}_{\bf a}$
such that $Q \subset R$ and that $|R| \le 6^n|Q|$.
Proving the boundedness property of the Hardy--Littlewood
maximal operator
or the local maximal operator
$M^{\rm loc}$
defined by (\ref{eq:local maximal operator}) below,
this property allows us to handle
the operator $M^{{\mathcal D}_{\bf a}}$ generated
by ${\mathcal D}_{\bf a}$
instead of these maximal operators.
See (\ref{eq:D-maximal operator}) and (\ref{eq:220205-71}), below.
Recall that
we can handle ${\mathcal D}_{\bf a}$
similarly
for other values of ${\bf a} \in \{0,1,2\}^n$
so that in particular,
we consider the dyadic grid ${\mathfrak D}={\mathcal D}_{(1,1,\ldots,1)}$.
Denote by
${\mathfrak D}_k=
{\mathcal D}_{k,(1,1,\ldots,1)}$
the set of all cubes in ${\mathfrak D}$ with
$\ell(Q)=2^{-k}$.
Here, given a cube $Q$,
it is denoted
by $\ell(Q)$, which is the sidelength of $Q$:
$\ell(Q)\equiv|Q|^{1/n}$,
where $|Q|$ denotes the volume of cube $Q$.
Two cubes $Q_1,Q_2$ in ${\mathfrak D}$
may intersect at a point but that the difference set
$Q_1 \ominus Q_2=(Q_1 \setminus Q_2) \cup (Q_2 \setminus Q_1)$
is not empty.

For
$f \in L^0({\mathbb R}^n)$,
we define the local (Hardy--Littlewood) maximal operator $M^{\rm loc}$
by 
\begin{equation}\label{eq:local maximal operator}
M^{\rm loc} f(x)\equiv 
\sup_{Q \in {\mathcal Q}, |Q| \le 1}\frac{\chi_Q(x)}{|Q|}\int_Q |f(y)|{\rm d}y
\quad (x \in {\mathbb R}^n).
\end{equation}
Note that this definition is analogous to 
the Hardy--Littlewood maximal operator $M$ defined 
by 
\begin{equation}\label{eq:maximal operator}
M f(x)\equiv 
\sup_{Q \in {\mathcal Q}}\frac{\chi_Q(x)}{|Q|}\int_Q |f(y)|{\rm d}y
\quad (x \in {\mathbb R}^n).
\end{equation}
Let $M^{{\mathcal D}_{\bf a}}$,
${\bf a} \in \{0,1,2\}^n$,
be the maximal operator generated by  grid
${\mathcal D}_{\bf a}$ given by
\begin{equation}\label{eq:D-maximal operator}
M^{{\mathcal D}_{\bf a}}f(x)=\sup_{Q \in {{\mathcal D}_{\bf a}}}
\frac{\chi_Q(x)}{|Q|}\int_Q|f(y)|{\rm d}y
\quad (x \in {\mathbb R}^n).
\end{equation}
Using the above property of the grid ${\mathcal D}_{\bf a}$,
${\bf a} \in \{0,1,2\}^n$,
we prove
\begin{equation}\label{eq:220205-71}
M f(x)\le 6^n\sum_{{\bf a} \in \{0,1,2\}^n}M^{{\mathcal D}_{\bf a}}f(x)
\end{equation}
for $f \in L^0({\mathbb R}^n)$.
Once we prove the boundedness property
of $M^{{\mathcal D}_{\bf a}}$,
${\bf a} \in \{0,1,2\}^n$ on $L^{p(\cdot)}(w)$,
(\ref{eq:220205-71}) yields the one of $M$.
In \cite[\S 4]{NS-local},
we also establish that
the boundedness property
of $M^{{\mathcal D}_{\bf a}}$,
${\bf a} \in \{0,1,2\}^n$ on $L^{p(\cdot)}(w)$
yields the one of
$M^{\rm loc}$.
\subsection{Weighted variable Lebesgue spaces}
\label{subsection:Weighted variable Lebesgue spaces}

For any measurable subset $\Omega \subset \mathbb{R}^n$, denote
\[
p_+(\Omega)
\equiv
{\rm esssup}_{x \in \Omega} p(x),
\quad
p_-(\Omega)
\equiv
{\rm essinf}_{x \in \Omega} p(x).
\]

Let
$p(\cdot)$ satisfy $1\le p(\cdot) \le \infty.$
If $p(\cdot) \in {\rm LH}_{0}$ then $p'(\cdot) \in {\rm LH}_{0}$.
Likewise, if $p(\cdot) \in {\rm LH}_{\infty}$ then $p'(\cdot) \in {\rm LH}_{\infty}$.
Furthermore, $(p_\infty)'=(p')_\infty$. 

Recall 
the generalized H\"{o}lder inequality.
\begin{lemma}[Generalized H\"{o}lder inequality]\label{thm:gHolder}
Let $p(\cdot) :{\mathbb R}^n \to [1,\infty]$ be a variable exponent.
Then for all $f\in L^{p(\cdot)}({\mathbb R}^n)$ 
and $g\in L^{p'(\cdot)}({\mathbb R}^n)$, 
\begin{equation}\label{gHolder}
 \|f \cdot g\|_{L^1} 
 \le 
 r_p 
 \|f\|_{L^{p(\cdot)}} 
 \|g\|_{L^{p'(\cdot)}},
\end{equation}
where
\begin{equation}\label{rp}
 r_p\equiv 1+\frac{1}{p_-} -\frac{1}{p_+}=
\frac{1}{p_-}+\frac{1}{(p')_-} \le 2.
\end{equation}
\end{lemma}
Now
let us recall some properties for the variable Lebesgue space 
$L^{p(\cdot)}({\mathbb R}^n)$.
The first one concerns the norm growth.
\begin{lemma} {\rm \cite[Lemma 2.1]{DiHapre}, \cite[Lemma 2.2]{NaSa12}} \label{lem:190613-3}
Let $p(\cdot) \in {\mathcal P} \cap {\rm LH}_0 \cap {\rm LH}_\infty$.
\begin{enumerate}
\item[$(1)$]
For all cubes $Q$ with $|Q| \le 1$, we have 
$
|Q|^{{1}/{p_{-}(Q)}} \lesssim |Q|^{{1}/{p_{+}(Q)}}.
$
In particular, we have
$
|Q|^{{1}/{p_{-}(Q)}} \sim |Q|^{{1}/{p_{+}(Q)}} \sim |Q|^{{1}/{p(z)}} \sim \|\chi_Q\|_{L^{p(\cdot)}}.
$
\item[$(2)$]
For all cubes $Q$ with $|Q| \ge 1$, we have 
$
\|\chi_Q\|_{L^{p(\cdot)}} \sim |Q|^{{1}/{p_\infty}}.
$
\end{enumerate} 
\end{lemma}

Next, consider the modular inequality.
\begin{lemma} {\rm \cite[Lemma 2.2]{CFN-2012}, \cite[Lemma 2.17]{NoSa12}} \label{lem:190410-1}
Let $p(\cdot)\,:\,\mathbb{R}^n \to [1,\infty)$ be a variable exponent such that $p_+<\infty$.
Then given any
measurable set $\Omega$ and any $f \in L^0({\mathbb R}^n)$,
we have the following:
\begin{enumerate}
\item[$(1)$]
If $\|\chi_{\Omega}f\|_{L^{p(\cdot)}} \le 1$, then
$\|\chi_{\Omega}f\|_{L^{p(\cdot)}}^{p_+(\Omega)}
\le
\int\limits_{\Omega} |f(x)|^{p(x)} {\rm d}x
\le
\|\chi_{\Omega}f\|_{L^{p(\cdot)}}^{p_-(\Omega)}
$.
\item[$(2)$]
If $\|\chi_{\Omega}f\|_{L^{p(\cdot)}} \ge 1$, then
$\|\chi_{\Omega}f\|_{L^{p(\cdot)}}^{p_-(\Omega)}
\le
\int\limits_{\Omega} |f(x)|^{p(x)} {\rm d}x
\le
\|\chi_{\Omega}f\|_{L^{p(\cdot)}}^{p_+(\Omega)}$.
\end{enumerate}
For a variable exponent $p(\cdot) :\mathbb{R}^n \to [1, \infty)$ and $f \in L^0({\mathbb R}^n)$,
$\|f\|_{L^{p(\cdot)}(w)} \le 1$ if and only if
$\int\limits_{\mathbb{R}^n} |f(x)|^{p(x)}w(x) {\rm d}x \le 1$.
\end{lemma}

We apply Lemma \ref{lem:190410-1}
to compare $w(Q)$ and $\|\chi_Q\|_{L^{p(\cdot)}(w)}$.
\begin{remark} \label{rem:200807-1}
Let $Q$ be a cube.
In Lemma \ref{lem:190410-1},
let $f=w^{\frac{1}{p(\cdot)}}\chi_Q$
to obtain the following equivalence:
\begin{equation}\label{eq:191116-1}
\|\chi_Q\|_{L^{p(\cdot)}(w)} \le 1 
\Longleftrightarrow 
w(Q) \le 1.
\end{equation}
A direct consequence of 
(\ref{eq:191116-1}) is the following:
\begin{enumerate}
\item[$(1)$]
If
$\|\chi_Q\|_{L^{p(\cdot)}(w)}\le1$,
then
\begin{equation}\label{eq:210719-1}
\|\chi_Q\|_{L^{p(\cdot)}(w)}^{p_+(Q)}
\le
w(Q)
\le
\|\chi_Q\|_{L^{p(\cdot)}(w)}^{p_-(Q)}.
\end{equation}
\item[$(2)$]
If
$\|\chi_Q\|_{L^{p(\cdot)}(w)}\ge1$,
then
\begin{equation}\label{eq:210719-2}
\|\chi_Q\|_{L^{p(\cdot)}(w)}^{p_-(Q)}
\le
w(Q)
\le
\|\chi_Q\|_{L^{p(\cdot)}(w)}^{p_+(Q)}.
\end{equation}
\end{enumerate}
\end{remark}

Finally,
recall 
the localization principle due to H\"asto.
We state it in a form
we use in the present paper.
\begin{lemma}{\rm \cite{Hasto2009}} \label{lem:190627-11}
Let $p(\cdot) \in {\mathcal P} \cap {\rm LH}_0 \cap {\rm LH}_\infty$
and $k_0 \in {\mathbb Z}$.
Then
\[
\|f\|_{L^{p(\cdot)}}
\sim
\left(
\sum_{Q \in {\mathfrak D}_{k_0}}
(\|\chi_Q f\|_{L^{p(\cdot)}})^{p_\infty}
\right)^{\frac1{p_\infty}}
\]
for all $f \in L^0({\mathbb R}^n)$.
\end{lemma}

\subsection{Maximal inequalities}
\label{subsection:Maximal inequalities}
For
$f \in L^0({\mathbb R}^n)$,
we recall the local (Hardy--Littlewood) maximal operator $M^{\rm loc}$
by 
\[
M^{\rm loc} f(x)\equiv 
\sup_{Q \in {\mathcal Q}, |Q| \le 1}\frac{\chi_Q(x)}{|Q|}\int_Q |f(y)|{\rm d}y
\quad (x \in {\mathbb R}^n).
\]
We can replace $M^{\rm loc}$ by
\[
M^{{\rm loc},R} f(x)\equiv 
\sup_{Q \in {\mathcal Q}, |Q| \le R^n}
\frac{\chi_Q(x)}{|Q|}\int_Q |f(y)|{\rm d}y
\quad (x \in {\mathbb R}^n).
\]
Here $R>0$.
Although 
$M^{\rm loc}$
is typically defined by (\ref{eq:local maximal operator}),
sometimes replace $1$ by $R$ as above.
This substitution
avoids the self-composition of
$M^{\rm loc}$ defined by (\ref{eq:local maximal operator}). 
We remark that 
the definition of the $A_{p(\cdot)}^{\rm loc}$-norm 
is essentially independent of the cube size restriction. 
That is, given an exponent $p(\cdot)\,:\,\mathbb{R}^n \to (1,\infty)$
with $p_->1$,
a positive number $R >0$ and a weight $w$, 
we say that $w \in A^{{\rm loc},R}_{p(\cdot)}$ if 
$
[w]_{A^{{\rm loc},R}_{p(\cdot)}} 
\equiv 
\sup\limits_{|Q|\le R^n} 
|Q|^{-1} \|\chi_Q\|_{L^{p(\cdot)}(w)}\|\chi_Q\|_{L^{p'(\cdot)}(\sigma)}<\infty,
$
where $\sigma \equiv w^{-\frac{1}{p(\cdot)-1}}$ as before
and the supremum is taken over all cubes $Q \in \mathcal{Q}$
with 
volumes less than or equal to $R^n$.
Then, $w \in A^{{\rm loc},R}_{p(\cdot)}$ 
if and only if $w \in A^{{\rm loc}}_{p(\cdot)}.$
For more detail, see \cite[Section 3.3]{NS-local}.
Furthermore, we have
\begin{equation}\label{eq:211127-1}
\|M^{{\rm loc},R}f\|_{L^{p(\cdot)}(w)}
\lesssim
\|f\|_{L^{p(\cdot)}(w)}
\end{equation}
for all $f \in L^{p(\cdot)}(w)$
whenever $R>0$ and $w \in A^{{\rm loc},1}_{p(\cdot)}=A^{{\rm loc}}_{p(\cdot)}$.
For this reason, we subsume the parameter $R>0$
like this.

The next lemma is analogous of \cite[Theorem 1.5]{CFN-2012}.
In our earlier work \cite{NS-local},
we established that
the class $A_{p(\cdot)}^{\rm loc}$
is suitable for
this maximal operator.
\begin{lemma}{\rm \cite[Theorem 1.2]{NS-local}}\label{lem:190409-100}
Let $p(\cdot) \in {\mathcal P} \cap {\rm LH}_0 \cap {\rm LH}_\infty$.
Also let $w$ be a weight.
Then there exists a constant $D>0$ such that
\begin{equation}\label{eq:210713-1}
\|M^{\rm loc}f\|_{L^{p(\cdot)}(w)}
\le D
\|f\|_{L^{p(\cdot)}(w)}
\end{equation}
for all $f \in L^{p(\cdot)}(w)$
if and only if
$w \in A^{\rm loc}_{p(\cdot)}$.
\end{lemma}

Define $A_1^{\rm loc}$ as the collection of all weights
for which there exists a constant $C>0$ such that
$M^{\rm loc}w \le C w$.
The infimum of such $C$ is called the $A_1^{\rm loc}$-constant 
and is denoted by $[w]_{A_1^{\rm loc}}$.
Remark that
a similar remark for $A_{p(\cdot)}^{\rm loc}$ for $R \ge 1$ applies
to $A_1^{\rm loc}$.
We use the following local reverse H\"{o}lder property:
\begin{lemma}\label{lem:210921-113}
Let 
$w \in A_1^{\rm loc}$.
If we set
\[
\varepsilon\equiv\frac{1}{2^{n+11}[w]_{A_1^{\rm loc}}}>0,
\]
then
$w^{1+\varepsilon} \in A_1^{\rm loc}$.
\end{lemma}

\begin{proof}
This is a local version of \cite[Theorem 2.3]{HyPe13}.
We omit the further details.
\end{proof}

We collect some corollaries from
$(\ref{eq:210921-111})$
and 
Lemmas \ref{lem:190409-100} and \ref{lem:210921-113}.
As a byproduct,
we learn that $q$ satisfying
all requirements in 
Theorem \ref{cor:211009-1} exists. 
The next assertions are known for the global Muckenhoupt class.
However, the corresponding assertion for local class is missing.
So, we supply the proof.
\begin{lemma}\label{lem:210921-112}
Let
$w$ be a weight
and let
$p(\cdot) \in {\mathcal P} \cap {\rm LH}_0 \cap {\rm LH}_\infty$.
\begin{enumerate}
\item
The following are equivalent:
\begin{itemize}
\item
$w \in A_{p(\cdot)}^{\rm loc}$.
\item
$w^{-\frac{1}{p(\cdot)-1}} \in A_{p'(\cdot)}^{\rm loc}$.
\end{itemize}
\item
Let
$w \in A_{p(\cdot)}^{\rm loc}$.
Then
there exists $\varepsilon>0$ such that
$1+\varepsilon<p_-$ and that
$w \in A_{p(\cdot)/(1+\varepsilon)}^{\rm loc}$.
\end{enumerate}
\end{lemma}

\begin{proof}\
\begin{enumerate}
\item
This is an immediate consequence
of 
$(\ref{eq:210921-111})$.
\item
Let
$f \in L^{p(\cdot)}(w) \setminus \{0\}$.
We employ the Rubio de Francia algorithm.
Then define
\[
F
\equiv
\sum_{k=0}^\infty \frac{(M^{\rm loc})^k f}{2^kD^k}
\]
where
$D>0$ is the constant in (\ref{eq:210713-1}),
$(M^{\rm loc})^k$ denotes the $k$-fold composition if $k \ge 1$ and
$(M^{\rm loc})^0 f\equiv|f|$.
Then
$F \le M^{\rm loc} F \le 2DF$.
Hence, $F \in A_1^{\rm loc}$ and $[F]_{A_1^{\rm loc}} \le 2D$.
Due to Lemma \ref{lem:210921-113},
there exists a constant $\varepsilon\in(0,p_--1)$ which depends on
$D$ and $p(\cdot)$
such that
$M^{\rm loc}(F^{1+\varepsilon})^{\frac{1}{1+\varepsilon}}\le 2M^{\rm loc}F$.
Thus,
\begin{align*}
\|M^{\rm loc}(|f|^{1+\varepsilon})^{\frac{1}{1+\varepsilon}}\|_{L^{p(\cdot)}(w)}
&\le
\|M^{\rm loc}(F^{1+\varepsilon})^{\frac{1}{1+\varepsilon}}\|_{L^{p(\cdot)}(w)}\\
&\le
2
\|M^{\rm loc}F\|_{L^{p(\cdot)}(w)}\\
&\le
4D
\|F\|_{L^{p(\cdot)}(w)}\\
&\le
8D
\|f\|_{L^{p(\cdot)}(w)}.
\end{align*}
Since $f \in L^{p(\cdot)}(w) \setminus \{0\}$ is arbitrary,
it follows
from Lemma \ref{lem:190409-100}
 that
$M^{\rm loc}$ is bounded on $L^{p(\cdot)/(1+\varepsilon)}(w)$.
Hence
$w \in A_{p(\cdot)/(1+\varepsilon)}^{\rm loc}$.
\end{enumerate}
\end{proof}
We use the following monotone property
of the class $A_{p(\cdot)}^{\rm loc}$.
The proof is postponed until Appendix;
see the remark after Corollary \ref{cor:9.4}.
\begin{proposition}\label{prop:211027-111}
Let
$p(\cdot),q(\cdot)
\in {\mathcal P} \cap {\rm LH}_0 \cap {\rm LH}_\infty$.
If $q(\cdot) \ge p(\cdot)$,
then $A_{q(\cdot)}^{\rm loc} \supset A_{p(\cdot)}^{\rm loc}$.
\end{proposition}
A clarifying remark may be in order.
\begin{remark}\label{rem:211031-11}
Let
$p(\cdot),q(\cdot)
\in {\mathcal P} \cap {\rm LH}_0 \cap {\rm LH}_\infty$.
Combining the result by Diening and H\"{a}sto
\cite{DiHapre} and the one by 
Cruz-Uribe, Fiorenza and Neugebauer \cite{CFN-2012}, 
we learn that
$A_{q(\cdot)} \supset A_{p(\cdot)}$
whenever
$q(\cdot) \ge p(\cdot)$.
In this paper, we will follow the idea of \cite{DiHapre}
to prove Proposition \ref{prop:211027-111}.
\end{remark}
We 
move on to the vector-valued inequality, which
is an extension
of \cite{CFMP} to the setting of the $A_\infty^{\rm loc}$-class.
\begin{lemma}\label{lem:190408-100}{\rm \cite[Theorem 1.11]{NS-local}}
Let $p(\cdot) \in {\mathcal P} \cap {\rm LH}_0 \cap {\rm LH}_\infty$.
Let $1<q \le \infty$ and $w \in A^{\rm loc}_{p(\cdot)}$.
Then
\[
\left\|\left(\sum_{j=1}^{\infty}(M^{\rm loc}f_j)^q\right)^{\frac1q}\right\|_{L^{p(\cdot)}(w)}
\lesssim
\left\|\left(\sum_{j=1}^{\infty}|f_j|^q\right)^{\frac1q}\right\|_{L^{p(\cdot)}(w)}
\]
for all $\{f_j\}_{j=1}^\infty \subset L^0({\mathbb R}^n)$.
\end{lemma}

\begin{proposition}{\cite[Proposition 2.11]{INNS-wavelet}}\label{prop:190408-10011}
Let
$1<q_1,q_2<\infty$.
Assume
that
$p(\cdot) \in {\mathcal P} \cap {\rm L H}_0 \cap {\rm L H}_\infty$
and
$w \in A^{\rm loc}_{p(\cdot)}$.
Then
\begin{align}\label{eq:210927-111}
\left\|
\left(
\sum_{j_2=1}^\infty
\left(
\sum_{j_1=1}^\infty
(M^{\rm loc}f_{j_1,j_2})^{q_1}
\right)^{\frac{q_2}{q_1}}
\right)^{\frac{1}{q_2}}
\right\|_{L^{p(\cdot)}(w)}
\lesssim
\left\|
\left(
\sum_{j_2=1}^\infty
\left(
\sum_{j_1=1}^\infty
|f_{j_1,j_2}|^{q_1}
\right)^{\frac{q_2}{q_1}}
\right)^{\frac{1}{q_2}}
\right\|_{L^{p(\cdot)}(w)}
\end{align}
for all
$\{f_{j_1,j_2}\}_{j_1,j_2=1}^\infty \subset L^0({\mathbb R}^n)$.
\end{proposition}

\subsection{Openness property--A variant of Lemma \ref{lem:210921-113}}
\label{s2.20}

Let $R \in {\mathfrak D}$.
We set
\[
[w]_{A_{\infty,R^\times}}^{\mathfrak D}=\sup_{Q \in {\mathfrak D}, |Q \setminus R|>0}
m_{Q \setminus R}(w)
\exp\left(-m_{Q \setminus R}(\log w)\right)
\]
and define
the class $A_{\infty,R^\times}^{\mathfrak D}$ is the set of all weights
such that 
$[w]_{A_{\infty,R^\times}}^{\mathfrak D}<\infty$.
Next, we define maximal operators 
$M^{{\mathfrak D}}_{R^\times}$ and $M^{0,{\mathfrak D}}_{R^\times}$
as follows:
\[
M^{{\mathfrak D}}_{R^\times} f(x)\equiv 
\sup_{S \in {\mathfrak D}, \, |S \setminus R|>0}
\frac{\chi_{S \setminus R}(x)}{|S \setminus R|}
\int_{S \setminus R} |f(y)| {\rm d}y
\quad (x \in {\mathbb R}^n),
\]
and
\[
M^{0,{\mathfrak D}}_{R^\times} f(x)\equiv 
\sup_{S \in {\mathfrak D}}
\chi_{S \setminus R}(x)
\exp(m_{S \setminus R}(-\log|f|))
\quad (x \in {\mathbb R}^n).
\]
Here roughly speaking,
\lq \lq $0$" stands for the maximal operator based
on the $L^{0+}({\mathbb R}^n)$-average.
Following the idea in \cite[Lemma 2.1]{HyPe13}, 
we investigate the maximal operators $M^{{\mathfrak D}}_{R^\times}$ and $M^{0,{\mathfrak D}}_{R^\times}$.
 Note that by the stndard argument for the weak type estimate, we have
\begin{equation} \label{eq:220419-1}
|\{x \in {\mathbb R}^n \setminus R\,:\,
M^{{\mathfrak D}}_{R^\times} f(x)>\lambda\}|
\le\frac1{\lambda}
\|\chi_{\{x \in {\mathbb R}^n \setminus R\,:\,
M^{{\mathfrak D}}_{R^\times} f(x)>\lambda\}}f
\|_{L^1({\mathbb R}^n \setminus R)}.
\end{equation}
By Jensen's inequality, 
the layer cake formula and (\ref{eq:220419-1}),
\[
\|M^{0,{\mathfrak D}}_{R^\times}f\|_{L^p}
\le
\|M^{{\mathfrak D}}_{R^\times}f\|_{L^p}
\le
\frac{p}{p-1}\|f\|_{L^p}
\quad
(1<p<\infty)
\]
for all measurable functions $f$.

Since 
\begin{equation}\label{eq:220131-1}
M^{0,{\mathfrak D}}_{R^\times}[|f|^{u}]
=(M^{0,{\mathfrak D}}_{R^\times}f)^{u}
\end{equation}
for all $u>0$, we have
\[
\|M^{0,{\mathfrak D}}_{R^\times}f\|_{L^1}
\le
\left(\frac{p}{p-1}\right)^{p}\|f\|_{L^1}
\]
Letting $p \to \infty$
gives
\begin{equation}\label{eq:220131-2}
\|M^{0,{\mathfrak D}}_{R^\times}f\|_{L^1}
\le
e\|f\|_{L^1}
\end{equation}
for all measurable functions $f$.

\begin{lemma}
\label{Lemma-Muckenhoupt}
Let $R \in {\mathfrak D}$,
$w\in A_{\infty,R^\times}^{\mathfrak D}$, and 
\begin{equation}\label{eq:220131-1110}
q \equiv 1+\frac{1}{4^{n+6}[w]_{A_{\infty,R^\times}^{\mathfrak D}}}.
\end{equation}
Then for all cubes $Q \in {\mathfrak D}$
satisfying $|Q \setminus R|>0$, 
\begin{equation}\label{eq:160729-1}
m_{Q \setminus R}(w^q)^{\frac1q}
\le 2m_{Q \setminus R}(w).
\end{equation}
\end{lemma}

\begin{proof}
Fix $Q \in {\mathfrak D}$.
We can assume that $w$ is bounded
by approximating $w$ with a function in the form
$$\displaystyle
\sum_{S \in {\mathcal D}_j(Q)}m_{S \setminus R}(w)\chi_{S \setminus R},
$$
where ${\mathcal D}_j(Q)$ denotes the set of all cubes
obtained by bisecting $Q$ $j$ times.
Let
$\varepsilon\equiv q-1$
and
$\displaystyle {\mathcal D}(Q)\equiv\bigcup_{j=0}^{\infty}{\mathcal D}_j(Q)$.
When
denoted by
$\tilde{M}^{{\mathcal D}(Q)}$, the dyadic maximal operator
given by
\[
\tilde{M}^{{\mathcal D}(Q)}f(x):=\sup_{S \in {\mathcal D}(Q), \frac{1}{2n}\log_2\frac{|Q|}{|S|} \in {\mathbb Z}}
\chi_S(x)m_S(|f|)=\sup_{S \in {\mathcal D}(Q) \cap {\mathfrak D}}
\chi_S(x)m_S(|f|).
\]
We use the layer cake formula to obtain
\begin{align*}
\int_{Q\setminus R} \tilde{M}^{{\mathcal D}(Q)}[\chi_{Q\setminus R} w](x)^{\varepsilon}w(x){\rm d}x
&=
\varepsilon
\int_0^\infty
\lambda^{\varepsilon-1}
w({(Q\setminus R)} \cap \{\tilde{M}^{{\mathcal D}(Q)}[\chi_{Q\setminus R} w]>\lambda\}){\rm d}\lambda.
\end{align*}
We suppose
$$
\lambda>m_{Q\setminus R}(w).
$$
Consider the set of all maximal dyadic cubes
$\{U_j\}_{j \in {J(\lambda)}}$ in ${\mathcal D}(Q) \cap {\mathfrak D}$
in the set 
${(Q\setminus R)} \cap \{\tilde{M}^{{\mathcal D}(Q)}[\chi_{Q\setminus R} w]>\lambda\}$
whose average of $\chi_{Q\setminus R} w$ exceeds $\lambda$.
Then due to the maximality of each $U_j$,
the grand parent $\tilde{\tilde{U_j}}$ satisfies
$m_{\tilde{\tilde{U_j}}}(w) \le \lambda$.
Thus,
\begin{align*}
\frac{w({(Q\setminus R)} \cap \{\tilde{M}^{{\mathcal D}(Q)}[\chi_{Q\setminus R} w]>\lambda\})}{4^n \lambda}
&=
\sum_{j \in J(\lambda)}\frac{w(U_j)}{4^n\lambda}\\
&{\color{red}\lesssim}
\sum_{j \in J(\lambda)}|U_j|\\
&=
|{(Q\setminus R)} \cap \{\tilde{M}^{{\mathcal D}(Q)}[\chi_{Q\setminus R} w]>\lambda\}|.
\end{align*}
We also note that
${(Q\setminus R)} \cap \{\tilde{M}^{{\mathcal D}(Q)}[\chi_{Q\setminus R} w]>\lambda\} \subset {(Q\setminus R)}$
for any $\lambda>0$.
Thus,
\begin{align}\label{eq:220417-1}
&\int_{Q\setminus R} 
\tilde{M}^{{\mathcal D}(Q)}w(x)^{\varepsilon}w(x){\rm d}x\\
&\qquad\le
\varepsilon
\int_0^{m_{Q \setminus R}(w)}
\lambda^{\varepsilon-1}
w({(Q\setminus R)} \cap \{\tilde{M}^{{\mathcal D}(Q)}[\chi_{Q\setminus R} w]>\lambda\}){\rm d}\lambda \nonumber\\
&\qquad\quad+
4^n\varepsilon
\int_{m_{Q \setminus R}(w)}^\infty
\lambda^{\varepsilon}
|{(Q\setminus R)} \cap \{\tilde{M}^{{\mathcal D}(Q)}[\chi_{Q\setminus R} w]>\lambda\}|{\rm d}\lambda \nonumber\\
&\qquad\le
|Q\setminus R|\left(m_{Q\setminus R}(w)\right)^{1+\varepsilon}
+
\frac{4^n\varepsilon}{1+\varepsilon}
\int_{Q\setminus R} \tilde{M}^{{\mathcal D}(Q)}[\chi_{Q\setminus R} w](x)^{1+\varepsilon}{\rm d}x.\nonumber
\end{align}
Since
$w \in A_{\infty,R^\times}^{\mathfrak D}$,
we have
\begin{equation}\label{eq:170106-91}
m_S(w)
\le \frac{2^n}{2^n-1}
m_{S\setminus R}(\chi_{{\mathbb R}^n \setminus R}w)
\le 
[w]_{A_{\infty,R^\times}^{\mathfrak D}}
\exp(m_{S\setminus R}(-\log w))
\end{equation}
for all $S \in {\mathcal D}(Q) \cap {\mathfrak D}$ such that $|S \setminus R|>0$.

A geometric observation shows
\begin{equation} \label{eq:220417-2}
\tilde{M}^{{\mathcal D}(Q)}[\chi_{Q\setminus R} w](x) 
\le \frac{2^n}{2^n-1}[w]_{A_{\infty,R^\times}^{\mathfrak D}}M^{0,{\mathfrak D}}_{R^\times}w.
\end{equation}
Inserting (\ref{eq:220417-2}) into (\ref{eq:220417-1}), we obtain
\begin{align*}
\lefteqn{
\int_{Q\setminus R} \tilde{M}^{{\mathcal D}(Q)}w(x)^{\varepsilon}w(x){\rm d}x
}\\
&\le
|Q\setminus R|\left(m_{Q\setminus R}(w)\right)^{1+\varepsilon}
+
\frac{4^n\varepsilon}{1+\varepsilon}
\left(\frac{2^n}{2^n-1}[w]_{A_{\infty,R^\times}^{\mathfrak D}}\right)^{1+\varepsilon}
\int_{Q\setminus R} M^{0,{\mathfrak D}}_{R^\times}[\chi_{Q\setminus R} w](x)^{1+\varepsilon}{\rm d}x.
\end{align*}
By the Lebesgue differentiation theorem
and
(\ref{eq:220131-1}),
we have
\begin{align*}
\lefteqn{
\int_{Q\setminus R} w(x)^{1+\varepsilon}{\rm d}x
}\\
&\le
|Q\setminus R|\left(m_{Q\setminus R}(w)\right)^{1+\varepsilon}
+
\frac{4^n\varepsilon}{1+\varepsilon}
\left(\frac{2^n}{2^n-1}[w]_{A_{\infty,R^\times}^{\mathfrak D}}\right)^{1+\varepsilon}
\int_{Q\setminus R} M^{0,{\mathfrak D}}_{R^\times}[\chi_{Q\setminus R} w^{1+\varepsilon}](x){\rm d}x.
\end{align*}
From (\ref{eq:220131-2})
\begin{equation}\label{eq:220131-1111}
m_{Q\setminus R}^{(1+\varepsilon)}(w)^{1+\varepsilon}
\le
(m_{Q\setminus R}(w))^{1+\varepsilon}
+
\frac{4^n\varepsilon e}{1+\varepsilon}
\left(\frac{2^n}{2^n-1}[w]_{A_{\infty,R^\times}^{\mathfrak D}}\right)^{1+\varepsilon}
m_{Q\setminus R}^{(1+\varepsilon)}(w)^{1+\varepsilon}.
\end{equation}
Arithmetic shows
\[
\left(\frac{2^n}{2^n-1}\right)^{\frac{1}{4^n}}
=
\left(1+\frac{1}{2^n-1}\right)^{\frac{1}{4^n}}
\le
\left(1+\frac{1}{2^n-1}\right)^{\frac{1}{2^n-1}}
\le e.
\]
Since $[w]_{A_{\infty,R^\times}^{\mathfrak D}} \ge 1$
and $1+\varepsilon=q$, we have
\begin{align*}
\frac{4^n\varepsilon}{1+\varepsilon}
\left(\frac{2^n}{2^n-1}[w]_{A_{\infty,R^\times}^{\mathfrak D}}\right)^{1+\varepsilon}
&\le
2\left(\frac{2^n}{2^n-1}\right)^{\frac{1}{4^n}}
\frac{4^n[w]_{A_{\infty,R^\times}^{\mathfrak D}}
([w]_{A_{\infty,R^\times}^{\mathfrak D}})^{\frac{1}{4^{n+6}[w]_{A_{\infty,R^\times}^{\mathfrak D}}}}}{4^{n+6}[w]_{A_{\infty,R^\times}^{\mathfrak D}}+1}\\
&\le
\frac{1}{512}\exp\left(\frac{1}{4^{n+6}e}\right)\\
&\le
\frac{1}{7}.
\end{align*}
It follows that
$\displaystyle
m_{Q\setminus R}(w^{1+\varepsilon})
\le
\left(m_{Q\setminus R}(w)\right)^{1+\varepsilon}
+
\frac{e}{7}
m_{Q\setminus R}(w^{1+\varepsilon}).
$
Since $w$ is assumed to be bounded, $2e \le 7$ and $1+\varepsilon=q$,
if we absorb the second term
of the right-hand side 
of
(\ref{eq:220131-1111})
into the left-hand side,
we obtain
\[
m_{Q \setminus R}(w^q)^{\frac1q}
\le 2m_{Q \setminus R}(w),
\]
which proves (\ref{eq:160729-1}).

\end{proof}

\subsection{The operator $K_B$}
\label{subsection:operator}
For $B>0$,
define a convolution operator
$K_B$ by
\[
K_B f(x)
\equiv
\int_{{\mathbb R}^n}e^{-B|x-y|}f(y) {\rm d}y
\]
for $f \in L^0({\mathbb R}^n)$ as long as the definition makes sense.
We invoke an estimate from \cite[Corollary 2.6]{INNS-wavelet}.
\begin{lemma}\label{lem:210923-1}
Let 
$p(\cdot) \in {\mathcal P} \cap {\rm LH}_0 \cap {\rm LH}_\infty$.
Let $D$ be the constant satisfying
$\eqref{eq:210713-1}$.
If $B>8n+6\log D$ and $w \in A^{\rm loc}_{p(\cdot)}$,
then $K_B$ is bounded on $L^{p(\cdot)}(w)$.
\end{lemma}
Recall that we used a pointwise estimate
in the proof of Lemma \ref{lem:210923-1}.
Thus,
if we examine its proof,
then
we can obtain
a vector-valued inequality
from
Lemmas \ref{lem:190408-100}
and
 \ref{lem:210923-1}.
 Actually,
the following inequality can be proven
using the local maximal operator,
whose proof we omit.
\begin{corollary}\label{cor:210923-3}Let 
$p(\cdot) \in {\mathcal P} \cap {\rm LH}_0 \cap {\rm LH}_\infty$.
Let $w \in A^{\rm loc}_{p(\cdot)}$
and $1<q\le\infty$.
Let $D$ be the constant satisfying
$\eqref{eq:210713-1}$.
If $B>8n+6\log D$,
then
\[
\left\|\left(
\sum_{j=1}^\infty|K_B f_j|^q
\right)^{\frac1q}\right\|_{L^{p(\cdot)}(w)}
\lesssim
\left\|\left(
\sum_{j=1}^\infty|f_j|^q
\right)^{\frac1q}\right\|_{L^{p(\cdot)}(w)}
\]
for all $\{f_j\}_{j=1}^\infty \subset L^{p(\cdot)}(w)$.
\end{corollary}

We transform
Corollary \ref{cor:210923-3}
into a form for later considerations
by writing
\begin{equation}\label{eq:210923-1}
m_{j,A,B}(x)\equiv(1+2^j|x|)^A e^{|x|B}
\quad
(x \in {\mathbb R}^n)
\end{equation}
for $j=0,1,\ldots$
and
$A,B>0$.
\begin{corollary}\label{cor:210923-5}
Let 
$p(\cdot) \in {\mathcal P} \cap {\rm LH}_0 \cap {\rm LH}_\infty$
and $w \in A^{\rm loc}_{p(\cdot)}$.
Let $D$ be the constant satisfying
$\eqref{eq:210713-1}$.
If $A>0$, $B>8n+6\log D$ and $1<r<\infty$,
then
\[
\left\|\left(
\sum_{j=1}^\infty
\left|2^{jn}
\int_{{\mathbb R}^n}
\frac{f_j(\cdot-y)}{m_{j,A,B}(y)}{\rm d}y
\right|^r
\right)^{\frac1r}\right\|_{L^{p(\cdot)}(w)}
\lesssim
\left\|\left(
\sum_{j=1}^\infty|f_j|^r
\right)^{\frac1r}\right\|_{L^{p(\cdot)}(w)}
\]
for all $\{f_j\}_{j=1}^\infty \subset L^{p(\cdot)}(w)$.
\end{corollary}

\begin{proof}
Simply observe that
\[
\int_{{\mathbb R}^n}
\frac{|f_j(x-y)|}{m_{j,A,B}(y)}{\rm d}y
\lesssim
K_B|f_j|(x)+M^{\rm loc}f_j(x)
\]
for all $x \in {\mathbb R}^n$, $j=1,2, \ldots$
and
$A,B>0$ (cf. \cite[Lemma 2.10]{Rychkov01}).
Thus, we are in the position to use
Lemma \ref{lem:190408-100}
and
Corollary \ref{cor:210923-3}.
\end{proof}
\subsection{Powered local weighted maximal operator}
\label{subsection:Powered local weighted maximal operator}

For $0<u<\infty$ and a weight $w$, 
define the powered local weighted maximal operator $M^{(u), \rm loc}_w$ by
\[
M^{(u),\rm loc}_w f(x)
\equiv
\sup_{Q \in {\mathcal Q}, |Q| \le 1}
\left(
\frac{\chi_Q(x)}{w(Q)} \int_Q |f(y)|^u w(y) {\rm d}y
\right)^{\frac1u}
\qquad
(f \in L^0({\mathbb R}^n)).
\]
We write
$M_w^{\rm loc}\equiv M_w^{(1),\rm loc}$.

We work in the Euclidean space with the weighted measure $w{\rm d}x$.
\begin{proposition} \label{prop:210805-1}
Let $p(\cdot) \in {\mathcal P} \cap {\rm LH}_0 \cap {\rm LH}_\infty$.
Let $0<u<p_-$ and
$w \in A_{\infty}^{\rm loc}$.
Then
\[
\left\|M^{(u),\rm loc}_w f\right\|_{L^{p(\cdot)}(w)}
\lesssim
\|f\|_{L^{p(\cdot)}(w)}
\]
for $f \in L^{p(\cdot)}(w)$.
\end{proposition}
Some lemmas and intricate arguments are needed to prove this lemma.
First, we prove Proposition \ref{prop:210805-1}
if the exponent is constant.
Since $0<u<p_-$, we can assume $u=1$ by a scaling argument.
Since
$w \in A_\infty^{\rm loc}$,
 $w$ is a locally
doubling weight. That is,
$w$ satisfies $w(5Q) \lesssim w(2Q) \lesssim w(Q)$
for all cubes $Q$
with $|Q| \le 1$.
In the case where
$p(\cdot)$ is a constant we can use the theory of general Radon measures
in \cite[Section 3]{NTV98}.
In fact, we can replace
$M^{(u),\rm loc}_w$ by the maximal operator given by
\[
\tilde{M}^{(u)}_w f(x)
\equiv
\sup_{Q \in {\mathcal Q}}
\left(
\frac{\chi_Q(x)}{w(5Q)} \int_Q |f(y)|^u w(y) {\rm d}y
\right)^{\frac1u}
\qquad
(f \in L^0({\mathbb R}^n)).
\]Thus, the proof of Proposition \ref{prop:210805-1}
is complete if $p(\cdot)$ is a constant exponent.

Now we consider the case where
$p(\cdot)$ is a variable exponent.
We define
\[
[w]_{A_{p(\cdot)}^{\mathfrak D}}
\equiv
\sup_{Q \in {\mathfrak D}}
\frac{1}{|Q|}\|\chi_Q\|_{L^{p(\cdot)}(w)}
\|\chi_Q\|_{L^{p'(\cdot)}(\sigma)},
\]
where
$\sigma \equiv w^{-\frac{1}{p(\cdot)-1}}$
is the dual weight.
The class
$A_{p(\cdot)}^{\mathfrak D}$
collects all weights $w$ for which
$[w]_{A_{p(\cdot)}^{\mathfrak D}}<\infty$.
Hence,
we have only to deal with $M_w^{\mathfrak D}$ instead of $M_w^{\rm loc}$
assuming that $w \in A_{\infty}^{\mathfrak D}$ instead of $w \in A_{\infty}^{\rm loc}$
by the use of a technique similar to that developed
in \cite{NS-local}.
Here, $M_w^{\mathfrak D}$ stands for
\[
M_w^{\mathfrak D} f(x)
\equiv
\sup_{Q \in {\mathfrak D}}
\frac{\chi_Q(x)}{w(Q)} \int_Q |f(y)| w(y) {\rm d}y
\qquad
(f \in L^0({\mathbb R}^n)).
\]

First, we consider the case
where $f$ is unbounded
to find a pointwise estimate of $M^{\mathfrak D}_{w}f$.
\begin{lemma}\label{lem:210807-11}
Let $p(\cdot) \in {\mathcal P}_0$,
$w \in A_{\infty}^{\mathfrak D}$
and $Q \in {\mathfrak D}$.
Let
$f \in L^{p(\cdot)}(w)$
be
a non-negative real-valued function
with 
$\|f\|_{L^{p(\cdot)}(w)} \le 1$.
Assume
$f \le f^2$.
Then
\[
\left(\frac{1}{w(Q)}\int_Q f(y)w(y){\rm d}y\right)^{p(x)}
\lesssim 
\left(
\frac{1}{w(Q)}\int_Q f(y)^{\frac{p(y)}{p_-}}w(y){\rm d}y
\right)^{p_-}
\]
for all $x \in Q$.
\end{lemma}

\begin{proof}
Fix $x \in Q$.
If
\[k
\equiv
\frac{1}{w(Q)}\int_Q f(y)^{\frac{p(y)}{p_-}}w(y){\rm d}y
\le 1,
\]
then
the desired estimate is clear since
\[
\frac{1}{w(Q)}\int_Q f(y)w(y){\rm d}y
\le
\frac{1}{w(Q)}\int_Q f(y)^{\frac{p(y)}{p_-}}w(y){\rm d}y
\le
\left(
\frac{1}{w(Q)}\int_Q f(y)^{\frac{p(y)}{p_-}}w(y){\rm d}y
\right)^{\frac{p_-}{p(x)}}.
\]
Otherwise,
assume
$k\ge 1.$
Then $f$ can be decomposed
according to $\{y \in {\mathbb R}^n\,:\,
f(y) \ge k^{\frac{p_-}{p(x)}}\}$
to give
\begin{align}
\frac{1}{w(Q)}\int_Q f(y)w(y){\rm d}y
&\le
k^{\frac{p_-}{p(x)}}
+
\frac{1}{w(Q)}\int_Q f(y)\chi_{[k,\infty]}(f(y)^{\frac{p(x)}{p_-}})w(y){\rm d}y
\nonumber\\
\label{eq:211101-52}
&\le
k^{\frac{p_-}{p(x)}}
+
\frac{1}{w(Q)}\int_Q f(y)^{\frac{p(y)}{p_-}}k^{\left(-\frac{p(y)}{p_-}+1\right)\frac{p_-}{p(x)}}w(y){\rm d}y.
\end{align}
Since
$w \in A_\infty^{\mathfrak D}$ and
\[
1 \le k=\frac{1}{w(Q)}\int_Q f(y)^{\frac{p(y)}{p_-}}w(y){\rm d}y \le \frac{1}{w(Q)},
\]
we have
\begin{equation}\label{eq:211101-51}
k^{1-\frac{p(y)}{p(x)}}
\sim
\left(
\frac{1}{w(Q)}
\right)^{1-\frac{p(y)}{p(x)}} \sim 1
\end{equation}
thanks to Remark \ref{rem:200807-1} and the global counterpart to
\cite[Lemma 2.13]{NS-local}, whose proof is similar to the original proposition
\cite[Lemma 2.13]{NS-local}.
If we insert
(\ref{eq:211101-51})
into
(\ref{eq:211101-52}),
then 
\begin{align*}
\frac{1}{w(Q)}\int_Q f(y)w(y){\rm d}y
&\lesssim
k^{\frac{p_-}{p(x)}}
+
\frac{1}{w(Q)}\int_Q f(y)^{\frac{p(y)}{p_-}}
k^{\left(-\frac{p(x)}{p_-}+1\right)\frac{p_-}{p(x)}}w(y){\rm d}y\\
&=
k^{\frac{p_-}{p(x)}}
+
\frac{k^{\frac{p_-}{p(x)}-1}}{w(Q)}\int_Q f(y)^{\frac{p(y)}{p_-}}w(y){\rm d}y\\
&=
2k^{\frac{p_-}{p(x)}}.
\end{align*}
Thus,
from the definition of $k$, we conclude
\[
\frac{1}{w(Q)}\int_Q f(y)w(y){\rm d}y
\lesssim 
\left(
\frac{1}{w(Q)}\int_Q f(y)^{\frac{p(y)}{p_-}}w(y){\rm d}y
\right)^{\frac{p_-}{p(x)}}.
\]
Recall that $p_+<\infty$.
Hence if we take the $p(x)$-th power of the above inequality,
then we obtain the desired result.
\end{proof}
We define a variable exponent $s(\cdot)$ by
\begin{equation}\label{eq:211025-1}
\frac{1}{s(x)}\equiv\left|\frac{1}{p_\infty}-\frac{1}{p(x)}\right|
\left(\lesssim \frac{1}{\log(e+|x|)}\right)
\end{equation}
for $x \in {\mathbb R}^n$. Roughly speaking, the function
$s(\cdot)$ measures how differs $p(\cdot)$ from $p_\infty$.
It turns out that
the log-H\"{o}lder condition at infinity is transformed 
into the integrability
of
$\gamma^{s(\cdot)}$
for small $\gamma>0$.
\begin{lemma}\label{lem:210807-12}
Let $p(\cdot) \in {\mathcal P}_0$, 
$w \in A_{\infty}^{\mathfrak D}$
and $Q\in {\mathfrak D}$.
Let
$f \in L^{p(\cdot)}(w)$
be
a real-valued function
with 
$\|f\|_{L^{p(\cdot)}(w)} \le 1$.
Assume
$0 \le f \le 1$.
Then for all $\gamma \in (0,1)$,
\[
\left(\frac{1}{w(Q)}\int_Q f(y)w(y){\rm d}y\right)^{p(x)}
\lesssim 
\left(
\frac{1}{w(Q)}\int_Q f(y)^{\frac{p(y)}{p_-}}w(y){\rm d}y
\right)^{p_-}+
M_w^{\mathfrak D}[\gamma^{\frac{s(\cdot)}{p_-}}](x)^{p_-}
\]
for all $x \in Q$.
\end{lemma}

\begin{proof}
Fix $x \in Q$.
We set
\[
f_1(y)\equiv\chi_{[p(y),\infty)}(p(x))f(y), \quad
f_2(y)\equiv f(y)-f_1(y)
\]
for $y\in {\mathbb R}^n$.
Then
$f=f_1+f_2$.
Hence, we have
\begin{align*}
\left(\frac{\gamma^2}{2w(Q)}\int_Q f(y)w(y){\rm d}y\right)^{p(x)}
\le
\frac12
\sum_{j=1}^2
\left(\frac{\gamma^2}{w(Q)}\int_Q f_j(y)w(y){\rm d}y\right)^{p(x)}.
\end{align*}

As for $f_1$,
we have
\begin{align}
\label{eq:211029-12}
\frac{\gamma^2}{w(Q)}\int_Q f_1(y)w(y){\rm d}y
&\le
\frac{1}{w(Q)}\int_Q f_1(y)w(y){\rm d}y\\
\nonumber
&\le
\left(\frac{1}{w(Q)}\int_Q f_1(y)^{\frac{p(x)}{p_-}}w(y){\rm d}y\right)^{\frac{p_-}{p(x)}}\\
\nonumber
&\le
\left(\frac{1}{w(Q)}\int_Q f(y)^{\frac{p(y)}{p_-}}w(y){\rm d}y\right)^{\frac{p_-}{p(x)}}
\end{align}
by H\"{o}lder's inequality
and the fact that $f_1(y) \in [0,1]$ and $p(x) \ge p(y)$
for all $y \in Q$ such that $f_1(y) \ne 0$.

As for $f_2$, we define
a variable exponent $q(x,y)$ by
\[
\frac{1}{q(x,y)}=\frac{1}{p(x)}-\frac{1}{p(y)}>0
\]
for all $y \in Q$ with $p(x)<p(y)$.
Then
$2q(x,y) \ge \min(s(x),s(y))$,
since
\[
\frac{1}{q(x,y)} \le \frac{1}{s(x)}+\frac{1}{s(y)}
\le 
2\max\left(\frac{1}{s(x)},\frac{1}{s(y)}\right).
\]
Thus,
using
the H\"{o}lder inequality
and then
the 
Young inequality,
we obtain
\begin{align}
\label{eq:211029-11}
\left(\frac{\gamma^2}{w(Q)}\int_Q f_2(y)w(y){\rm d}y\right)^{\frac{p(x)}{p_-}}
&\le
\frac{1}{w(Q)}\int_Q \gamma^{\frac{2p(x)}{p_-}}f_2(y)^{\frac{p(x)}{p_-}}w(y){\rm d}y
\\
\nonumber
&\le
\frac{1}{w(Q)}\int_Q \left(\gamma^{\frac{2q(x,y)}{p_-}}+f(y)^{\frac{p(y)}{p_-}}\right)w(y){\rm d}y\\
\nonumber
&\le
\frac{1}{w(Q)}\int_Q \left(\gamma^{\frac{s(x)}{p_-}}+\gamma^{\frac{s(y)}{p_-}}+f(y)^{\frac{p(y)}{p_-}}\right)w(y){\rm d}y.
\end{align}
If we use the Lebesgue differentiation theorem,
then 
\begin{equation}\label{eq:211029-13}
\gamma^{\frac{s(x)}{p_-}}
\le M_w^{\mathfrak D}[\gamma^{\frac{s(\cdot)}{p_-}}](x).
\end{equation}
If we insert (\ref{eq:211029-13})
into (\ref{eq:211029-11}),
then
\begin{equation}\label{eq:211029-14}
\left(\frac{\gamma^2}{2w(Q)}\int_Q f_2(y)w(y){\rm d}y\right)^{p(x)}
\lesssim 
\left(
\frac{1}{w(Q)}\int_Q f(y)^{\frac{p(y)}{p_-}}w(y){\rm d}y
\right)^{p_-}+
M_w^{\mathfrak D}[\gamma^{\frac{s(\cdot)}{p_-}}](x)^{p_-}.
\end{equation}Combining 
(\ref{eq:211029-12}) and (\ref{eq:211029-14}),
we obtain the desired result.
\end{proof}
We use
the local $\log$-H\"{o}lder continuity at infinity
to show that $\gamma^{s(\cdot)}$ is integrable
as long as $\gamma \ll 1$.
We solidify this idea in the context of weights
as follows:
\begin{lemma}\label{lem:210807-13}
Let $p(\cdot) \in {\mathcal P} \cap {\rm LH}_0 \cap {\rm LH}_\infty$.
If $0<\gamma \ll 1$
and
$w \in A_{\infty}^{\mathfrak D}$,
then 
$M_w^{\mathfrak D}[\gamma^{\frac{s(\cdot)}{p_-}}] \in L^{p_-}(w)$.
\end{lemma}

\begin{proof}
It suffices to show that
\[
\int_{{\mathbb R}^n}M_w^{\mathfrak D}[\gamma^{\frac{s(\cdot)}{p_-}}](x)^{p_-}w(x){\rm d}x<\infty.
\]
A geometric observation shows
that $M_w^{\mathfrak D}$ 
is weak $L^1(w)$-bounded.
As we mentioned
in the beginning of Proposition \ref{prop:210805-1},
$M_w^{\mathfrak D}$  is bounded on $L^{p_-}(w)$,
assuming that $p(\cdot)$ is a constant exponent.
Thus,
thanks to the $\log$-H\"{o}lder continuity, this is equivalent to
\begin{equation}\label{eq:210807-11}
\int_{{\mathbb R}^n}\gamma^{C\log(e+|x|)}w(x){\rm d}x<\infty
\end{equation}
for some $C>0$. 
Note that
(\ref{eq:210807-11}) paraphrases
\cite[Corollary 2.14]{NS-local}.
\end{proof}
We conclude the proof of Proposition \ref{prop:210805-1}.
Let $f \in L^{p(\cdot)}(w)$ with $\|f\|_{L^{p(\cdot)}(w)} \le 1$.
Combining
Lemmas \ref{lem:210807-11} and \ref{lem:210807-12},
we have
\[
M_w^{\mathfrak D} f(x)^{p(x)}
\lesssim 
\left(M_w^{\mathfrak D}[|f(\cdot)|^{\frac{p(\cdot)}{p_-}}](x)
\right)^{p_-}+
M_w^{\mathfrak D}[\gamma^{\frac{s(\cdot)}{p_-}}](x)^{p_-}
\]
for all $x \in {\mathbb R}^n$.
Due to 
Lemma \ref{lem:210807-13},
the right-hand side is integrable
with respect to the weighted measure $w(x){\rm d}x$.
Thus, we have the desired result.

\subsection{Diening's comparison principle}
We invoke the following variant of the norm equivalence due to Diening
for weights that have at most polynomial growth.
We remark that a weight has at most polynomial growth,
if $w(\{y \in {\mathbb R}^n\,:\,|y| \le |x|\}) \lesssim
(1+|x|)^N$ for some large $N$.
\begin{lemma}[{\rm cf. \cite[Lemma 2.7]{DiHapre}}]
\label{lem:210806-11}
Let
$p(\cdot) \in {\rm L H}_0 \cap {\rm L H}_\infty \cap {\mathcal P}_0$.
Also, let
$f \in L^1_{\rm loc}({\mathbb R}^n)$
satisfy
$|f(x)| \le (1+|x|)^N$
for some large $N$.
Suppose that a weight $w$ has at most polynomial growth.
Then, 
\begin{enumerate}
\item[\rm (i)]
If
$\|f\|_{L^{p(\cdot)}(w)} \le A$,
then
$\|f\|_{L^{p_\infty}(w)}\le C_A$.
\item[\rm (ii)]
If
$\|f\|_{L^{p_\infty}(w)} \le B$,
then
$\|f\|_{L^{p(\cdot)}(w)}\le C_B$.
\end{enumerate}
Here, $C_A$ and $C_B$ are constants,
which depend on $A$ and $B$, respectively.
\end{lemma}

\begin{proof}

We let $\tilde{p}(\cdot)\equiv\min(p_\infty,p(\cdot))$
and
$p^\dagger(\cdot)\equiv\max(p_\infty,p(\cdot))$.
Denote by $X$ the set of all measurable functions
satisfying $|f(x)| \le (1+|x|)^N$
for some large $N$.
We claim that there exist constants $K_1, K_2, K_3, K_4>1$ 
with the following properties:
  \begin{enumerate}
  \item
There exists a constant $K_1 \ge 1$ such that 
$\|f\|_{L^{\tilde{p}(\cdot)}(w)} \le K_1 \|f\|_{L^{p(\cdot)}(w)}$
for $f \in L^{p(\cdot)}(w)$.  
\item
If $f \in L^{\tilde{p}(\cdot)}(w) \cap X$ satisfies
$\|f\|_{L^{\tilde{p}(\cdot)}(w)} \le 1$,
then
$\|f\|_{L^{p(\cdot)}(w)} \le K_2$.
  \item
There exists a constant $K_3 \ge 1$ such that 
$\|f\|_{L^{p(\cdot)}(w)} \le K_3 \|f\|_{L^{p^\dagger(\cdot)}(w)}$
for $f \in L^{p^\dagger(\cdot)}(w)$.  
  \item
If $f \in L^{p(\cdot)}(w) \cap X$
satisfies
$\|f\|_{L^{p(\cdot)}(w)}\le1$,
then
$\|f\|_{L^{p_\dagger(\cdot)}(w)} \le K_4$.
  \end{enumerate}
  Once
  we prove
(1)--(4),
we obtain the equivalence as follows:
\begin{itemize}
\item[(i)]
 Taking $A$ as $\max(1,A)$, we can assume that $A \ge1$.
  Suppose that 
$f \in L^{p(\cdot)}(w) \cap X$
with
$\|f\|_{L^{p(\cdot)}(w)}\le A$.
Then
by (1)
we have
$\|f\|_{L^{\tilde{p}(\cdot)}(w)} \le K_1A$.
Since $K_1A \ge 1$, we have $(K_1A)^{-1}f \in X$.
By using (4) for the exponent $\tilde{p}(\cdot)$,
$\|(K_1A)^{-1}f\|_{L^{\max(p_\infty,\tilde{p}(\cdot))}(w)} \le K_4$.
This implies that
$\|f\|_{L^{p_\infty}(w)} \le K_1K_4A$.
\item[(ii)]
For the same reason as above, we can assume that $B \ge1$.
Suppose instead that 
$f \in L^{p_\infty} \cap X$
with
$B \ge \|f\|_{L^{p_\infty}(w)}=\|f\|_{L^{\min(p_\infty, p^\dagger(\cdot))}(w)}  $.
Since $B \ge 1$, we have
$\|B^{-1}f\|_{L^{\max(p_\infty,p(\cdot))}(w)} \le K_2$
by (2).
Hence
$\|(K_2B)^{-1}f\|_{L^{\max(p_\infty,p(\cdot))}(w)} \le 1$.
This implies that
$\|f\|_{L^{p(\cdot)}(w)} \le K_2K_3B$
by (3).
\end{itemize}
So, let us prove (1)--(4).
\begin{enumerate}
\item
We prove
$L^{p(\cdot)}(w) \hookrightarrow L^{\tilde{p}(\cdot)}(w)$.  
Let $1/\tilde{p}(\cdot)=1/\tilde{r}(\cdot)+1/p(\cdot)$.
To this end, we claim
\[
\int_{{\mathbb R}^n}
\lambda^{\tilde{r}(x)}w(x){\rm d}x <\infty.
\]
Once this is proved,
we have 
$L^{p(\cdot)}(w) \hookrightarrow L^{\tilde{p}(\cdot)}(w)=L^{\min(p_\infty,p(\cdot))}(w)$
by the H\"{o}lder inequality
and the fact $w \in L^{\tilde{r}(\cdot)}({\mathbb R}^n)$.

Note that
\[
\frac{1}{\tilde{r}(\cdot)}=\max\left\{\frac{1}{p_\infty}-\frac{1}{p(\cdot)},0\right\}
\lesssim \frac{1}{\log(e+|\cdot|)}.
\]
Thus,
$\tilde{r}(\cdot) \gtrsim \log(e+|\cdot|)$.
Consequently,
assuming that $\tilde{r}(\cdot)<\infty$ everywhere
(since it is trivial that $1 \in L^\infty(w)$),
for small $\lambda \in (0,1)$,
we have
\begin{align*}
\int_{{\mathbb R}^n}
\lambda^{\tilde{r}(x)}w(x){\rm d}x
\lesssim
\sum_{j=1}^\infty
(e+j)^{c\log \lambda}w(B(j))
<\infty.
\end{align*}
\item
Suppose that
$f \in L^{\tilde{p}(\cdot)}(w) \cap X$ satisfies
$\|f\|_{L^{\tilde{p}(\cdot)}(w)} \le 1$.
Since $|f(x)| \le (1+|x|)^N$,
\[
|f(x)|^{p(x)-\tilde{p}(x)} 
\le (1+|x|)^{N(p(x)-\tilde{p}(x))} \le \max(1, e^{Nc^*}).
\]
Here, we use the estimate 
\[
(1+|x|)^{(p(x)-\tilde{p}(x))}
=
\max\left(
1, (1+|x|)^{p_\infty-\tilde{p}(x)}
\right)
\le
\max\left(
1, (1+|x|)^{\frac{c^*}{\log(e+|x|)}}
\right)
\le
e^{c^*}
\]
by the log-H\"older-type decay condition. 
This proves that
\[
\int_{{\mathbb R}^n}|f(x)|^{p(x)}w(x){\rm d}x 
\le \max(1, e^{Nc^*})
\int_{{\mathbb R}^n}|f(x)|^{\tilde{p}(x)}w(x){\rm d}x
\le
\max(1, e^{Nc^*}).
\]
Thus,
$\|f\|_{L^{p(\cdot)}(w)} \le C_{c^*}$.
\item
Define an exponent $r_\dagger(\cdot)$ by
\[
\frac{1}{\tilde{p}(\cdot)}=\frac{1}{p(\cdot)}-\frac{1}{r_\dagger(\cdot)}.
\]
We claim that
 $1 \in L^{r_\dagger(\cdot)}(w)$
if weight $w$ has at most polynomial growth.

Note that
\[
\frac{1}{r_\dagger(\cdot)}=\max\left\{\frac{1}{p(\cdot)}-\frac{1}{p_\infty},0\right\}
\lesssim \frac{1}{\log(e+|\cdot|)}.
\]
Thus,
$r_\dagger(\cdot) \gtrsim \log(e+|\cdot|)$.
  
Assuming that $r_\dagger(\cdot)<\infty$ everywhere
(since it is trivial thati $1 \in L^\infty(w)$),
for small $\lambda \in (0,1)$,
we have
\begin{align*}
\int_{{\mathbb R}^n}
\lambda^{r_\dagger(x)}w(x){\rm dx}
\lesssim
\sum_{j=1}^\infty
(e+j)^{c\log \lambda}w(B(j))
<\infty.
\end{align*}

Consequently,
since $w \in L^{r_\dagger(\cdot)}({\mathbb R}^n)$,
we have 
$L^{p(\cdot)}(w) \hookleftarrow L^{p_\dagger(\cdot)}(w)=L^{\max(p_\infty,p(\cdot))}(w)$
by the H\"{o}lder inequality.
\item

Since $|f(x)| \le (1+|x|)^N$,
\[
|f(x)|^{p_\dagger(x)-p(x)} \le (1+|x|)^{N(p_\dagger(x)-p(x))} 
\le \max(1, e^{Nc_*}).
\]
This proves that
\[
\int_{{\mathbb R}^n}|f(x)|^{p_\dagger(x)}w(x)dx 
\le \max(1, e^{Nc_*}) \int_{{\mathbb R}^n}|f(x)|^{p(x)}w(x)dx
\le \max(1, e^{Nc_*}).
\]
Hence, we have 
$\|f\|_{L^{p^\dagger(\cdot)}(w)} \le C_{c_*}$.
\end{enumerate}
\end{proof}

\subsection{Inequality in ${\mathcal D}({\mathbb R}^n)$}
\label{subsection:distribution}
In this subsection, we prepare some lemmas by Rychkov \cite{Rychkov01}.
Especially, these lemmas play an important role to consider 
the Littlewood--Paley and wavelet characterization.
(See Section \ref{s77} and Section \ref{s80}, respectively.)

First, we recall the following lemma
on the moment condition of functions:
\begin{lemma}[{\rm Grafakos\,\cite[p.466]{Gr} or \cite[p.595]{Gr14}}]\label{lem3sw}
Let $\mu,\nu\in\Bbb R$, $M,N>0$, and $L\in {\mathbb N}_0$ satisfy $\nu\geq \mu$ and $N>M+L+n$. 
Suppose that $\phi_{(\mu)}\in C^L(\Bbb R^n)$ satisfies
\begin{equation*}
|\partial^\alpha\phi_{(\mu)}(x)|\leq A_\alpha\frac{2^{\mu(n+L)}}{(1+2^{\mu}|x-x_{\mu}|)^M}\quad\text{for all \,}|\alpha|=L.
\end{equation*}
Furthermore, suppose that $\phi_{(\nu)}$ is a measurable function satisfying 
\begin{equation*}
\int_{\Bbb R^n}\phi_{(\nu)}(x) (x-x_\nu)^{\beta}\,dx=0\quad\text{for all \,}|\beta|\leq L-1,\text{ \,and \,}
|\phi_{(\nu)}(x)|\leq B\frac{2^{\nu n}}{(1+2^\nu|x-x_\nu|)^N},
\end{equation*}
where the former condition is supposed to be vacuous when $L=0$. Then it holds
\begin{equation*}
\left|\int_{\Bbb R^n}\phi_{(\mu)}(x)\phi_{(\nu)}(x)dx\right|
\leq 
C_{A_\alpha, B, L, M, N}\,2^{\mu n-(\nu-\mu)L}(1+2^\mu|x_\mu-x_\nu|)^{-M}
\end{equation*}
with a constant $C_{A_\alpha, B, L, M, N}$ taken as
\begin{equation*}
C_{A_\alpha, B, L, M, N}=B\left(\sum_{|\alpha|=L}\frac{A_\alpha}{\alpha!}\right)\frac{\omega_n(N-M-L)}{N-M-L-n},
\end{equation*}
where $\omega_n$ denotes the volume of the unit ball in $\Bbb R^n$.
\end{lemma}

We also recall the decomposition formula of Dirac's delta
and invoke \cite[Theorem 1.6]{Rychkov01}.
Note that for $\varphi \in {\mathcal D}({\mathbb R}^n)$ and $t>0$,
we write $\varphi_t\equiv t^{-n}\varphi(t^{-1}\cdot)$.
\begin{lemma}\label{lem:210712-1}
Let $L \in {\mathbb N} \cup \{-1,0\}$ 
and $\phi \in {\mathcal D}({\mathbb R}^n) \setminus {\mathcal P}_0^\perp({\mathbb R}^n)$.
Then there exist $\phi^*, \psi, \psi^* \in {\mathcal D}({\mathbb R}^n)$ such that
\[
\phi^*=\phi-2^{-n}\phi\left(\frac{\cdot}{2}\right), \quad
\phi^*, \psi^* \in {\mathcal P}_L({\mathbb R}^n), \quad
\phi*\psi+\sum_{j=1}^{\infty} \phi^*_{2^{-j}}*\psi^*_{2^{-j}}=\delta
\]
in the topology of ${\mathcal D}'({\mathbb R}^n)$.
\end{lemma}
Remark that if $\phi$ is even (resp. radial)
then the actual construction in \cite[Theorem 1.6]{Rychkov01}
shows that
$\phi^*,\psi,\psi^*$ are even (resp. radial).

Furthermore, in \cite[Lemma 2.9]{Rychkov01},
Rychkov proved the following estimate
for the functions which are constructed in Lemma \ref{lem:210712-1}:
\begin{lemma}\label{lem:211029-111}
Let $A,B,r>0$ and $L \in {\mathbb N} \cap [A,\infty)$.
Then
in Lemma \ref{lem:210712-1},
for all
$j \in {\mathbb N}_0$, $t>0$ and $f \in {\mathcal D}'({\mathbb R}^n)$,
\begin{align*}\lefteqn{
|\phi_{2^{-j}t}*f(x)|^r
}\\
&\lesssim
2^{j n}
\int_{{\mathbb R}^n}
\frac{|\phi_{2^{-j}t}*f(x-y)|^r}{m_{j,A r,B r}(y)}{\rm d}y
+
\sum_{k=j+1}^{\infty}
2^{k n+(j-k)(L+1) r}
\int_{{\mathbb R}^n}
\frac{|\phi^*_{2^{-k}t}*f(x-y)|^r}{m_{j,A r,B r}(y)}{\rm d}y
\end{align*}
and
\begin{align*}
|\phi^*_{2^{-j}t}*f(x)|^r
&\lesssim
\sum_{k=j}^{\infty}
2^{k n+(j-k)(L+1) r}
\int_{{\mathbb R}^n}
\frac{|\phi^*_{2^{-k}t}*f(x-y)|^r}{m_{j,A r,B r}(y)}{\rm d}y.
\end{align*}
In particular,
\begin{align*}
\lefteqn{
\sup_{y \in {\mathbb R}^n}
\frac{|\phi_{2^{-j}t}*f(x-y)|^r}{m_{j,A r,B r}(y)}
}\\
&\lesssim
2^{j n}
\int_{{\mathbb R}^n}
\frac{|\phi_{2^{-j}t}*f(x-y)|^r}{m_{j,A r,B r}(y)}{\rm d}y
+
\sum_{k=j+1}^{\infty}
2^{k n+(j-k)(L+1) r}
\int_{{\mathbb R}^n}
\frac{|\phi^*_{2^{-k}t}*f(x-y)|^r}{m_{j,A r,B r}(y)}{\rm d}y
\end{align*}
and
\begin{align*}
\sup_{y \in {\mathbb R}^n}
\frac{|\phi^*_{2^{-j}t}*f(x-y)|^r}{m_{j,A r,B r}(y)}
\lesssim
\sum_{k=j}^{\infty} 
2^{k n+(j-k)(L+1) r}
\int_{{\mathbb R}^n}
\frac{|\phi^*_{2^{-k}t}*f(x-y)|^r}{m_{j,A r,B r}(y)}{\rm d}y.
\end{align*}
\end{lemma}

In fact, these estimates hold if the function $\phi$ in the right-hand side is replaced
by a function having similar properties.

\begin{lemma}[{\cite[Theorem 2.5]{Rychkov01}}]\label{lem:211029-211}
In addition to the assumptions in Lemmas \ref{lem:210712-1} and \ref{lem:211029-111},
let $\zeta \in {\mathcal S}({\mathbb R}^n)$.
Then
$j \in {\mathbb N}_0$, $t>0$ and $f \in {\mathcal D}'({\mathbb R}^n)$,
\begin{align*}\lefteqn{
|\zeta_{2^{-j}t}*f(x)|^r
}\\&\lesssim
2^{j n}
\int_{{\mathbb R}^n}
\frac{|\phi_{2^{-j}t}*f(x-y)|^r}{m_{j,A r,B r}(y)}{\rm d}y
+
\sum_{k=j+1}^{\infty}
2^{k n+(j-k)(L+1) r}
\int_{{\mathbb R}^n}
\frac{|\phi^*_{2^{-k}t}*f(x-y)|^r}{m_{j,A r,B r}(y)}{\rm d}y.
\end{align*}
\end{lemma}

\section{Fundamental properties of $h^{p(\cdot)}(w)$ (including the proof of Theorem \ref{thm:210712-2})}
\label{s53}

Here,
we investigate the structure of
$h^{p(\cdot)}(w)$.
We first
verify that ${\mathcal D}'({\mathbb R}^n)$ is a suitable space
to consider $h^{p(\cdot)}(w)$.
Note that Propositions \ref{prop:190406-2} and \ref{prop:190406-3}
are proved
by Tang \cite[Propositions 3.1 and 3.2]{Tang12} 
when $p(\cdot)$ is a constant exponent
in $(0,1]$.

\begin{proposition}\label{prop:190406-2}
Let 
$w \in A^{\rm loc}_\infty$.
If $N \ge N_{p(\cdot),w}$,
then the inclusion
$h^{p(\cdot)}(w) \hookrightarrow {\mathcal D}'({\mathbb R}^n)$
is continuous.
\end{proposition}
\begin{proof}
The proof is the same as that in
\cite[Proposition 3.1]{Tang12},
where we take the $L^{p(\cdot)}(w)$-norm
instead of $L^p(w)$-norm.
\end{proof}

\begin{proposition}\label{prop:190406-3}
Let 
$w \in A^{\rm loc}_\infty$.
If $N \ge N_{p(\cdot),w}$,
then 
$h^{p(\cdot)}(w)$
is complete.
\end{proposition}

\begin{proof}
This is a direct consequence of Proposition \ref{prop:190406-2}.
The proof is similar to \cite[Lemma 2.15]{Rychkov01}.
Here, we omit the details.
\end{proof}

We prove Theorem \ref{thm:210712-2}.
\begin{proof}[Proof of Theorem \ref{thm:210712-2}]
Let $f \in h^{p(\cdot)}(w)$.
Take $\psi \in {\mathcal D}({\mathbb R}^n) \setminus {\mathcal P}_0({\mathbb R}^n)^\perp$.
Write $\psi_t\equiv t^{-n}\psi(t^{-1}\cdot)$ as before.
Then $\{\psi_t*f\}_{t>0}$
is a bounded set
of $L^{p(\cdot)}(w)=(L^{p'(\cdot)}(\sigma))^*$.
By
the Banach--Alaoglu theorem
there exists a sequence
$\{t_j\}_{j=1}^{\infty}$ decreasing to $0$
such that
$\{\psi_{t_j}*f\}_{j=1}^{\infty}$
converges to a function $g$
in the weak-* topology of $L^{p(\cdot)}(w)$.
Meanwhile, it can be shown that 
$\lim\limits_{t \downarrow 0}\psi_t*f=f$
in the topology of ${\mathcal D}'({\mathbb R}^n)$.
Since the weak-* topology of $L^{p(\cdot)}(w)$
is stronger than the topology of ${\mathcal D}'({\mathbb R}^n)$,
it follows that $f=g \in L^{p(\cdot)}(w)$.
\end{proof}

\section{Proof of Theorem \ref{thm:210712-1}}
\label{s54}
From the definition of the three local grand maximal operators,
it suffices to handle the most right-hand inequality. That is, 
\[
\|{\mathcal M}_Nf\|_{L^{p(\cdot)}(w)} 
\lesssim \|{\mathcal M}_N^0f\|_{L^{p(\cdot)}(w)}.
\]
Let $L \in {\mathbb N}$ be sufficiently large.
Fix $x \in {\mathbb R}^n$ for now.
Let
$(z,t) \in {\mathbb R}^n_+$ satisfy
$|x-z|<t \le 1$.
Fix $\phi \in {\mathcal D}({\mathbb R}^n) \setminus {\mathcal P}_0^\perp$ 
and take $\phi^*, \psi, \psi^* \in {\mathcal D}({\mathbb R}^n)$
as in Lemma \ref{lem:210712-1}.

Then 
\begin{align*}
\varphi_t*f(z)
=
[\varphi_t*\psi_t(\cdot-x+z)]*\phi_t*f(x)
+
\sum_{j=1}^\infty
[\varphi_t*\psi^*_{2^{-j}t}(\cdot-x+z)]*\phi^*_{2^{-j}t}*f(x).
\end{align*}
Let $A>\frac{n}{r}$, $B>\frac{8n+6\log D}{r}$
 and
 assume $L\in{\mathbb N} \cap (A,\infty)$.
By the assumption on $N$,
we can assume that
\[
\frac{p_-}{r}>q_w.
\]
By the moment condition on $\phi^*$ and $\psi^*$
and the equality
$\phi^*=\phi-2^{-n}\phi(2^{-1}\cdot)$,
\begin{align*}
{\mathcal M}_Nf(x)&=
\sup_{(z,t) \in {\mathbb R}^n, |z-x|<t \le 1}|\varphi_t*f(z)|\\
&\lesssim
\left\{
\sum_{j=0}^{\infty}
2^{-j(L-A)r}
M^{\rm loc}\left[
\int_{{\mathbb R}^n}
\sup_{t \in (0,1]}
\frac{|\phi_{2^{-j}t}*f(\cdot-y)|^r}{(1+2^j|y|)^{A r}2^{|y|B r}}
dy
\right](x)\right\}^{\frac{1}{r}}
\\
&\lesssim
\left\{
\sum_{j=0}^{\infty}
2^{-j(L-n-A)r}
M^{\rm loc} \circ K_{B r}\left[
\sup_{t \in (0,1]}|\phi_{2^{-j}t}*f|^r\right](x)\right\}^{\frac{1}{r}}\\
&\lesssim
\left\{
M^{\rm loc} \circ K_{B r}\left[({\mathcal M}_N^0f)^r\right](x)\right\}^{\frac{1}{r}}.
\end{align*}
If we use Lemmas \ref{lem:190409-100} and \ref{lem:210923-1},
 we obtain the desired result.
\section{Proof of Theorem \ref{thm:190406-2}, including the implication $({\rm II})/({\rm III})\Longrightarrow({\rm I})$ in Theorem \ref{cor:211009-1}}
\label{s55}

The proof of Theorem \ref{thm:190406-2}
has two parts.
One controls
the sum of atoms
by the norm of $h^{p(\cdot)}(w)$
(Sections \ref{subsection:5.3} and \ref{subsection:5.1}).
The other decomposes
the
distributions $h^{p(\cdot)}(w)$
into the sum of atoms
(Section \ref{subsection:5.2}).

\subsection{Some norm estimates}
\label{subsection:5.3}

Here and below we write
\[
m_{Q,w}^{(u)}(f)\equiv \frac{\|\chi_Qf\|_{L^u(w)}}{w(Q)^{\frac1u}}
\]
for
a cube $Q$, $0<u<\infty$ and $f \in L^0({\mathbb R}^n)$.
We establish the following key estimate
for
the proof of Theorem \ref{thm:190406-2}
and
the proof of the implication $({\rm II})\Longrightarrow({\rm I})$ 
in Theorem \ref{cor:211009-1}:
\begin{theorem}\label{lem:210805-1}
Let $p(\cdot) 
\in {\mathcal P} \cap {\rm L H}_0 \cap {\rm L H}_\infty$
and
 $w \in A_{\infty}^{\rm loc}$.
Assume that
$(u,v) \in (0,\infty)\times((0,p_-) \cap [0,1])$
satisfies $u v>p_+$.
Suppose
$f_j \in L^{u v}({\mathbb R}^n)$ which is supported on a cube $Q_j$
with $|Q_j| \le 1$
for each $j$.
Then
\begin{align*}
\left\|
\left(\sum_{j=1}^{\infty}|f_j|^{v}\right)^{\frac{1}{v}}
\right\|_{L^{p(\cdot)}(w)}
&\lesssim
{\mathcal A}_{p(\cdot),w,v}(\{m_{Q_j,w}^{(u v)}(f_j)\}_{j=1}^{\infty};\{Q_j\}_{j=1}^{\infty}).
\end{align*}
\end{theorem}
It is noteworthy that the case where $v=1<p_-$ proves the implication $({\rm II})\Longrightarrow({\rm I})$ 
in Theorem \ref{cor:211009-1}.
\begin{proof}
Write $P(\cdot)\equiv\frac{p(\cdot)}{v}$
and $\Sigma \equiv w^{-\frac{1}{P(\cdot)-1}}$.
Since $p_->v$, $P_->1$.
By duality
and the definition of $P(\cdot)$, the matters are reduced to the estimate
\begin{equation}\label{eq:211029-15}
\left(\int_{{\mathbb R}^n}\sum_{j=1}^{\infty} |f_j(x)|^{v}|g(x)|{\rm d}x\right)^{\frac{1}{v}}
\lesssim
{\mathcal A}_{p(\cdot),w,v}(\{m_{Q_j,w}^{(u v)}(f_j)\}_{j=1}^{\infty};\{Q_j\}_{j=1}^{\infty})
\left(\|g\|_{L^{P'(\cdot)}(\Sigma)}\right)^{\frac{1}{v}}
\end{equation}
for all $g \in L^0({\mathbb R}^n)$.
Since the functions
in the summand
of the left-hand side are non-negative,
\begin{align*}
\left(\int_{{\mathbb R}^n}\sum_{j=1}^{\infty} |f_j(x)|^{v}|g(x)|{\rm d}x\right)^{\frac{1}{v}}
&=
\left(\sum_{j=1}^{\infty} \int_{{\mathbb R}^n}|f_j(x)|^{v}|g(x)|{\rm d}x\right)^{\frac{1}{v}}\\
&=
\left(\sum_{j=1}^{\infty} \int_{{\mathbb R}^n}|f_j(x)|^{v}|g(x)w(x)^{-1}|w(x){\rm d}x\right)^{\frac{1}{v}}.
\end{align*}By the H\"{o}lder inequality,
we have
\begin{align*}
\left(\int_{{\mathbb R}^n}\sum_{j=1}^{\infty} |f_j(x)|^{v}|g(x)|{\rm d}x\right)^{\frac{1}{v}}
&\le
\left\{
\sum_{j=1}^{\infty}
w(Q_j)
m_{Q_j,w}^{(u)}(|f_j|^v)m_{Q_j,w}^{(u')}(g w^{-1})
\right\}^{\frac{1}{v}}.
\end{align*}
Using the
powered weighted local maximal operator
$M^{(u'),\rm loc}_w$,
we have
\begin{align*}
\left(\int_{{\mathbb R}^n}\sum_{j=1}^{\infty} |f_j(x)|^{v}|g(x)|{\rm d}x\right)^{\frac{1}{v}}
\le
\left\{
\int_{{\mathbb R}^n}
\sum_{j=1}^{\infty} m_{Q_j,w}^{(u)}(|f_j|^v)\chi_{Q_j}(x)M^{(u'),\rm loc}_w(gw^{-1})(x)w(x){\rm d}x
\right\}^{\frac{1}{v}}.
\end{align*}
By the H\"{o}lder inequality for Lebesgue spaces with variable exponents
(see Lemma \ref{thm:gHolder}),
we have
\begin{align}
\lefteqn{
\left(\int_{{\mathbb R}^n}\sum_{j=1}^{\infty} |f_j(x)|^{v}|g(x)|{\rm d}x\right)^{\frac{1}{v}}
}
\nonumber\\
&\le
2^{\frac1v}
\left\{
{\mathcal A}_{P(\cdot),w,1}(\{m_{Q_j,w}^{(u)}(|f_j|^v)\}_{j=1}^{\infty};\{Q_j\}_{j=1}^{\infty})
\left\|M^{(u'),\rm loc}_w(gw^{-1})w\right\|_{L^{P'(\cdot)}(\Sigma)}
\right\}^{\frac{1}{v}}
\nonumber\\
\label{eq:211029-16}
&=
2^{\frac1v}
{\mathcal A}_{p(\cdot),w,v}(\{m_{Q_j,w}^{(u v)}(f_j)\}_{j=1}^{\infty};\{Q_j\}_{j=1}^{\infty})
\left(
\left\|M^{(u'),\rm loc}_w(gw^{-1})\right\|_{L^{P'(\cdot)}(w)}
\right)^{\frac1v}.
\end{align}
Recall that we assume
$u v>p_+$. That is, $P_+<u$. Hence
$(P')_->u'$.
Since
 $w \in A_{\infty}^{\rm loc}$,
by Proposition \ref{prop:210805-1},
we have
\begin{align}
\nonumber
\|M^{(u'),\rm loc}_w(gw^{-1})\|_{L^{P'(\cdot)}(w)}
&=
(\|M^{\rm loc}_w[|gw^{-1}|^{u'}]\|_{L^{P'(\cdot)/u'}(w)})^{\frac{1}{u'}}\\
\nonumber
&\lesssim
\|gw^{-1}\|_{L^{P'(\cdot)}(w)}\\
\label{eq:211029-17}
&=
\|g\|_{L^{P'(\cdot)}(\Sigma)}.
\end{align}
Combining
(\ref{eq:211029-16})
and
(\ref{eq:211029-17}),
gives
(\ref{eq:211029-15}).
Thus, the proof is complete.
\end{proof}
As
mentioned,
the implication $({\rm I})\Longrightarrow({\rm II})/({\rm III})$ is included in
Theorems 
\ref{thm:190406-2}
and
\ref{thm:210712-2}.
Now let us 
prove the implication $({\rm II})/({\rm III})\Longrightarrow({\rm I})$.

We rephrase and prove
the implication $({\rm III})\Longrightarrow({\rm I})$ of
in Theorem \ref{cor:211009-1}
as follows:
\begin{theorem}\label{thm:211201-1}
Let
$p(\cdot) \in {\mathcal P} \cap {\rm L H}_0 \cap {\rm L H}_\infty$.
Also let
$w \in A^{\rm loc}_{p(\cdot)}$
and
$q_0>p_+$.
Assume that $\sigma=w^{-\frac{1}{p(\cdot)-1}} \in A^{\rm loc}_{p'(\cdot)/q_0'}$
If we have a collection
$\{\lambda_j\}_{j=1}^{\infty}$
of complex constants
and
a collection $\{a_j\}_{j=1}^{\infty} \subset L^{q_0}({\mathbb R}^n)$
such that
each $a_j, j \in {\mathbb N}$ is 
a $(p(\cdot),{q_0},L)$-atom supported on a cube $Q_j$
with $|Q_j| \le 1$ and that
\[
{\mathcal A}_{p(\cdot),w,1}(\{\lambda_j\}_{j=1}^{\infty};\{Q_j\}_{j=1}^{\infty})= 
\left\|\sum_{j=1}^{\infty} |\lambda_j|\chi_{Q_j}
\right\|_{L^{p(\cdot)}(w)}
<\infty,
\]
then
\[
f=\sum\limits_{j=1}^{\infty} \lambda_j a_j
\]
converges in $L^{p(\cdot)}(w)$ and 
satisfies
\begin{equation}\label{eq:211201-201}
\left\|f\right\|_{L^{p(\cdot)}(w)}
\lesssim
{\mathcal A}_{p(\cdot),w,1}
(\{\lambda_j\}_{j=1}^{\infty};\{Q_j\}_{j=1}^{\infty}).
\end{equation}
\end{theorem}Unlike
Theorem \ref{cor:211009-1},
we do not have to distinguish two cases.
In fact, if ${q_0}>p_+$ and $w \in L^1({\mathbb R}^n)$,
then $L^{q_0}(w) \subset L^{p(\cdot)}(w)$.
\begin{proof}
We can assume that each $a_j$ is a non-negative function
and each $\lambda_j$ is a non-negative real number.
We dualize the conclusion
\begin{equation}\label{eq:211201-301}
\int_{{\mathbb R}^n}f(x)g(x){\rm d}x
\lesssim
{\mathcal A}_{p(\cdot),w,1}(\{\lambda_j\}_{j=1}^{\infty};\{Q_j\}_{j=1}^{\infty})
\|g\|_{L^{p'(\cdot)}(\sigma)},
\end{equation}
where
 $g\in L^0({\mathbb R}^n)$ is a non-negative function.

Using H\"{o}lder's inequality twice,
we have
\begin{align*}
\int_{{\mathbb R}^n}
\sum\limits_{j=1}^{\infty} \lambda_j a_j(x)g(x){\rm d}x
&\le
\int_{{\mathbb R}^n}
\sum\limits_{j=1}^{\infty} \lambda_j m^{({q_0})}_{Q_j}(a_j)m_{Q_j}^{(q_0')}(g)\chi_{Q_j}(x){\rm d}x\\
&\le
\int_{{\mathbb R}^n}
\sum\limits_{j=1}^{\infty} \lambda_j m^{({q_0})}_{Q_j}(a_j)M^{\rm loc}[g^{q_0'}](x)^{\frac{1}{q_0'}}
\chi_{Q_j}(x){\rm d}x\\
&\lesssim
{\mathcal A}_{p(\cdot),w,1}(\{\lambda_j\}_{j=1}^{\infty};\{Q_j\}_{j=1}^{\infty})
\|(M^{\rm loc}[g^{q_0'}])^{\frac{1}{q_0'}}\|_{L^{p'(\cdot)}(\sigma)}.
\end{align*}
Since ${q_0}>p_+(>1)$, we have $q_0'<p'_-=(p_+)'$.
Since we assume $\sigma \in A^{\rm loc}_{p'(\cdot)/q_0'}$,
\begin{equation}\label{eq:211201-302}
\|(M^{\rm loc}[g^{q_0'}])^{\frac{1}{q_0'}}\|_{L^{p'(\cdot)}(\sigma)}
\lesssim
\|g\|_{L^{p'(\cdot)}(\sigma)}.
\end{equation}If we insert
(\ref{eq:211201-302})
into
the above expression,
we obtain 
(\ref{eq:211201-301}).
\end{proof}

\subsection{Proof of $h^{p(\cdot),q,L;v}(w) \hookrightarrow h^{p(\cdot)}(w)$}
\label{subsection:5.1}

The following estimate is a passage of 
\cite[(3.2)]{Tang12} to the setting of variable exponents.
Recall that 
\[
N_{p(\cdot),w}
\equiv 
2+\left[n\left(\frac{q_w}{\min(1,p_-)}-1\right)\right].
\]
\begin{lemma}\label{lem:190408-3}
Let $w \in A^{\rm loc}_\infty \cap L^1({\mathbb R}^n)$ and $p(\cdot) \in {\mathcal P} \cap {\rm L H}_0 \cap {\rm L H}_\infty$.
We assume that $q>\max(q_w,p_+)$.
Let $a$ be a single $(p(\cdot),q)_w$-atom.
Then
$\|{\mathcal M}^0_{N_{p(\cdot),w}}a\|_{L^{p(\cdot)}(w)} \lesssim 1.$
\end{lemma}

\begin{proof}
Since $q>q_w$, $w \in A^{\rm loc}_q$.
Thus,
${\mathcal M}^0_{N_{p(\cdot),w}}$ is bounded on $L^q(w)$ 
(\cite[Proposition 2.2]{Tang12}).
Define an exponent $r(\cdot)$ by $1/p(\cdot)=1/q+1/r(\cdot)$
while recalling that $q>p_+$.
Using the H\"{o}lder inequality
(see Lemma \ref{thm:gHolder}), we have
\begin{align*}
\|{\mathcal M}^0_{N_{p(\cdot),w}}a\|_{L^{p(\cdot)}(w)}
&\lesssim
\|{\mathcal M}^0_{N_{p(\cdot),w}}a\|_{L^{q}(w)}
\|\chi_{{\mathbb R}^n}\|_{L^{r(\cdot)}(w)}\\
&\lesssim
\|a\|_{L^{q}(w)}
\|\chi_{{\mathbb R}^n}\|_{L^{r(\cdot)}(w)}
\le
w({\mathbb R}^n)^{\frac1q}\|\chi_{{\mathbb R}^n}\|_{L^{r(\cdot)}(w)} \lesssim1.
\end{align*}
This is the desired result.
\end{proof}

Tang pointed out an important
feature
of
${\mathcal M}^0_{N_{p(\cdot),w}}a$
for $(p(\cdot),q,L)_w$-atoms
$a$.
\begin{lemma}{\rm\cite[(3.3)]{Tang12}}\label{lem:190408-1}
Let $w \in A^{\rm loc}_\infty$ and $p(\cdot) \in {\mathcal P} \cap {\rm L H}_0 \cap {\rm L H}_\infty$.
Suppose $L \in {\mathbb Z} \cap [-1,N_{p(\cdot),w}]$ and $q>q_w$.
Let $a$ be a $(p(\cdot),q,L)_w$-atom supported on a cube $Q=Q(x_0,r)$
with
$|Q|<1$. Then
\begin{equation}\label{eq:220314-1}
{\mathcal M}^0_{N_{p(\cdot),w}}a(x)
\lesssim
(M^{\rm loc}\chi_{Q}(x))^{\frac{n+L+1}{n}}
\end{equation}
for $x \in {\mathbb R}^n \setminus 2Q$.
\end{lemma}
Here we recall 
the proof of Lemma \ref{lem:190408-1}
since we must rephrase in terms of the local maximal operator.
\begin{proof}
If $x \in {\mathbb R}^n \setminus Q(x_0,4n)$,
then ${\mathcal M}^0_{N_{p(\cdot),w}}a(x)=0$.
So, we assume that
$x \in Q(x_0,4n) \setminus 2Q$.
Let $\varphi \in {\mathcal D}_N^0$.
By the support condition of $a$, if ${\rm supp} \,a \cap {\rm supp} \,\varphi_t(x-\cdot) \neq \phi$,
then
$r<t$ and $|x-x_0|<2t$.
Let 
$P$ be the Taylor expansion of $\varphi$ at the point $(x-x_0)/t$ of order 
$L \le N_{p(\cdot),w}$. 
Since $a$ is a $(p(\cdot), q, L)_w$-atom, we have 
\begin{align*}
\left|a*\varphi_t(x)\right|
&=
\left|t^{-n}
\int_{{\mathbb R}^n} a(y) \left(\varphi\left(\frac{x-y}{t}\right)-P\left(\frac{x-y}{t}\right)\right) {\rm d}y\right|.
\end{align*}
Thus,
by the Taylor remainder theorem, 
\begin{equation}\label{eq:220209-1111}
\left|a*\varphi_t(x)\right|
\lesssim
t^{-n} \int_Q |a(y)| \left|\frac{x_0-y}{t}\right|^{L+1} {\rm d}y.
\end{equation}
By the triangle inequality and H\"{o}lder's inequality,
\begin{align*}
\left|a*\varphi_t(x)\right|
&\lesssim
|x-x_0|^{-n-L-1} \ell(Q)^{L+1} \int_Q |a(y)| {\rm d}y\\
&\le
|x-x_0|^{-n-L-1} \ell(Q)^{L+1} 
\|a\|_{L^q(w)} \left(\int_Q w(y)^{-\frac{1}{q-1}} {\rm d}y\right)^{1-\frac1q}\\
&\le
|x-x_0|^{-n-L-1} \ell(Q)^{L+1}w(Q)^{\frac1q}\left(\int_Q w(y)^{-\frac{1}{q-1}} {\rm d}y\right)^{1-\frac1q}\\
&\lesssim
(M^{\rm loc}\chi_{Q}(x))^{\frac{n+L+1}{n}}|Q|^{-1}w(Q)^{\frac1q}\left(\int_Q w(y)^{-\frac{1}{q-1}} {\rm d}y\right)^{1-\frac1q}.
\end{align*}
Since $w \in A_q^{\rm loc}$, we have
\begin{align*}
\left|a*\varphi_t(x)\right|
&\lesssim
(M^{\rm loc}\chi_{Q}(x))^{\frac{n+L+1}{n}}.
\end{align*}
If we take the supremum over $t \in (0,1)$,
then the desired inequality is obtained.
\end{proof}




Keeping Lemma \ref{lem:190408-1}
in mind, let us prove
the one inclusion of Theorem \ref{thm:190406-2}.
That is, 
we will prove
$h^{p(\cdot),q,L;v}(w) \hookrightarrow h^{p(\cdot)}(w)$.

We start with the setup.
Let
$u>0$ satisfy
\[
\max(q_w,p_+)<u v<q.
\]
Since
$w \in A_{p(\cdot)}^{\rm loc}$
and
$q>q_w$,
$w \in A_q^{\rm loc}$
thanks to Proposition \ref{prop:211027-111}.

Let
$f \in h^{p(\cdot),q,L;v}(w)$.
We suppose that $w({\mathbb R}^n)<\infty$.
Otherwise we can modify the proof below.
There is a decomposition
$
f=\sum\limits_{j=0}^{\infty} \lambda_j a_j+\sum\limits_{j=1}^{\infty} \lambda_j' b_j,
$
where
$a_0$ is a simple $(p(\cdot),q)_w$-atom,
each $a_j$, $j \in {\mathbb N}$ is a $(p(\cdot),q,L)_w$-atom
supported on $Q_j=Q(x_j,r_j)$ with $r_j <1/2$,
each $b_j$, $j \in {\mathbb N}$ is a $(p(\cdot),q,L)_w$-atom
supported on $R_j=Q(x_j,1/2)$
and
coefficients
$\{\lambda_j\}_{j=0}^\infty$
and
$\{\lambda_j'\}_{j=1}^\infty$ satisfy
\[
|\lambda_0|
+
{\mathcal A}_{p(\cdot),w,v}(
\{\lambda_j\}_{j=1}^\infty;
\{Q_j\}_{j=1}^\infty)
+
{\mathcal A}_{p(\cdot),w,v}(
\{\lambda_j'\}_{j=1}^\infty;
\{R_j\}_{j=1}^\infty)
<\infty.
\]
By the triangle inequality,
the sublinearlity of
${\mathcal M}^0_{N_{p(\cdot)},w}$ and Lemma \ref{lem:190408-1},
we have
\begin{align*}
\lefteqn{
{\mathcal M}^0_{N_{p(\cdot),w}}f
}\\
&\le
|\lambda_0|{\mathcal M}^0_{N_{p(\cdot),w}}a_0
+
\sum\limits_{j=1}^{\infty} |\lambda_j| {\mathcal M}^0_{N_{p(\cdot),w}}a_j
+
\sum\limits_{j=1}^{\infty} |\lambda_j'| {\mathcal M}^0_{N_{p(\cdot),w}}b_j\\
&\lesssim
|\lambda_0|{\mathcal M}^0_{N_{p(\cdot),w}}a_0
+
\sum\limits_{j=1}^{\infty} |\lambda_j| 
(M^{\rm loc}\chi_{Q(x_j,r_j)})^{\frac{n+L+1}{n}}
+
\sum\limits_{j=1}^{\infty} |\lambda_j| 
\chi_{2Q_j}M^{\rm loc}a_j
+
\sum\limits_{j=1}^{\infty} |\lambda_j'| {\mathcal M}^0_{N_{p(\cdot),w}}b_j.
\end{align*}
The first term is controlled by
Lemma \ref{lem:190408-3}
as
\begin{equation}\label{eq:211124-201}
\|{\mathcal M}^0_{N_{p(\cdot),w}}a_0\|_{L^{p(\cdot)}(w)}
\lesssim 1.
\end{equation}
To handle the second term we use
\[
\sum\limits_{j=1}^{\infty} |\lambda_j| 
(M^{\rm loc}\chi_{Q(x_j,r_j)})^{\frac{n+L+1}{n}}
\le\
\left(
\sum\limits_{j=1}^{\infty} |\lambda_j|^v
(M^{\rm loc}\chi_{Q(x_j,r_j)})^{v\frac{n+L+1}{n}}
\right)^{\frac{1}{v}}.
\]
We
recall that $v \le \min(1,p_-)$ and that (\ref{eq:211124-101}) holds.
Arithmetic shows that
\[
 v\frac{n+L+1}{n}
\ge
\frac{v}{n}
\left(n+1+\left[n\left(\frac{q_w}{v}-1\right)\right]\right)
>q_w \ge1,
\]
and that
\[
\frac{n+L+1}{n}p(\cdot)
\ge
\frac{p_-}{n}
\left(n+1+\left[n\left(\frac{q_w}{v}-1\right)\right]\right)
>q_w\frac{p_-}{v}
>q_w
\ge1.
\]
So we can use the vector-valued inequality
of $M^{\rm loc}$
for weighted Lebesgue spaces with variable exponents
(see Lemma \ref{lem:190408-100})
to give
\begin{equation}\label{eq:211124-202}
\left\|
\sum\limits_{j=1}^{\infty} |\lambda_j| 
(M^{\rm loc}\chi_{Q(x_j,r_j)})^{\frac{n+L+1}{n}}\right\|_{L^{p(\cdot)}(w)}
\lesssim
{\mathcal A}_{p(\cdot),w,v}(
\{\lambda_j\}_{j=1}^\infty;
\{Q_j\}_{j=1}^\infty).
\end{equation}

For the third and fourth terms,
while recalling that ${\mathcal M}^0_{N_{p(\cdot)},w}b_j$ is supported
on $20n R_j$,
we use 
Theorem \ref{lem:210805-1}
to deduce
\begin{align}\label{eq:211101-101}
\left\|
\sum\limits_{j=1}^{\infty} |\lambda_j| 
\chi_{2Q_j}M^{\rm loc}a_j
\right\|_{L^{p(\cdot)}(w)}
\lesssim
{\mathcal A}_{p(\cdot),w,v}(
\{\lambda_jm_{2Q_j,w}^{(u v)}(M^{\rm loc}a_j)\}_{j=1}^\infty;
\{2Q_j\}_{j=1}^\infty)
\end{align}
and
\begin{align}\label{eq:211101-102}
\left\|
\sum\limits_{j=1}^{\infty} |\lambda_j'| {\mathcal M}^0_{N_{p(\cdot),w}}b_j
\right\|_{L^{p(\cdot)}(w)}
\lesssim
{\mathcal A}_{p(\cdot),w,v}(
\{\lambda_jm_{20n R_j,w}^{(u v)}(M^{\rm loc}b_j)\}_{j=1}^\infty;
\{20n R_j\}_{j=1}^\infty).
\end{align}

We estimate $m_{Q_j,w}^{(u v)}(M^{\rm loc}a_j)$.
Recall that $q>q_w$.
Hence,
$M^{\rm loc}$ is bounded on $L^q(w)$.
Since we assume that 
$1<u v<q$, we choose $r\in(0, \infty)$ such that $\frac1{u v}=\frac1q+\frac1r$.
Then, using the H\"{o}lder inequality
and the condition of the $(p(\cdot), q, L)_w$-atom, 
we have
\begin{align*}
m_{2Q_j,w}^{(u v)}(M^{\rm loc}a_j)
&=
\frac{1}{w(Q_j)^{\frac1{u v}}} 
\left\|\chi_{2Q_j}M^{\rm loc}a_j\right\|_{L^{u v}(w)}\\
&\le
\frac{1}{w(Q_j)^{\frac1{u v}}} \left\|M^{\rm loc}a_j\right\|_{L^q(w)}
\left\|\chi_{2Q_j}\right\|_{L^r(w)}\\
&\lesssim
\frac{1}{w(Q_j)^{\frac1{u v}}} \left\|a_j\right\|_{L^q(w)}
\left\|\chi_{2Q_j}\right\|_{L^r(w)}
\le
\frac{1}{w(Q_j)^{\frac1{u v}}} w(Q_j)^{\frac1q}w(Q_j)^{\frac1r}
=1.
\end{align*}
A geometric observation shows that
$\chi_{2Q_j} \lesssim M^{\rm loc}\chi_{Q_j}$.
By virtue of Lemma
\ref{lem:190408-100}
and
Theorem \ref{lem:210805-1}
along with the above estimate,
we obtain
\begin{align*}
\left\|
\sum\limits_{j=1}^{\infty} |\lambda_j| 
\chi_{2Q_j}M^{\rm loc}a_j
\right\|_{L^{p(\cdot)}(w)}
&\lesssim
{\mathcal A}_{p(\cdot),w,v}(
\{\lambda_jm_{2Q_j,w}^{(u)}(M^{\rm loc}a_j)\}_{j=1}^\infty;
\{2Q_j\}_{j=1}^\infty)\\
&\lesssim
{\mathcal A}_{p(\cdot),w,v}(
\{\lambda_j\}_{j=1}^\infty;
\{2Q_j\}_{j=1}^\infty)\\
&\lesssim
{\mathcal A}_{p(\cdot),w,v}(
\{\lambda_j\}_{j=1}^\infty;
\{Q_j\}_{j=1}^\infty).
\end{align*}
In total,
\begin{equation}\label{eq:211124-203}
\left\|
\sum\limits_{j=1}^{\infty} |\lambda_j| 
\chi_{2Q_j}M^{\rm loc}a_j
\right\|_{L^{p(\cdot)}(w)}
\lesssim
{\mathcal A}_{p(\cdot),w,v}(
\{\lambda_j\}_{j=1}^\infty;
\{Q_j\}_{j=1}^\infty).
\end{equation}
In a similar fashion,
we estimate 
the right-hand side of (\ref{eq:211101-102}).
The result is
\begin{equation}\label{eq:211124-204}
\left\|
\sum\limits_{j=1}^{\infty} |\lambda_j'| {\mathcal M}^0_{N_{p(\cdot),w}}b_j
\right\|_{L^{p(\cdot)}(w)}
\lesssim
{\mathcal A}_{p(\cdot),w,v}(
\{\lambda_j'\}_{j=1}^\infty;
\{R_j\}_{j=1}^\infty).
\end{equation}
Combining (\ref{eq:211124-201}),
(\ref{eq:211124-202}),
(\ref{eq:211124-203})
and
(\ref{eq:211124-204}), we obtain the desired result.

Let us reexamine the above proof to polish the conditions
on atoms.

\begin{remark}\label{rem:210921-114}
Let $v\in(0,p_-) \cap[0,1]$
and $s_0\equiv\left[n\left(\frac{q_w}{v}-1\right)\right]_+$.
A close inspection of the proof shows that
the condition on
the atom $a_j$
may be relaxed.
It suffices to assume
each $a_j$ satisfies the pointwise estimate
\begin{equation}\label{eq:211125-401}
|a_j| \le |a_j|\chi_{3Q_j}+(M^{\rm loc}\chi_{Q_j})^{\frac{n+s_0+1}{n}}
\end{equation}
and the norm estimate
$\|a_j\chi_{3Q_j}\|_{L^q(w)} \le w(Q_j)^{\frac1q}$
for $q<\infty$.
In fact, using the same argument as Lemma \ref{lem:220209-1} below,
we can establish estimate (\ref{eq:220314-1}) under a milder condition
(\ref{eq:211125-401}).
Thus, it is not necessary to assume that each $a_j$
is compactly supported.
Instead, 
a weaker assumption
(\ref{eq:211125-401})
is sufficient.
\end{remark}

\subsection{Proof of $h^{p(\cdot),q,L;v}(w) \hookleftarrow h^{p(\cdot)}(w)$}
\label{subsection:5.2}

Let us consider the decomposition
following \cite[pp.
461--462]{Tang12}.
Set
$f \in {\mathcal D}'({\mathbb R}^n)$ and $\lambda>0$.
We set
\[
\Omega_\lambda\equiv \{x \in {\mathbb R}^n\,:\,{\mathcal M}_N f(x)>\lambda\}.
\]
Here and below, consider a distribution
$f \in {\mathcal D}'({\mathbb R}^n)$
such that
$
w(\Omega_\lambda)<\infty
$
for all $\lambda>0$.
From the Whitney decomposition
$\{Q_k\}_{k \in K}$
of $\Omega_\lambda$,
a decomposition
$\Omega_\lambda=\bigcup\limits_{k \in K}Q_k$
such that
\[
{\rm diam}(Q_k)
\le
2^{-n-6}
{\rm dist}(Q_k,{\mathbb R}^n \setminus \Omega_\lambda)
\le
4{\rm diam}(Q_k) \quad ({k \in K})
\]
and that the overlapping property is satisfied
as
\[
\sum_{k \in K}\chi_{(1+2^{-n-10})Q_k} \lesssim 1.
\]
Fix $k \in K$.
Set
\begin{equation}\label{eq:211031-1}
Q_k\equiv Q(x_k,l_k), \quad
Q_k^*\equiv (1+2^{-n-10})Q_k.
\end{equation}
Let $\xi \in C^{\infty}({\mathbb R}^n)$ be a bump function such that
$\chi_{Q(1+2^{-n-11})} \le \xi \le \chi_{Q(1+2^{-n-10})}$.
Define
$\xi_k\equiv \xi\big(\frac{\cdot-x_k}{l_k}\big)$
and
$\eta_k\equiv \xi_k \div \sum\limits_{l \in K} \xi_l$.
Choose a polynomial $P_k \in {\mathcal P}_L({\mathbb R}^n)$ so that
$\langle f,Q \xi_k \rangle=\langle P_k,Q \xi_k \rangle$
for all $Q \in {\mathcal P}_L({\mathbb R}^n)$.
Set
$b_k\equiv (f-P_k)\eta_k$
for each $k$
and
$g\equiv f-\sum\limits_{k \in K} b_k$.
Thus,
$f=g+\sum\limits_{k \in K} b_k$.
We invoke the following estimate from \cite{Tang12},
which uses the maximal operators 
${\mathcal M}^0_N$
and
${\mathcal M}_N$.
\begin{lemma}{\rm \cite[Lemmas 4.3 and 4.5]{Tang12}}
\label{lem:190408-71}
Let $L \in [0,N)\cap{\mathbb Z}$.
Then
there exists $D>0$ such that, for all $k \in {\mathbb N}$,
\[
{\mathcal M}^0_N b_k
\lesssim
\chi_{Q_k^*}{\mathcal M}_N f+
\lambda
\chi_{(0,D)}(|Q_k|)
(M^{\rm loc}\chi_{Q_k})^{\frac{n+L+1}{n}}.
\]
\end{lemma}
Since
we have the vector-valued boundedness
of $M^{\rm loc}$
(see Lemma \ref{lem:190408-100}),
taking the norm
on both sides
gives
a norm estimate for the bad parts.
\begin{corollary}\label{cor:211025-3}
Let $w \in A^{\rm loc}_\infty$
and $p(\cdot) \in {\mathcal P}_0 \cap {\rm L H}_0 \cap {\rm L H}_\infty$.
Let $v \in (0,p_-)$ and $q>q_w$.
Let
$L,N \in {\mathbb Z}$
satisfy
\begin{equation}\label{eq:211027-1}
N \ge L \ge \left[n\left(\frac{q_w}{v}-1\right)\right].
\end{equation}
Then 
\begin{equation}\label{eq:211201-2}
\left\|
\left(\sum\limits_{k \in K}
({\mathcal M}^0_N b_k)^v
\right)^{\frac1v}\right\|_{L^{p(\cdot)}(w)}
\lesssim
\|{\mathcal M}_N f\|_{L^{p(\cdot)}(w)}.
\end{equation}
In particular,
if $v \le 1$,
$\sum\limits_{k \in K} b_k$
converges in $h^{p(\cdot)}(w)$. 
Hence 
it converges also in ${\mathcal D}'({\mathbb R}^n)$.
\end{corollary}

\begin{proof}
Using 
Lemma \ref{lem:190408-71}, we estimate
\begin{align*}
\lefteqn{
\left\|
\left(\sum\limits_{k \in K}
({\mathcal M}^0_N b_k)^v
\right)^{\frac1v}\right\|_{L^{p(\cdot)}(w)}
}\\
&\lesssim
\left\|
\left(\sum\limits_{k \in K}
\left(\chi_{Q_k^*}{\mathcal M}_N f+
\lambda
\chi_{(0,D)}(|Q_k|)
(M^{\rm loc}\chi_{Q_k})^{\frac{n+L+1}{n}}
\right)^v
\right)^{\frac1v}\right\|_{L^{p(\cdot)}(w)}.
\end{align*}
We remark that $p_->v$.
Using the constant sequence
$
\{\lambda\}_{k \in K}$,
the triangle inequality, 
the vector-valued inequality (Lemma \ref{lem:190408-100})
and the fact that 
$\{Q_k\}_{k\in K}$
is a Whitney decomposition of $\Omega_\lambda=\{x \in {\mathbb R}^n\,:\,{\mathcal M}_N f(x)>\lambda\}$, we have
\begin{align*}
\lefteqn{
\left\|
\left(\sum\limits_{k \in K}
({\mathcal M}^0_N b_k)^v
\right)^{\frac1v}\right\|_{L^{p(\cdot)}(w)}
}\\
&\lesssim
\left\|
\left(\sum\limits_{k \in K}
\left(\chi_{Q_k^*}{\mathcal M}_N f
\right)^v
\right)^{\frac1v}
\right\|_{L^{p(\cdot)(w)}}
+
{\mathcal A}_{p(\cdot),w,v}(
\{\lambda\}_{k \in K};
\{Q_k\}_{k \in K})\\
&\lesssim
\left\|{\mathcal M}_N f\right\|_{L^{p(\cdot)}(w)},
\end{align*}
proving (\ref{eq:211201-2}).

Now
assuming that $v \le 1$,
we can easily prove that
$\sum\limits_{k \in K} b_k$
converges in $h^{p(\cdot)}(w)$
since
\[
\left\|
\sum\limits_{k \in K}
{\mathcal M}^0_N b_k
\right\|_{L^{p(\cdot)}(w)}
\lesssim
\left\|
\left(\sum\limits_{k \in K}
({\mathcal M}^0_N b_k)^v
\right)^{\frac1v}\right\|_{L^{p(\cdot)}(w)}
\lesssim
\|{\mathcal M}_N f\|_{L^{p(\cdot)}(w)}.
\]
\end{proof}
Concerning condition (\ref{eq:211027-1}),
a helpful remark is in order.
\begin{remark}
Let $q \ge 1$ and $p_0>0$.
Due to the right-continuity of the function
$v \in(0,\infty)\mapsto
\left[n\left(\frac{q}{v}-1\right)\right] \in {\mathbb R}
$,
for an integer $L$,
there exists $v \in (0,p_0)$ such that
$L \ge \left[n\left(\frac{q}{v}-1\right)\right]_+$
if and only if
$L \ge \left[n\left(\frac{q}{p_0}-1\right)\right]_+$.
\end{remark}
Having guaranteed the convergence
of
$\sum\limits_{k \in K} b_k$
in ${\mathcal D}'({\mathbb R}^n)$,
we can employ another estimate by Tang,
which uses ${\mathcal M}_N^0$ and $M^{\rm loc}$ only.
\begin{lemma}{\rm \cite[Lemma 4.8]{Tang12}}
\label{lem:190408-72}
Let $L \in [0,N) \cap {\mathbb Z}$.
Then
\[
{\mathcal M}^0_N g
\lesssim
\chi_{{\mathbb R}^n \setminus \Omega_\lambda}{\mathcal M}^0_N f
+\lambda
\sum_{k \in K} 
(M^{\rm loc}\chi_{Q_k})^{\frac{n+L+1}{n}}.
\]
\end{lemma}

A direct consequence of Lemma \ref{lem:190408-72} is:
\begin{lemma}\label{cor:211025-2}
Let $w \in A^{\rm loc}_\infty$
and $p(\cdot) \in {\mathcal P}_0 \cap {\rm L H}_0 \cap {\rm L H}_\infty$.
If $L \in [0,N)$,
then $g \in L^{p_++q_w}(w)$.
\end{lemma}

\begin{proof}
Using Lemmas \ref{lem:190408-100} 
and 
\ref{lem:190408-72} 
as well as the fact that
$\{Q_k\}_{k\in K}$ is the Whitney decomposition
of $\Omega_\lambda$, 
we estimate
\begin{align*}
\|{\mathcal M}^0_N g\|_{L^{p_++q_w}(w)}
&\lesssim
\left\|
\chi_{{\mathbb R}^n \setminus \Omega_\lambda}{\mathcal M}^0_N f
+\lambda
\sum_{k \in K}
(M^{\rm loc}\chi_{Q_k})^{\frac{n+L+1}{n}}
\right\|_{L^{p_++q_w}(w)}\\
&\lesssim 
\left\|
\min\{{\mathcal M}^0_Nf,\lambda\}
\right\|_{L^{p_++q_w}(w)}.
\end{align*}
Note that
\[
\min\{{\mathcal M}_N^0f,\lambda\}
\in
L^{p(\cdot)}(w) \cap L^{\infty}({\mathbb R}^n)
\subset
L^{p_++q_w}(w).
\]
In fact,
with the implicit constant depending on $\lambda$,
\[
\int_{{\mathbb R}^n}
\min(\lambda,{\mathcal M}_N^0f(x))^{p_++q_w}{\rm d}x
\lesssim
\int_{{\mathbb R}^n}{\mathcal M}_N^0f(x)^{p(x)}{\rm d}x<\infty.
\]

Since
$
\min(\lambda,{\mathcal M}_N^0f) \in L^{p_++q_w}(w)$,
${\mathcal M}^0_N g \in L^{p_++q_w}(w)$.
Hence
$g \in h^{p_++q_w}(w)=L^{p_++q_w}(w)$
thanks to Theorem \ref{thm:210712-2}
and
the fact that $q_w \ge 1$ and $p_->0$.
\end{proof}
Next, let us look for a dense subspace
of $h^{p(\cdot)}(w)$, which consists of
regular distributions.
Specifically,
we are interested in distributions
in $h^{p(\cdot)}(w)$ which are realized
as a function in $L^u(w)$
for some $u \gg 1$.
A certain dense subspace
which is included in $L^0({\mathbb R}^n)$
is needed to consider the wavelet characterization
in Section \ref{s80} below.

Using
 the same argument as \cite[Lemma 4.9]{Tang12},
where Tang assumed $f \in L^1(w)$,
then we see that $g$ is essentially bounded if
$f \in L^{p_++q_w}(w)$.
We summarize this observation as follows:
\begin{lemma}{\rm \cite[Lemma 4.9]{Tang12}}\label{lem:211029-112}
Let $p(\cdot) \in {\mathcal P}_0 \cap {\rm L H}_0 \cap {\rm L H}_\infty$.
If $N \ge 2$ and $f \in L^{p_++q_w}(w)$,
then
$g \in L^{\infty}({\mathbb R}^n)$ and satisfies
$\|g\|_{L^{\infty}} \lesssim \lambda$.
\end{lemma}
\begin{lemma}
\label{lem:190408-74}
Let $w \in A^{\rm loc}_\infty$
and 
$p(\cdot) \in {\mathcal P}_0 \cap {\rm L H}_0 \cap {\rm L H}_\infty$.
Then the space 
$h^{p(\cdot)}(w) \cap L^{p_++q_w}(w) \cap L^\infty({\mathbb R}^n)$
is dense in
$h^{p(\cdot)}(w)$.
\end{lemma}

\begin{proof}
Let $L \in {\mathbb N}_0$
satisfy
$L \ge \left[n\left(\frac{q_w}{\min(1,p_-)}-1\right)\right]$.
Then
\[
p_-\frac{n+L+1}{n}
\ge
\frac{p_-}{n}\left(n+1+\left[n\left(\frac{q_w}{\min(1,p_-)}-1\right)\right]
\right)>q_w.
\]
Let $f \in h^{p(\cdot)}(w)$ and $\displaystyle g=f-\sum\limits_{k \in K} b_k$.
By the subadditivity of
${\mathcal M}^0_N$,
\begin{align*}
\|f-g\|_{h^{p(\cdot)}(w)}
=
\|{\mathcal M}^0_N[f-g]\|_{L^{p(\cdot)}(w)}
\lesssim 
\left\|
\sum\limits_{k \in K}
{\mathcal M}^0_N b_k
\right\|_{L^{p(\cdot)}(w)}.
\end{align*}
Recall that
$\{Q_k\}_{k\in K}$ is the Whitney decomposition.
Using Lemmas \ref{lem:190408-100} 
and 
\ref{lem:190408-71},
the triangle inequality and the H\"{o}lder inequality, 
we obtain
\begin{align*}
\|f-g\|_{h^{p(\cdot)}(w)}
&\lesssim 
\left\|
\sum\limits_{k \in K}
\chi_{Q_k^*}{\mathcal M}_N f+
\sum\limits_{k \in K}\lambda
\chi_{(0,D)}(|Q_k|)
(M^{\rm loc}\chi_{Q_k})^{\frac{n+L+1}{n}}
\right\|_{L^{p(\cdot)}(w)}\\
&\lesssim 
\left\|
\chi_{\Omega_\lambda}{\mathcal M}_N f+
\chi_{\Omega_\lambda}
\lambda
\right\|_{L^{p(\cdot)}(w)}\\
&\lesssim 
\left\|
\chi_{\Omega_\lambda}{\mathcal M}_N f
\right\|_{L^{p(\cdot)}(w)}.
\end{align*}
If we let $\lambda \uparrow \infty$, then we obtain 
$g \to f$ in $h^{p(\cdot)}(w)$.
Since
$g \in L^{p_++q_w}(w) \cap L^\infty({\mathbb R}^n) \cap h^{p(\cdot)}(w)$
thanks to Lemmas \ref{cor:211025-2} and \ref{lem:211029-112},
we obtain the desired result.
\end{proof}

As Tang noted,
if $w \in L^1({\mathbb R}^n)$,
there is a standard method to create
single $(p(\cdot),\infty)_w$-atoms.
\begin{lemma}{\rm \cite[Lemma 5.4]{Tang12}}
\label{lem:190408-73}
Let $p(\cdot) \in {\mathcal P}_0 \cap {\rm L H}_0 \cap {\rm L H}_\infty$.
Assume that $w({\mathbb R}^n)<\infty$ and that $q>\max(q_w,p_+)$.
Then there exists a constant $D_0>0$ with the following property{\rm:}
Suppose
that
$\lambda>0$
satisfies
$\lambda \le \inf\limits_{x \in {\mathbb R}^n}{\mathcal M}_Nf(x)<2\lambda$.
Then 
$a_0\equiv D_0\lambda^{-1}g$
is a single $(p(\cdot),\infty)_w$-atom.
\end{lemma}

With Lemmas \ref{lem:190408-71}--\ref{lem:190408-74}
in mind, we prove 
$h^{p(\cdot)}(w) \hookrightarrow h^{p(\cdot),\infty,L;v}(w) 
(\hookrightarrow h^{p(\cdot),q,L;v}(w))$.
We assume $w({\mathbb R}^n)<\infty$
and that
$2^{j_0} \le \inf\limits_{x \in {\mathbb R}^n}{\mathcal M}_Nf(x)<2^{j_0+1}$
for some $j_0 \in {\mathbb Z}$;
otherwise we can readily modify the argument below.
We use the above observation for $\lambda=2^j$,
where $j$ ranges over $[j_0,\infty) \cap {\mathbb Z}$.
We will add a subindex $j$
to what we have obtained
to indicate that it comes from
$\Omega_{2^j}$.
Thus,
we  obtain
cubes
$\{Q_{j,k}\}_{k \in K_j}$,
smooth functions
$\{\eta_{j,k}\}_{k \in K_j}$
and
polynomials
$\{P_{j,k}\}_{k \in K_j}$.
Then we have a decomposition
$f=g_j+b_j$,
where
$b_{j,k}\equiv (f-P_{j,k})\eta_{j,k}$
and
$b_j\equiv \sum\limits_{k \in K_j}b_{j,k}$.
We write
$\tilde{Q}^j_k\equiv (1+2^{-n-12})Q^j_k$.
We use the following observation:
\begin{lemma}\label{lem:211129-1}
Under the assumption above,
we have
\begin{equation}\label{eq:211129-1}
{\mathcal A}_{p(\cdot),w,v}(
\{\lambda^j_k\}_{j\in[j_0,\infty) \cap {\mathbb Z}, k \in K_j};
\{\tilde{Q}^j_k\}_{j\in[j_0,\infty) \cap {\mathbb Z}, k \in K_j})
\lesssim
\|{\mathcal M}_N f\|_{L^{p(\cdot)}(w)}.
\end{equation}
\end{lemma}

\begin{proof}
Due to the bounded overlapping property,
\begin{align}
\nonumber
\lefteqn{
{\mathcal A}_{p(\cdot),w,v}(
\{\lambda^j_k\}_{j\in[j_0,\infty) \cap {\mathbb Z}, k \in K_j};
\{\tilde{Q}^j_k\}_{j\in[j_0,\infty) \cap {\mathbb Z}, k \in K_j})
}\\
\nonumber
&=
{\mathcal A}_{p(\cdot),w,v}(
\{2^j\}_{j\in[j_0,\infty) \cap {\mathbb Z}, k \in K_j};
\{\tilde{Q}^j_k\}_{j\in[j_0,\infty) \cap {\mathbb Z}, k \in K_j})\\
\label{eq:211201-3}
&\lesssim
\left\|
\left(
\sum_{j=j_0}^{\infty} 
2^{j v} \chi_{\Omega_{2^j}}
\right)^{\frac{1}{v}}
\right\|_{L^{p(\cdot)}(w)}.
\end{align}
Let $x \in {\mathbb R}^n$.
Then we have
\[
\sum_{j=j_0}^{\infty} 2^{j v}\chi_{\Omega_{2^j}}(x)
=
\sum_{j \in {\mathbb Z} \cap [j_0,\infty) \cap (-\infty,\log_2{\mathcal M}_N f(x))}
2^{j v}
\le
\sum_{j \in {\mathbb Z} \cap (-\infty,\log_2{\mathcal M}_N f(x))}
2^{j v}
\lesssim
{\mathcal M}_N f(x)^{v}.
\]
Thus,
(\ref{eq:211129-1}) follows
from
(\ref{eq:211201-3}).
\end{proof}
Let us return to the proof of
$h^{p(\cdot)}(w) \hookrightarrow h^{p(\cdot),\infty,L,v}(w)$.
We follow the idea in \cite{Stein-text-93},
which allows us to assume
$f \in h^{p(\cdot)}(w) \cap L^{p_++q_w}(w) 
\cap L^\infty({\mathbb R}^n)$.
Also,
assume that
$f \in h^{p(\cdot)}(w) \cap L^{p_++q_w}(w)
 \cap L^\infty({\mathbb R}^n)$
keeping Lemma \ref{lem:190408-74} in mind.

Let $D_0>0$ be a constant
from Lemma \ref{lem:190408-73}.
Using the same argument
as \cite[Lemma 5.4]{Tang12},
we have a decomposition
$f=g_{j_0}+\sum\limits_{j=j_0}^{\infty}\sum\limits_{k \in K_j} \lambda_{j,k} a_{j,k}$,
where 
$D_02^{-j_0}g_{j_0}$ is a single $(p(\cdot),\infty)_w$-atom,
each
$a_{j,k}$ is a $(p(\cdot),\infty,L)_w$-atom supported
on a cube $\tilde{Q}^j_k=(1+2^{-n-12})Q^j_k$,
and
$\lambda^j_k=2^j$.
We set
\[
K_j^-\equiv
\{k \in K_j\,:\,|\tilde{Q}^j_k|<1\}, \quad
K_j^+\equiv K_j \setminus K_j^-.
\]We write
\[
X\equiv\{(j,k,l)\,:\,j \in [j_0,\infty) \cap {\mathbb Z},
k \in K_j^+,
\tilde{Q}^j_k \cap (l+[0,1]^n) \ne \emptyset\}.
\]

We further decompose
\begin{align*}
f
&=g_{j_0}
+\sum\limits_{j=j_0}^{\infty}\sum\limits_{k \in K_j^+} \lambda_{j,k} a_{j,k}
+\sum\limits_{j=j_0}^{\infty}\sum\limits_{k \in K_j^-} \lambda_{j,k} a_{j,k}\\
&=g_{j_0}
+\sum_{(j,k,l) \in X}\lambda_{j,k}\chi_{l+[0,1]^n}a_{j,k}
+\sum\limits_{j=j_0}^{\infty}\sum\limits_{k \in K_j^-} \lambda_{j,k} a_{j,k}.
\end{align*}
We remark that
each $\chi_{l+[0,1]^n}a_{j,k}$
is a $(p(\cdot),\infty,L)_w$-atom
supported on a cube $l+[0,1]^n$
as long as $k \in K_j^+$ and $l \in {\mathbb Z}^n$
satisfies $\tilde{Q}^j_k \cap (l+[0,1]^n) \ne \emptyset$.

Let us prove the norm estimate.
If $(j,k,l) \in X$,
then
$l+[0,1]^n \subset 3\tilde{Q}^j_k$.
Since $\chi_{3\tilde{Q}^j_k} 
\lesssim M^{\rm loc}[\chi_{\tilde{Q}^j_k}]^{\frac{1}{v-\varepsilon}}$
for some $\varepsilon \in (0, v)$, 
we deduce
from the vector-valued inequality (Lemma \ref{lem:190408-100}),
\begin{align*}
\left\|
\left(\sum_{(j,k,l) \in X}
(\lambda_{j,k}\chi_{l+[0,1]^n})^v
\right)^{\frac{1}{v}}
\right\|_{L^{p(\cdot)}(w)}
&\lesssim
\left\|
\left(\sum_{(j,k,l) \in X}
(\lambda_{j,k}M^{\rm loc}[\chi_{\tilde{Q}^j_k}]^{\frac{1}{v-\varepsilon}})^v
\right)^{\frac{1}{v}}
\right\|_{L^{p(\cdot)}(w)}\\
&\lesssim
{\mathcal A}_{p(\cdot),w,v}(
\{\lambda^j_k\}_{j\in[j_0,\infty) \cap {\mathbb Z}, k \in K_j^+};
\{\tilde{Q}^j_k\}_{j\in[j_0,\infty) \cap {\mathbb Z}, k \in K_j^+}).
\end{align*}
Thus by Lemma \ref{lem:211129-1}, we have
\begin{align}\label{eq:211129-3}
\left\|
\left(\sum_{(j,k,l) \in X}
(\lambda_{j,k}\chi_{l+[0,1]^n})^v
\right)^{\frac{1}{v}}
\right\|_{L^{p(\cdot)}(w)}
&\lesssim
{\mathcal A}_{p(\cdot),w,v}(
\{\lambda^j_k\}_{j\in[j_0,\infty) \cap {\mathbb Z}, k \in K_j};
\{\tilde{Q}^j_k\}_{j\in[j_0,\infty) \cap {\mathbb Z}, k \in K_j}) \nonumber\\
&\lesssim
\|{\mathcal M}_N f\|_{L^{p(\cdot)}(w)}.
\end{align}
Meanwhile,
from
Lemma \ref{lem:211129-1},
we obtain
\begin{align}\label{eq:211129-2}
\lefteqn{
{\mathcal A}_{p(\cdot),w,v}(
\{\lambda^j_k\}_{j\in[j_0,\infty) \cap {\mathbb Z}, k \in K_j^-};
\{\tilde{Q}^j_k\}_{j\in[j_0,\infty) \cap {\mathbb Z}, k \in K_j^-})
}\\
&\nonumber
\lesssim
{\mathcal A}_{p(\cdot),w,v}(
\{\lambda^j_k\}_{j\in[j_0,\infty) \cap {\mathbb Z}, k \in K_j};
\{\tilde{Q}^j_k\}_{j\in[j_0,\infty) \cap {\mathbb Z}, k \in K_j})\\
&\nonumber
\lesssim
\|{\mathcal M}_N f\|_{L^{p(\cdot)}(w)}.
\end{align}
Combining 
(\ref{eq:211129-3})
and
(\ref{eq:211129-2}),
we obtain the desired norm estimate
$\|f\|_{h^{p(\cdot),\infty,L,v}(w)} \lesssim \|f\|_{h^{p(\cdot)}(w)}$.

\section{Applications 
to singular integral operators}
\label{s60}

Now that the structure of the weighted local Hardy space
$h^{p(\cdot)}(w)$ is clarified,
we present some applications.
In Section \ref{s60.1},
we establish that generalized local singular integral operators
considered in \cite{NS-local} are bounded
from
$h^{p(\cdot)}(w)$
to
$L^{p(\cdot)}(w)$ as long as $p_->\frac{n}{n+1}$.
Section \ref{s60.2} is devoted to a special case 
of Section \ref{s60.1}.
We are interested in removing the condition
$p_->\frac{n}{n+1}$ by considering a narrower but important class
of operators.
Among the generalized local singular integral operators,
we consider the convolution operators generated by compactly supported
smooth functions.

We make a brief remark on the method of the proof
used in Section \ref{s60}.
There are several ways to prove
the boundedness of singular integral operators
from Hardy spaces to Hardy spaces.
Fefferman and Stein investigated the boundedness
of singular integral operators
by investigating the distribution function
of the image by singular integral operator
\cite[Lemma 11 and Theorem 12]{FeSt72}.
We can not employ the method in \cite{FeSt72}
since we are considering function spaces
which is not rearrangement invariant.
Our method is to use the atomic decomposition
as 
Garc\'{i}a-Cuerva and Rubio de Francia
\cite{GaRu85}
and
Tang \cite{Tang12} did.
Garc\'{i}a-Cuerva and Rubio de Francia
also considered the boundedness property
of singular integral operators \cite[Theorems 7.8 and 7.9]{GaRu85}.
They used the atomic decomposition.
We can say that our method is akin to theirs.
See the proof of Theorem \ref{thm:6.4}.
We also remark that Tang took the same strategy,
where he also analyzed the image of atoms \cite[Theorem 7.1]{Tang12}.
What is different from \cite{GaRu85,Tang12} is that
we must take care of the position of the cubes on which
atoms are supported
by using Lemma \ref{lem:190408-100}.
This approach is taken in \cite{NaSa12}.
However, since we need to consider the local grand maximal operators,
we can not employ the estimate directly.
What we do is to adjust what we did in \cite[\S 5.1]{NaSa12}.

\subsection{Generalized local singular integral operators}
\label{s60.1}

An $L^2$-bounded linear operator $T$
is called a
generalized local Calder\'{o}n--Zygmund operator
$($with the kernel $K$$)$,
if $T$ satisfies the following conditions:
\index{generalized Calderon--Zygmund operator@(generalized) Calder\'{o}n--Zygmund operator}
\begin{enumerate}
\item[$(1)$]
There exists
$K \in L^{1}_{\rm loc}({\mathbb R}^n \times {\mathbb R}^n \setminus \{(x,x)\,:\,x \in {\mathbb R}^n\})$
such that,
for all $f \in L^2_{\rm c}({\mathbb R}^n)$, 
\begin{equation}\label{eq:sing1}
T f(x)=\int_{{\mathbb R}^n}K(x,y)f(y){\rm d}y
\mbox{ for almost all } x \notin {\rm supp}(f).
\end{equation}
\item[$(2)$]
There exist positive
constants $\gamma_0$, $D_{1}=D_{1}(T)$ and $D_{2}=D_{2}(T)$ such that
the two conditions below hold
for all $x,y,z \in {\mathbb R}^n$:
\begin{enumerate}
\item[$(i)$]
Local size condition:
\index{size condition@size condition}
\begin{align}
\label{eq:sing3}
|K(x,y)| \le D_{1}|x-y|^{-n}\chi_{[-\gamma_0,\gamma_0]^n}(x-y)
\end{align}
if $x \ne y$.
\item[$(ii)$]
H\"{o}rmander's condition:
\index{Hormander's condition@H\"{o}rmander's condition}
\begin{align}
\label{eq:sing2}
|K(x,z)-K(y,z)|
+
|K(z,x)-K(z,y)|
\le D_{2}\frac{|x-y|}{|x-z|^{n+1}}
\end{align}
if $0<2|x-y|<|z-x|$.
\end{enumerate}
\end{enumerate}
This is analogous to
the generalized singular integral operators
dealt with in \cite{Du},
which require
\begin{align}
\label{eq:sing3a}
|K(x,y)| \le D_{1}|x-y|^{-n}
\end{align}
instead of
(\ref{eq:sing3})
if $x \ne y$.
\cite{INNS-wavelet} shows
that all generalized local singular integral operators
initially defined on $L^2({\mathbb R}^n)$ can be extended to a bounded
linear operator on $L^{p(\cdot)}(w)$ 
for 
any
$p(\cdot) \in {\mathcal P}
\cap {\rm L H}_0 \cap {\rm L H}_{\infty}$ and
$w \in A_{p(\cdot)}^{\rm loc}$.
Recall that such generalized local singular integral operators
are bounded on $L^{p(\cdot)}(w)$.
\begin{proposition}\label{prop:200224-1}
Suppose that $p(\cdot) \in {\mathcal P}
\cap {\rm L H}_0 \cap {\rm L H}_{\infty}$.
Let $T$ be a generalized local singular integral operator
and $w \in A_{p(\cdot)}^{\rm loc}$. 
Then
$T$ extends to a bounded linear operator
on $L^{p(\cdot)}(w)$ with the norm estimate
\[
\|T\|_{L^{p(\cdot)}(w) \to L^{p(\cdot)}(w)}
\lesssim
\|T\|_{L^2 \to L^2}
+
D_1(T)+D_2(T).
\]
\end{proposition}

Proposition \ref{prop:200224-1}
was proved using the local sharp maximal operator
considered in \cite{Lerner13-2}.
We extend 
Proposition \ref{prop:200224-1} to weighted local Hardy spaces
with variable exponents and
investigate how 
generalized local singular integral operators
act on atoms.
If we reexamine the proof
of \cite[(5.2)]{NaSa12},
then we see that the following pointwise estimate holds.
\begin{lemma} \label{lem:210921-1}
Suppose that $p(\cdot) \in {\mathcal P}_0
\cap {\rm L H}_0 \cap {\rm L H}_{\infty}$.
Let $T$ be a generalized local singular integral operator.
Then 
any $(p(\cdot),\infty,1)_w$-atom $a$
 supported on $Q$ satisfies
\[
|T a(x)| \lesssim |T a(x)|\chi_{3Q}(x)+D_2(T)M^{\rm loc}\chi_Q(x)^{\frac{n+1}{n}}
\]
for all $x \in {\mathbb R}^n$.
\end{lemma}

\begin{proof}
We must consider two cases:
$x \in 3Q$,
$x \in R \setminus Q$,
where $R$ is the cube of volume $(2+2\gamma_0)^n$
concentric to $Q$.
For the first case, there is nothing to prove.
We use the 
H\"{o}rmander's condition for the second case.
\end{proof}

A direct consequence of
Lemma \ref{lem:210921-1}
is that generalized local singular integral operators
are bounded from
$h^{p(\cdot)}(w)$
to
$L^{p(\cdot)}(w)$
as long as
$p_->\frac{n}{n+1}$,
extending
Proposition \ref{prop:200224-1}
in terms of $h^{p(\cdot)}(w)$ considered in this paper.

\begin{theorem}\label{thm:6.3}
Suppose that $p(\cdot) \in {\mathcal P}_0
\cap {\rm L H}_0 \cap {\rm L H}_{\infty}$
satisfies $p_->\frac{n}{n+1}$.
Let $T$ be a generalized local singular integral operator,
and let $w \in A_{\infty}^{\rm loc}$. 
Then
$T$ is bounded from $h^{p(\cdot)}(w)$
to $L^{p(\cdot)}(w)$ with the norm estimate
\[
\|T\|_{h^{p(\cdot)}(w) \to L^{p(\cdot)}(w)}
\lesssim
\|T\|_{L^2 \to L^2}
+
D_1(T)+D_2(T).
\]
\end{theorem}

\begin{proof}
We assume that
$w \in L^1({\mathbb R}^n)$:
If
$w \notin L^1({\mathbb R}^n)$,
then we can modify the proof below.
It suffices to show that
\[
\|T f\|_{L^{p(\cdot)}(w)}
\lesssim
(
\|T\|_{L^2 \to L^2}
+
D_1(T)+D_2(T)
)
\|f\|_{h^{p(\cdot)}(w)}
\]
for all
$f \in h^{p(\cdot)}(w) \cap L^{p_++q_w}(w)
 \cap L^\infty({\mathbb R}^n)$
thanks to Lemma \ref{lem:190408-74}.
Let
$q$ satisfy $p_++q_w+1<q<\infty$ and $L \gg 1$.
Let
$f \in h^{p(\cdot)}(w) \cap L^{p_++q_w}(w)
 \cap L^\infty({\mathbb R}^n)$.
Due to Theorem \ref{thm:190406-2},
there exist
$\{a_j\}_{j=0}^\infty \subset L^0({\mathbb R}^n)$,
$\{b_j\}_{j=1}^\infty \subset L^0({\mathbb R}^n)$,
$\{\lambda_j\}_{j=0}^{\infty} \subset [0, \infty)$
and
$\{\lambda_j'\}_{j=1}^{\infty} \subset [0, \infty)$
such that
$a_0$ is
a single $(p(\cdot),q)_w$-atom,
that
each
$a_j$, $j \in {\mathbb N}$
is
a $(p(\cdot),q,L)_w$-atom supported on a cube $Q_j$ with $|Q_j|<1$,
that
each
$b_j$, $j \in {\mathbb N}$
is
a $(p(\cdot),q,L)_w$-atom supported on a cube $R_j$ with $|R_j|=1$,
that
$f=\sum\limits_{j=0}^{\infty} \lambda_j a_j+
\sum\limits_{j=1}^{\infty} \lambda_j' b_j$
holds
in the topology of
$h^{p(\cdot)}(w)$ 
and that
\[
|\lambda_0|+
{\mathcal A}_{p(\cdot),w,v}(\{\lambda_j\}_{j=1}^{\infty};\{Q_j\}_{j=1}^{\infty})+
{\mathcal A}_{p(\cdot),w,v}(\{\lambda_j'\}_{j=1}^{\infty};\{R_j\}_{j=1}^{\infty})
\lesssim
\|f\|_{h^{p(\cdot)}(w)}.
\]
Again using Theorem \ref{thm:190406-2}, we have
$f=\sum\limits_{j=0}^{\infty} \lambda_j a_j+
\sum\limits_{j=1}^{\infty} \lambda_j' b_j$
holds
in the topology of
$h^{p_++q_w}(w)$, especially 
in the topology of $L^{p_++q_w}(w)$ by Theorem \ref{thm:210712-2}.

We know that
$T$ maps
$L^{p_++q_w}(w)$
continuously to itself by Proposition \ref{prop:200224-1}.
Thus,
$T f=\sum\limits_{j=0}^{\infty} \lambda_j T a_j+
\sum\limits_{j=1}^{\infty} \lambda_j' T b_j$
holds
in the topology of
$L^{p_++q_w}(w)$.
Due to Lemma \ref{lem:210921-1} and the triangle inequality,
we have
\begin{align}
\nonumber
\|T f\|_{L^{p(\cdot)}(w)}
&\lesssim
\left\|
\sum_{j=1}^\infty (\lambda_j |T a_j|\chi_{3Q_j}+D_2(T)(M^{\rm loc}\chi_{Q_j})^{\frac{n+1}{n}})
\right\|_{L^{p(\cdot)}(w)}\\
&\nonumber 
\qquad +
\left\|
\sum_{j=1}^\infty
\lambda'_j |T b_j|
\right\|_{L^{p(\cdot)}(w)}
+|\lambda_0|\|T a_0\|_{L^{p(\cdot)}(w)}\\
\label{eq:211031-2}
&\lesssim
\left\|
\sum_{j=1}^\infty \lambda_j|T a_j|\chi_{3Q_j}
\right\|_{L^{p(\cdot)}(w)}+
\left\|
\sum_{j=1}^\infty \lambda_j'|T b_j|
\right\|_{L^{p(\cdot)}(w)}\\
\nonumber
&\qquad+D_2(T)
\left\|
\sum_{j=1}^\infty \lambda_j(M^{\rm loc}\chi_{Q_j})^{\frac{n+1}{n}}
\right\|_{L^{p(\cdot)}(w)}+|\lambda_0|\|T a_0\|_{L^{p(\cdot)}(w)}.
\end{align}
We choose $u,v>0$ and $L \in {\mathbb Z}$ so that
\[
v<\min(1,p_-), \quad
L>\left[n\left(\frac{q_w}{v}-1\right)\right], \quad
\max(q_w, p_+)<u v<q.
\]

For the first term we use
Theorem \ref{lem:210805-1} to give
\[
\left\|
\sum_{j=1}^\infty \lambda_j|T a_j|\chi_{3Q_j}
\right\|_{L^{p(\cdot)}(w)}
\lesssim
{\mathcal A}_{p(\cdot),w,v}(\{\lambda_j m_{3Q_j,w}^{(u v)}(T a_j)\}_{j=1}^{\infty};\{3Q_j\}_{j=1}^{\infty}).
\]
Recall that $u v>q_w$ 
and that $w \in \bigcup\limits_{\tilde{q}>q_w}A_{\tilde{q}}^{\rm loc}$.
Since $T$ is bounded
on $L^{u v}(w)$ (see Proposition \ref{prop:200224-1}),
\begin{eqnarray*}
&&
{\mathcal A}_{p(\cdot),w,v}(\{\lambda_j m_{3Q_j,w}^{(u v)}(T a_j)\}_{j=1}^{\infty};\{3Q_j\}_{j=1}^{\infty})\\
&&\lesssim
(\|T\|_{L^2 \to L^2}
+
D_1(T)+D_2(T))
{\mathcal A}_{p(\cdot),w,v}(\{\lambda_j\}_{j=1}^{\infty};\{3Q_j\}_{j=1}^{\infty}).
\end{eqnarray*}
Using
$\chi_{3Q_j}\lesssim M^{\rm loc}\chi_{Q_j}$
and the vector-valued inequality (Lemma \ref{lem:190408-100}), we have
\begin{equation}\label{eq:211129-201}
|\lambda_0|+
{\mathcal A}_{p(\cdot),w,v}(\{\lambda_j\}_{j=1}^{\infty};\{3Q_j\}_{j=1}^{\infty})
\lesssim
|\lambda_0|+
{\mathcal A}_{p(\cdot),w,v}
(\{\lambda_j\}_{j=1}^{\infty};\{Q_j\}_{j=1}^{\infty})
\lesssim
\|f\|_{h^{p(\cdot)}(w)}.
\end{equation}
Thus, the estimate for the first term
of (\ref{eq:211031-2})
 is valid.

The second term
of (\ref{eq:211031-2}) 
can be handled similarly to the first term.
The result is
\begin{equation}\label{eq:211129-202}
\left\|
\sum_{j=1}^\infty \lambda_j'|T b_j|
\right\|_{L^{p(\cdot)}(w)}
\lesssim
{\mathcal A}_{p(\cdot),w,v}(\{\lambda_j'\}_{j=1}^{\infty};\{R_j\}_{j=1}^{\infty})
\lesssim
\|f\|_{h^{p(\cdot)}(w)}.
\end{equation}

The third term
of (\ref{eq:211031-2})
 is easy to deal with.
As before,
by
the condition $0<v \le 1$ and
the vector-valued inequality (Lemma \ref{lem:190408-100})
\begin{align}\label{eq:211129-203}
\left\|
\sum_{j=1}^\infty \lambda_j(M^{\rm loc}\chi_{Q_j})^{\frac{n+1}{n}}
\right\|_{L^{p(\cdot)}(w)}
&\lesssim
{\mathcal A}_{p(\cdot),w,v}
(\{\lambda_j\}_{j=1}^{\infty};\{Q_j\}_{j=1}^{\infty})\\
\nonumber
&\le
{\mathcal A}_{p(\cdot),w,v}
(\{\lambda_j\}_{j=1}^{\infty};\{Q_j\}_{j=1}^{\infty})
\nonumber
\lesssim
\|f\|_{h^{p(\cdot)}(w)}.
\end{align}
Combining
(\ref{eq:211129-201}),
(\ref{eq:211129-202})
and
(\ref{eq:211129-203})
with $\|T a_0\|_{L^{p(\cdot)}(w)} \lesssim 1$,
we obtain the desired result.
\end{proof}
\subsection{Singular integral operators of the convolution type}
\label{s60.2}
Theorem \ref{thm:6.3}
estimates the integral kernel $K$
only up to order $1$.
Here we consider the case where
the kernel is smoother.
To avoid the bothersome argument
of justifying the definition of $T f=k*f$
for $f \in h^{p(\cdot)}(w)$,
we assume that $k \in C^\infty_{\rm c}({\mathbb R}^n)$.
Nevertheless, this assumption 
can be removed by a routine
limiting argument, which we omit.
Here we establish the following theorem.
\begin{theorem}\label{thm:6.4}
Let $p(\cdot) \in {\mathcal P}_0
\cap {\rm L H}_0 \cap {\rm L H}_{\infty}$
and $w \in A_{\infty}^{\rm loc}$. 
Let $\{B_m\}_{m=0}^\infty$ be a positive sequence and $\gamma_0>0$.
Let $k \in \mathcal{S}({\mathbb R}^n)$
satisfy
\[|x|^{n+m}|\nabla^m k(x)|\le B_m\chi_{[-\gamma_0,\gamma_0]^n}(x)
\quad (x \in {\mathbb R}^n, m \in {\mathbb N}_0).
\]
Define a convolution operator $T$ by
\[
T f \equiv k*f \quad (f \in L^2({\mathbb R}^n)).
\]
Then $T$ is
an $h^{p(\cdot)}(w)$-bounded operator
and the norm depends only on $\|\mathcal{F}k\|_{L^\infty}$
and a finite number of collections $B_0,B_1,\ldots,B_N$
with $N \in {\mathbb N}$ depending only on $p(\cdot)$.
\end{theorem}As we did in \cite[\S 2.5.8]{Triebel-text-83} and \cite[\S 5.3]{NaSa12},
the boundedness property of singular integral operators
is useful for the Littlewood--Paley characterization.
It matters that the estimate does not depend
on $\|k\|_{L^1}$.
This is absolutely necessary in Proposition \ref{prop:210923-11}.

The proof of Theorem \ref{thm:6.4} uses the following observations:

\begin{lemma}\label{lem:211029-113}
Let $p(\cdot) \in {\mathcal P}_0
\cap {\rm L H}_0 \cap {\rm L H}_{\infty}$
and $w \in A_{\infty}^{\rm loc}$. 
Also let $L \in {\mathbb N}$
and $T$ be the bounded linear operator on $L^2({\mathbb R}^n)$
as in Theorem \ref{thm:6.4}.
Then any $(p(\cdot),\infty,L)_w$-atom
$a$ supported on $Q$ satisfies
$T a \in {\mathcal P}_{L}^\perp({\mathbb R}^n)$
and
\[
|T a(x)| \lesssim |T a(x)|\chi_{3Q}(x)+\left(\sum_{j=0}^LB_j\right)M^{\rm loc}\chi_Q(x)^{\frac{n+L+1}{n}}
\quad (x \in {\mathbb R}^n).
\]
\end{lemma}
\begin{proof}
It can be easily verified that $T a \in {\mathcal P}_{L}^\perp({\mathbb R}^n)$ 
using the moment condition of $a$. 
Due to \cite[Propositions 5.3 and 5.4]{NaSa12},
we have
\[
|T a(x)| \lesssim |T a(x)|\chi_{3Q}(x)+\left(\sum_{j=0}^L B_j\right)M
\chi_Q(x)^{\frac{n+L+1}{n}}
\]
for some $L \in {\mathbb N}$ depending only on $p(\cdot)$.
Since $a$ is supported on a cube with $|Q| \le 1$ and ${\rm supp}(k) \subset [-\gamma_0,\gamma_0]^n$,
we see that $T f$ is supported on a cube with 
a
volume less than or equal to $(2+2\gamma_0)^n$.
Thus, we can replace the maximal operator by $M^{\rm loc}$.
\end{proof}

\begin{lemma}\label{lem:220216-1}
Let $p(\cdot) \in {\mathcal P}_0
\cap {\rm L H}_0 \cap {\rm L H}_{\infty}$
and $w \in A_{\infty}^{\rm loc}$. 
Also let $L \in {\mathbb N}$
and $T$ be the bounded linear operator on $L^2({\mathbb R}^n)$
as in Theorem \ref{thm:6.4}.
Assume that
$a$ is a $(p(\cdot),\infty,L)_w$-atom supported on a cube $Q$
with $|Q| \le 1$.
Then
\[
|T a|\chi_{{\mathbb R}^n \setminus 5Q}
\lesssim
(M^{\rm loc}\chi_Q)^{\frac{n+L+1}{n}}
\]
\end{lemma}

\begin{proof}Let
$x \in {\mathbb R}^n \setminus 5Q$.
Then
\[
T a(x)=
\int_{{\mathbb R}^n}
\left(
k(x-y)-\sum_{\alpha \in {\mathbb N}_0{}^n, |\alpha| \le L}
\frac{\partial^\alpha k(x-x_0)}{\alpha!}(x_0-y)^\alpha
\right)a(y)dy.
\]
By the mean value theorem,
there exist
$\theta \in(0,1)$,
which depends on
$x,y,x_0,L$ such that
\begin{align*}
\lefteqn{
k(x-y)-\sum_{\alpha \in {\mathbb N}_0{}^n, |\alpha| \le L}
\frac{\partial^\alpha k(x-x_0)}{\alpha!}(x_0-y)^\alpha
}\\
&=
\sum_{\alpha \in {\mathbb N}_0{}^n, |\alpha| = L+1}
\frac{\partial^\alpha k(x-x_0+\theta(x_0-y))}{\alpha!}(x_0-y)^\alpha.
\end{align*}
Hence
\begin{align*}
\lefteqn{
|x-y|^{n+L+1}
\left|
k(x-y)-\sum_{\alpha \in {\mathbb N}_0{}^n, |\alpha| \le L}
\frac{\partial^\alpha k(x-x_0)}{\alpha!}(x_0-y)^\alpha
\right|}\\
&\quad
\lesssim
\ell(Q)^{L+1}
\sup_{\alpha \in {\mathbb N}_0{}^n, |\alpha| = L+1}
|z|^{n+L+1}|\partial^\alpha k(z)|
\end{align*}
for some $z \in {\mathbb R}^n$.
Thus, we obtain
\begin{align*}
|Ta(x)|
\lesssim
\ell(Q)^{n+L+1}
|x-c(Q)|^{-n-L-1}
\end{align*}
Since $T a$ is supported on $Q(x_0,2+2\gamma_0)$,
we have the desired result.
\end{proof}

Fix $L \in {\mathbb N}$,
a cube
$Q$ and $f \in L^1_{\rm loc}({\mathbb R}^n)$.
We define
$P_{L,Q}f$
to be the unique polynomial of order $L$
such that
\[
\int_Q x^\beta(f(x)-P_{L,Q}f(x)){\rm d}x=0
\]
for all $\beta \in {\mathbb N}_0$ with $|\beta| \le L$.
If $Q=a+[0,r]^n$ for some $a \in {\mathbb R}^n$
and $r>0$,
then
\[
P_{L,Q}f=P_{L,[0,1]^n}[f(a+r\cdot)]\left(\frac{\cdot-a}{r}\right).
\]
Thus,
we have
\begin{equation}\label{eq:220224-111}
\|P_{L,Q}f\|_{L^\infty(Q)}
\lesssim
|Q|^{-\frac12}\|f\|_{L^2(Q)}.
\end{equation}

Similar to Lemma \ref{lem:190408-1},
we have the following pointwise estimate:
\begin{lemma}\label{lem:220209-1}
Let $p(\cdot) \in {\mathcal P}_0
\cap {\rm L H}_0 \cap {\rm L H}_{\infty}$
and $w \in A_{\infty}^{\rm loc}$. 
Also let $L \in {\mathbb N}$
and $T$ be the bounded linear operator on $L^2({\mathbb R}^n)$
as in Theorem \ref{thm:6.4}.
Let
$a$ be a $(p(\cdot),\infty,2L+2n+2)_w$-atom
supported on $Q$.
Under the assumption of  Theorem \ref{thm:6.4},
\[
{\mathcal M}^0_{N_{p(\cdot),w}}[k*a](x) \lesssim
{\mathcal M}^0_{N_{p(\cdot),w}}[\chi_{5Q}(k*a-P_{2L+2n+2,5Q}[k*a])](x)+ M^{\rm loc}\chi_Q(x)^{\frac{n+L+1}{n}}
\]
for all $x \in {\mathbb R}^n \setminus 5Q$.
\end{lemma}
\begin{proof}
Denote by $x_0$ the center of $Q$.
Let $t\in(0,1)$ be fixed.
Let $\varphi \in {\mathcal D}_{N_{p(\cdot),w}}^0({\mathbb R}^n)$.

We have
\[
k*a=\chi_{5Q}(k*a-P_{2L+2n+2,5Q}[k*a])
+\chi_{5Q}P_{2L+2n+2,5Q}[k*a]+\chi_{{\mathbb R}^n \setminus 5Q}k*a.
\]
Therefore, it suffices to show that
\[
{\mathcal M}^0_{N_{p(\cdot),w}}[\chi_{5Q}P_{2L+2n+2,5Q}[k*a]+\chi_{{\mathbb R}^n \setminus 5Q}k*a](x) \lesssim
 M^{\rm loc}\chi_Q(x)^{\frac{n+L+1}{n}}.
\]
Note that
\[
\chi_{5Q}P_{2L+2n+2,5Q}[k*a]+\chi_{{\mathbb R}^n \setminus 5Q}k*a
\in {\mathcal P}_{2L+2n+2}^\perp({\mathbb R}^n).
\]
By using Lemma \ref{lem3sw} twice 
(for the case $t \le \ell(Q)$ and $t \ge \ell(Q)$)
and Lemmas \ref{lem:220216-1} and (\ref{eq:220224-111}),
we have
\begin{align*}
\lefteqn{
|\varphi_t*(
\chi_{5Q}P_{2L+2n+2,5Q}[k*a]+\chi_{{\mathbb R}^n \setminus 5Q}k*a)
(x)|
}\\
&\lesssim
\min\left(1,\frac{\ell(Q)}{t}\right)^{2L+2n+3}
\frac{\ell(Q)^n\max(t,\ell(Q))^{-n}}{1+\max(t,\ell(Q))^{-L-n-1}|x-x_0|^{n+L+1}}\\
&=
\min\left(1,\frac{\ell(Q)}{t}\right)^{2L+2n+3}
\frac{\ell(Q)^n\max(t,\ell(Q))^{-n}}{1+\min(1,t^{-1}\ell(Q))^{n+L+1}\ell(Q)^{-n-L-1}|x-x_0|^{n+L+1}}\\
&\le
\min\left(1,\frac{\ell(Q)}{t}\right)^{2L+2n+3}
\frac{\ell(Q)^n\max(t,\ell(Q))^{-n}}{\min(1,t^{-1}\ell(Q))^{n+L+1}(1+\ell(Q)^{-n-L-1}|x-x_0|^{n+L+1})}\\
&=
\min\left(1,\frac{\ell(Q)}{t}\right)^{L+n+2}
\frac{\ell(Q)^n\max(t,\ell(Q))^{-n}}{1+\ell(Q)^{-n-L-1}|x-x_0|^{n+L+1}}\\
&\le
\frac{1}{1+\ell(Q)^{-n-L-1}|x-x_0|^{n+L+1}}.
\end{align*}
Therefore,
\[
\sup_{0<t<1}
|\varphi_t*(
\chi_{5Q}P_{2L+2n+2,5Q}[k*a]+\chi_{{\mathbb R}^n \setminus 5Q}k*a)
(x)|
\lesssim
\frac{1}{1+\ell(Q)^{-n-L-1}|x-x_0|^{n+L+1}}.
\]
Since
\[
{\rm supp}(
{\mathcal M}^0_{N_{p(\cdot),w}}[\chi_{5Q}P_{2L+2n+2,5Q}[k*a]+\chi_{{\mathbb R}^n \setminus 5Q}k*a])
\subset
Q(x_0,2+2\gamma_0)
\]
and
\[
\frac{\chi_{Q(x_0,2+2\gamma_0)}(x)}{1+\ell(Q)^{-n-L-1}|x-x_0|^{n+L+1}}
\lesssim
M^{\rm loc}\chi_Q(x)^{\frac{n+L+1}{n}},
\]
we obtain the desired result.
\end{proof}

\begin{proof}[Proof of Theorem \ref{thm:6.4}]
We assume $w \in L^1({\mathbb R}^n)$;
otherwise we can readily modify the proof below.
It suffices to show that
\[
\|T f\|_{h^{p(\cdot)}(w)}
\lesssim
\left(\|{\mathcal F}k\|_{L^\infty}+\sum_{j=0}^L B_j\right)
\|f\|_{h^{p(\cdot)}(w)}
\]
for all
$f \in h^{p(\cdot)}(w) \cap L^{p_++q_w}(w) \cap L^\infty({\mathbb R}^n)$
thanks to Lemma \ref{lem:190408-74}.

Let
$f \in h^{p(\cdot)}(w) \cap L^{p_++q_w}(w) \cap L^\infty({\mathbb R}^n)$
and
fix $L \gg n+1$.
Due to Theorem \ref{thm:190406-2},
there exist
$\{a_j\}_{j=0}^\infty \subset L^0({\mathbb R}^n)$,
$\{b_j\}_{j=1}^\infty \subset L^0({\mathbb R}^n)$,
$\{\lambda_j\}_{j=0}^{\infty} \subset [0, \infty)$
and
$\{\lambda_j'\}_{j=1}^{\infty} \subset [0, \infty)$
such that
$a_0$ is
a single $(p(\cdot),q)_w$-atom,
that
each
$a_j, j \in {\mathbb N}$
is
a $(p(\cdot),\infty,L)_w$-atom supported on a cube $Q_j$
with $|Q_j|<1$,
that
each
$b_j, j \in {\mathbb N}$
is
a $(p(\cdot),\infty,L)_w$-atom supported on a cube $R_j$
with $|R_j|=1$,
that
$f=\sum\limits_{j=0}^{\infty} \lambda_j a_j+\sum\limits_{j=1}^{\infty} \lambda_j' b_j$
holds
in the topology of
$h^{p(\cdot)}(w) \cap L^{p_++q_w}(w)$,
and that
\begin{equation}\label{eq:220216-101}
|\lambda_0|+
{\mathcal A}_{p(\cdot),w,v}(\{\lambda_j\}_{j=1}^{\infty};\{Q_j\}_{j=1}^{\infty})+
{\mathcal A}_{p(\cdot),w,v}(\{\lambda_j'\}_{j=1}^{\infty};\{R_j\}_{j=1}^{\infty})
\lesssim
\|f\|_{h^{p(\cdot)}(w)}.
\end{equation}

According to the famous Calder\'{o}n--Zygmund theory
(\cite{Gr,Gr14,Sawano-text-2018,Stein-text-93}),
\[
\|T\|_{L^2 \to L^2}
+
D_1(T)+D_2(T)
\lesssim
\|{\mathcal F}k\|_{L^\infty}+\sum_{j=0}^L B_j
\]
and that
$T$ maps
$L^{p_++q_w}(w)$
continuously to itself
and satisfies
\[
\|T\|_{L^{p_++q_w}(w) \to L^{p_++q_w}(w)}
\lesssim
\|{\mathcal F}k\|_{L^\infty}+\sum_{j=0}^L B_j
\] 
thanks to Lemma \ref{prop:200224-1}.
Thus,
$T f=\sum\limits_{j=0}^{\infty} \lambda_j T a_j+\sum\limits_{j=1}^{\infty} \lambda_j' T b_j$
holds
in the topology of
$L^{p_++q_w}(w)$.

We put
\[
{\rm I}_1:=
\left\|
\sum\limits_{j=0}^{\infty} \lambda_j 
\chi_{10Q_j}
{\mathcal M}^0_{N_{p(\cdot),w}}[T a_j]
\right\|_{L^{p(\cdot)}(w)}, \quad
{\rm I}_2:=
\left\|
\sum\limits_{j=1}^{\infty} \lambda_j' 
{\mathcal M}^0_{N_{p(\cdot),w}} [T b_j]
\right\|_{L^{p(\cdot)}(w)}
\]
and
\[
{\rm II}:=
\left\|
\sum\limits_{j=0}^{\infty} \lambda_j
\chi_{
{\mathbb R}^n \setminus
10Q_j} 
{\mathcal M}^0_{N_{p(\cdot),w}}[T a_j]
\right\|_{h^{p(\cdot)}(w)}.
\]
Set
\[
v:=\frac{\min(1,p_-)}{2},
\quad
uv>\max(p_+, q_w).
\]
Then, 
\begin{align*}
\|Tf\|_{h^{p(\cdot)}(w)}
&=
\left\|
\sum\limits_{j=0}^{\infty} \lambda_j T a_j+\sum\limits_{j=1}^{\infty} \lambda_j' T b_j
\right\|_{h^{p(\cdot)}(w)}\\
&\lesssim
\left\|
\sum\limits_{j=0}^{\infty} \lambda_j T a_j
\right\|_{h^{p(\cdot)}(w)}
+
\left\|
\sum\limits_{j=1}^{\infty} \lambda_j' T b_j
\right\|_{h^{p(\cdot)}(w)}\\
&=
\left\|
\sum\limits_{j=0}^{\infty} \lambda_j 
{\mathcal M}^0_{N_{p(\cdot),w}}[T a_j]
\right\|_{L^{p(\cdot)}(w)}
+
\left\|
\sum\limits_{j=1}^{\infty} \lambda_j' 
{\mathcal M}^0_{N_{p(\cdot),w}} [T b_j]
\right\|_{L^{p(\cdot)}(w)}\\
&\lesssim
{\rm I}_1+{\rm I}_2+{\rm II}.
\end{align*}
We employ 
Lemma \ref{lem:190408-100} and
Theorem \ref{lem:210805-1}
for ${\rm I}_1$ to give
\begin{align*}
{\rm I}_1
&\le
\left\|
\left(\sum\limits_{j=0}^{\infty} (\lambda_j 
\chi_{10Q_j}
{\mathcal M}^0_{N_{p(\cdot),w}}[T a_j])^v\right)^{\frac1v}
\right\|_{L^{p(\cdot)}(w)}\\
&\lesssim
{\mathcal A}_{p(\cdot),w,v}
(\{
\lambda_jm_{Q_j,w}^{(u v)}(\chi_{10Q_j}{\mathcal M}^0_{N_{p(\cdot),w}}[T a_j])
\}_{j=0}^{\infty};\{Q_j\}_{j=0}^{\infty}).
\end{align*}
Similar to Lemma \ref{lem:190408-3}, 
since $|a_j| \le \chi_{Q_j}$ and $w \in A_{u v}^{\rm loc}$,
we have
$$
m_{Q_j,w}^{(u v)}(\chi_{10Q_j}{\mathcal M}^0_{N_{p(\cdot),w}}[T a_j]) \lesssim
\frac{\|a_j\|_{L^{u v}(w)}}{\|\chi_{Q_j}\|_{L^{u v}(w)}} \lesssim 1.$$
Hence,
\[
{\rm I}_1
\lesssim
|\lambda_0|+
{\mathcal A}_{p(\cdot),w,v}
(\{\lambda_j\}_{j=1}^{\infty};\{Q_j\}_{j=1}^{\infty}).
\]
Meanwhile, 
by Theorem \ref{lem:210805-1} and 
Lemma \ref{lem:211029-113},
${\rm I}_2$ is estimated as
\begin{align*}
{\rm I}_2
&
\lesssim
{\mathcal A}_{p(\cdot),w,v}
(\{
\lambda_j' m_{R_j,w}^{(u v)}({\mathcal M}^0_{N_{p(\cdot),w}}[T b_j])
\}_{j=0}^{\infty};\{R_j\}_{j=0}^{\infty})\\
&\quad
\lesssim
{\mathcal A}_{p(\cdot),w,v}
(\{
\lambda_j' m_{R_j,w}^{(u v)}(M^{\rm loc}[T b_j])
\}_{j=0}^{\infty};\{R_j\}_{j=0}^{\infty})\\
&\quad
\lesssim
{\mathcal A}_{p(\cdot),w,v}
(\{
\lambda_j' m_{R_j,w}^{(u v)}(M^{\rm loc}[(T b_j)\chi_{3R_j}])
\}_{j=1}^{\infty};\{R_j\}_{j=1}^{\infty})\\
&\quad\qquad+
{\mathcal A}_{p(\cdot),w,v}
(\{
\lambda_j' m_{R_j,w}^{(u v)}(M^{\rm loc}[(M^{\rm loc}\chi_{R_j})^{\frac{n+L+1}{n}}])
\}_{j=1}^{\infty};\{R_j\}_{j=1}^{\infty})\\
&\quad
\lesssim
{\mathcal A}_{p(\cdot),w,v}(\{\lambda_j'\}_{j=1}^{\infty};\{R_j\}_{j=1}^{\infty}).
\end{align*}
Consequently, we have
\begin{equation}\label{eq:220216-102}
{\rm I}_1+{\rm I}_2\lesssim
|\lambda_0|+
{\mathcal A}_{p(\cdot),w,v}(\{\lambda_j\}_{j=1}^{\infty};\{Q_j\}_{j=1}^{\infty})+
{\mathcal A}_{p(\cdot),w,v}(\{\lambda_j'\}_{j=1}^{\infty};\{R_j\}_{j=1}^{\infty}).
\end{equation}
We employ 
Lemmas \ref{lem:190408-100}, \ref{lem:190408-1} and \ref{lem:220209-1}
for
${\rm II}$
to give
\begin{align}\label{eq:220216-103}
{\rm II}&\lesssim
\left\|
\sum\limits_{j=0}^{\infty} \lambda_j
\chi_{
{\mathbb R}^n \setminus
10Q_j} 
(M^{\rm loc}\chi_{Q_j})^{\frac{n+L+1}{n}}
\right\|_{L^{p(\cdot)}(w)}\\
&\quad+
\left\|
\sum\limits_{j=0}^{\infty} \lambda_j
\chi_{
{\mathbb R}^n \setminus
10Q_j} {\mathcal M}^0_{N_{p(\cdot),w}}[\chi_{5Q_j}(k*a-P_{2L+2n+2,5Q_j}[k*a_j])]
\right\|_{L^{p(\cdot)}(w)}\nonumber\\
&\lesssim
\left\|
\sum\limits_{j=0}^{\infty} \lambda_j
\chi_{
{\mathbb R}^n \setminus
10Q_j} 
(M^{\rm loc}\chi_{Q_j})^{\frac{n+L+1}{n}}
\right\|_{L^{p(\cdot)}(w)}\nonumber\\
&\lesssim
{\mathcal A}_{p(\cdot),w,v}(\{\lambda_j\}_{j=0}^{\infty};\{Q_j\}_{j=0}^{\infty}).\nonumber
\end{align}
If we combine (\ref{eq:220216-101})--(\ref{eq:220216-103}),
we obtain the desired result.
\end{proof}

\section{Littlewood--Paley characterization}
\label{s77}

Section \ref{s77}
considers
the Littlewood--Paley characterization of $h^{p(\cdot)}(w)$
as an application of the results 
in Section \ref{s60}.
The result of this section will be a natural
extension to the weighted case of the result
in \cite{NaSa12}.
What differs from \cite{NaSa12}
is that the Plancherel--Polya-Nikolski\'i inequality
is not available in this weighted setting.
To overcome this difficulty, we
use Corollary \ref{cor:210923-5}.
Section \ref{s731} modifies the idea in \cite{NaSa12},
where we refine what we obtained in Section \ref{s60}.
Under this modification, we combine the idea obtained in \cite{NaSa12}
with Corollary \ref{cor:210923-5}
in
Section \ref{s71}.
Section \ref{s7312} is devoted to the Littlewood--Paley characterization
of $h^{p(\cdot)}(w)$
as a preparatory step in Section \ref{s80}.

\subsection{Vector-valued extension of Theorem \ref{thm:6.4}}
\label{s731}

Theorem \ref{thm:7.2} is a natural extension
of Theorem \ref{thm:6.4}
in which $|\cdot|$ in the definition of ${\mathcal M}^0_{N_{p(\cdot),w}}f$
is
replaced by $\ell^2({\mathbb N}_0)$.
We introduce the $\ell^2({\mathbb N}_0)$-valued function space
$h^{p(\cdot)}(w;\ell^2({\mathbb N}_0))$.
Suppose that we are given a sequence
$\{f_j\}_{j=0}^\infty
\subset {\mathcal D}'({\mathbb R}^n)$.

Let $\psi \in {\mathcal D}({\mathbb R}^n)$
be a function such that $\chi_{[-1,1]^n} \le \psi \le \chi_{[-2,2]^n}$.
We set $\psi^k \equiv 2^{k n}\psi(2^{k}\cdot)$ for $k \in {\mathbb N}$.
With this in mind, we define
\[
\|\{f_j\}_{j=0}^\infty\|_{h^{p(\cdot)}(w,\ell^2)}
\equiv
\left\|
\sup_{k \in {\mathbb N}_0}
\left(
\sum_{j=0}^\infty |\psi^k*f_j|^2
\right)^{\frac12}
\right\|_{L^{p(\cdot)}(w)}.
\]
Observe that this is a natural vector-valued extension of
the quasi-norm equivalence 
\[
\|f\|_{h^{p(\cdot)}(w)}
\sim
\left\|
\sup_{k \in {\mathbb N}_0}
|\psi^k*f|\right\|_{L^{p(\cdot)}(w)}
\quad (f \in h^{p(\cdot)}(w)).
\]
The $\ell^2({\mathbb N}_0)$-valued function space
$h^{p(\cdot)}(w,\ell^2({\mathbb N}_0))$
is the set of all
$\{f_j\}_{j=0}^\infty
\subset {\mathcal D}'({\mathbb R}^n)$
for which
$
\|\{f_j\}_{j=0}^\infty\|_{h^{p(\cdot)}(w,\ell^2)}$ is finite.

Then the next theorem is analogous to Theorem \ref{thm:6.4}.
We omit the proof due to similarity.
\begin{theorem}[cf. {\cite[Theorem 5.6]{NaSa12}}]\label{thm:7.2}
Let $p(\cdot) \in {\mathcal P}_0
\cap {\rm L H}_0 \cap {\rm L H}_{\infty}$
and $w \in A_{\infty}^{\rm loc}$. 
If $T=\{T_k\}_{k \in {\mathbb N}_0}$ is a collection
of $L^2({\mathbb R}^n;\ell^2({\mathbb N}_0))$-$L^2({\mathbb R}^n)$ bounded operators
such that there exists a collection
$\{k_{ij}\}_{i,j \in {\mathbb N}_0} \subset {\mathcal D}({\mathbb R}^n)$
with the following properties{\rm:}
\begin{enumerate}
\item
There exists a constant $\gamma_0>0$ such that
$$\displaystyle
|x|^{n+m}\|\{\nabla^m k_{ij}(x)\}_{i,j \in {\mathbb N}_0}\|_{\ell^2({\mathbb N}_0) \to \ell^2({\mathbb N}_0)}
\lesssim
\chi_{[-\gamma_0,\gamma_0]^n}(x)
\quad (x \in {\mathbb R}^n). $$
for every $m \in {\mathbb N}_0$.
\item
If $\{f_j\}_{j=0}^\infty$ is a sequence of 
compactly supported $L^2({\mathbb R}^n)$-functions,
then 
\[
T_i[\{f_j\}_{j=0}^\infty](x)
=\sum_{j=0}^\infty
k_{ij}*f_j(x), \quad i \in {\mathbb N}_0
\]
for $x \in {\mathbb R}^n$.
\item
$k_{ij}\equiv 0$ if $|i|+|j|$ is large enough.
\end{enumerate}
Then 
\begin{align}
\label{eq:sawano-100429-2}
\left\|
\{T_i[\{f_j\}_{j=0}^\infty]\}_{i=0}^\infty
\right\|_{h^{p(\cdot)}(w,\ell^2)}
\lesssim
\left\|
\{f_j\}_{j=0}^\infty
\right\|_{h^{p(\cdot)}(w,\ell^2)}
\end{align}for all $\{f_j\}_{j=1}^\infty \in h^{p(\cdot)}(w,\ell^2({\mathbb N}_0))$.
\end{theorem}

\subsection{A vector-valued inequality}
\label{s71}

We will use the following vector-valued inequality
to obtain
the Littlewood--Paley characterization of $h^{p(\cdot)}(w)$.

Since the Fourier transform of non-zero compactly supported functions
is not compactly supported, we must taylor some
auxiliary estimate without using the Plancherel--Polya--Nikolski\'i inequality.
See \cite{Sawano-text-2018}
for
the Plancherel--Polya--Nikolski\'i inequality for example.
\begin{lemma}\label{lem:210923-6}
Let $p(\cdot) \in {\mathcal P}_0
\cap {\rm L H}_0 \cap {\rm L H}_{\infty}$
and $w \in A_{\infty}^{\rm loc}$. 
Assume that $L \gg 1$.
Let
$\phi,\phi^*
\in C^\infty_{\rm c}({\mathbb R}^n)$
satisfy
$\phi^* \in {\mathcal P}_{2L}^\perp({\mathbb R}^n)$,
that
\begin{equation}\label{eq:210923-1111}
|\phi|,
|\phi^*| \le \chi_{[-1,1]^n}
\end{equation}
and that
\begin{equation}\label{eq:210514-3}
\phi^*
=\phi-2^{-n}\phi\left(\frac{\cdot}{2}\right).
\end{equation}
Then, we have
\[
\left\|
\sup_{k \in {\mathbb N}_0}
\left(
\sum_{j=1}^\infty |\Phi_{2^{-k}}*\phi^*_{2^{-j}}*f|^2
\right)^{\frac12}
\right\|_{L^{p(\cdot)}(w)}
\lesssim
\|\phi*f\|_{L^{p(\cdot)}(w)}+
\left\|
\left(
\sum_{j=1}^\infty |\phi^*_{2^{-j}}*f|^2
\right)^{\frac12}
\right\|_{L^{p(\cdot)}(w)}
\]for all $f \in {\mathcal D}'({\mathbb R}^n)$ and
$\Phi \in C^\infty_{\rm c}({\mathbb R}^n)$
with ${\rm supp}(\Phi) \subset [-1,1]^n$.
In particular,
\[
\left\|
\{\phi^*_{2^{-j}}*f\}_{j=0}^\infty
\right\|_{h^{p(\cdot)}(w,\ell^2)}
\lesssim
\|\phi*f\|_{L^{p(\cdot)}(w)}+
\left\|
\left(
\sum_{j=1}^\infty |\phi^*_{2^{-j}}*f|^2
\right)^{\frac12}
\right\|_{L^{p(\cdot)}(w)}.
\]
\end{lemma}

We remark that the couple
$(\phi,\phi^* 
)$
exists according to \cite[Lemma 6.5]{IST10}.

\begin{proof}We employ Lemma \ref{lem:210712-1}.
Choose
$\psi, \psi^* \in C^\infty_{\rm c}({\mathbb R}^n)$
so that
\begin{equation}\label{eq:210514-1}
|\psi|,
|\psi^*| \le \chi_{[-1,1]^n}
\end{equation}
that
\begin{equation}\label{eq:210514-2}
\psi^* \in {\mathcal P}_{2L}^{\perp}({\mathbb R}^n),
\end{equation}
and that
\begin{equation}\label{eq:210514-5}
\phi*\psi+\sum_{j=1}^{\infty} \phi^*_{2^{-j}}*\psi^*_{2^{-j}}=\delta
\end{equation}
in the topology of ${\mathcal D}'({\mathbb R}^n)$.
Fix $k$ and $j$ for now.
We decompose
\[
\Phi_{2^{-k}}*\phi^*_{2^{-j}}*f
=
\Phi_{2^{-k}}*\psi*\phi^*_{2^{-j}}*\phi*f
+
\sum_{l=1}^\infty
\Phi_{2^{-k}}*\phi^*_{2^{-j}}*\psi^*_{2^{-l}}*\phi^*_{2^{-l}}*f
\]
It follows from Lemma \ref{lem3sw} that
\begin{equation}\label{eq:211031-4}
|\Phi_{2^{-k}}*\psi*\phi^*_{2^{-j}}|
\lesssim
2^{-2Lj}\chi_{[-3,3]^n}
\end{equation}
and that
\begin{align}\label{eq:211031-5}
|\Phi_{2^{-k}}*\phi^*_{2^{-j}}*\psi^*_{2^{-l}}|
&\lesssim
2^{n\min(j,k,l)-2L\min(k-j,k-l,l-j,j-l)}
\chi_{[-2^{2-\min(j,k,l)},2^{2-\min(j,k,l)}]^n}\\
&\lesssim
2^{n\min(j,k,l)+2L|l-j|}
\chi_{[-2^{2-\min(j,k,l)},2^{2-\min(j,k,l)}]^n}. \nonumber
\end{align}
Not that $\min(k-j,k-l,l-j,j-l)\le -|l-j|$.
Let $L \gg 2A>A>B \gg 1$.
Let $r$ be a constant, which is slightly less than 
$\frac{\min(1,p_-)}{q_w} (<1)$.
Thanks to Lemma \ref{lem:211029-111}, 
(\ref{eq:211031-4}) and H\"older's inequality for $l$, we have
\begin{align*}
\lefteqn{
|\Phi_{2^{-k}}*\psi*\phi^*_{2^{-j}}*\phi*f(x)|
}\\
&\lesssim
2^{-2Lj}\int_{[-3,3]^n}|\phi*f(x-z)|{\rm d}z\\
&\lesssim
2^{-2Lj}\int_{[-3,3]^n}
\left(
\int_{{\mathbb R}^n}\frac{|\phi*f(x-z-y)|^r}{m_{0,A r,B r}(y)}{\rm d}y
\right)^{\frac1r}
{\rm d}z\\
&\quad+
2^{-2Lj}\int_{[-3,3]^n}
\left(\sum_{l=1}^\infty
2^{l n-lLr}
\int_{{\mathbb R}^n}\frac{|\phi^*_{2^{-l}}*f(x-z-y)|^r}{m_{0,A r,B r}(y)}{\rm d}y
\right)^{\frac1r}
{\rm d}z\\
&\lesssim
2^{-2Lj}\int_{[-3,3]^n}
\left(
\int_{{\mathbb R}^n}\frac{|\phi*f(x-z-y)|^r}{m_{0,A r,B r}(y)}{\rm d}y
\right)^{\frac1r}
{\rm d}z\\
&\quad+
2^{-2Lj}\int_{[-3,3]^n}
\sum_{l=1}^\infty
2^{-2lA}
\left(2^{l n}
\int_{{\mathbb R}^n}\frac{|\phi^*_{2^{-l}}*f(x-z-y)|^r}{m_{0,A r,B r}(y)}{\rm d}y
\right)^{\frac1r}
{\rm d}z.
\end{align*}
Since $$m_{0,A r,B r}(y) \sim m_{0,A r,B r}(y+z)$$
for all $y \in {\mathbb R}^n$ and
$z \in [-3,3]^n$
and $$2^{l A r}m_{0,A r,B r}(y) \ge m_{l,A r,B r}(y)$$
for all $y \in {\mathbb R}^n$,
we see that
\begin{align}\label{eq:220221-2}
&|\Phi_{2^{-k}}*\psi*\phi^*_{2^{-j}}*\phi*f(x)|\\
&\lesssim
2^{-2Lj}
\left(
\int_{{\mathbb R}^n}\frac{|\phi*f(x-y)|^r}{m_{0,A r,B r}(y)}{\rm d}y
\right)^{\frac1r}
+
2^{-2Lj}
\sum_{l=1}^\infty
2^{-2l A}
\left(2^{l n}
\int_{{\mathbb R}^n}\frac{|\phi_{2^{-l}}^**f(x-y)|^r}{m_{0,A r,B r}(y)}{\rm d}y
\right)^{\frac1r}
\nonumber\\
&\lesssim
2^{-2Lj}
\left(
\int_{{\mathbb R}^n}\frac{|\phi*f(x-y)|^r}{m_{0,A r,B r}(y)}{\rm d}y
\right)^{\frac1r}
+
2^{-2Lj}
\left\{
\sum_{l=1}^\infty
\left(2^{l n}
\int_{{\mathbb R}^n}\frac{|\phi_{2^{-l}}^**f(x-y)|^r}{m_{l,A r,B r}(y)}{\rm d}y
\right)^{\frac2r}\right\}^{\frac12}\nonumber
\end{align}
by the H\"older inequality.
Likewise,
\begin{align*}
\lefteqn{
|\Phi_{2^{-k}}*\phi^*_{2^{-j}}*\psi^*_{2^{-l}}*\phi^*_{2^{-l}}*f(x)|
}\\
&\lesssim
2^{n\min(j,k,l)+2L|l-j|}
\int_{[-2^{2-\min(j,k,l)},2^{2-\min(j,k,l)}]^n}
|\phi^*_{2^{-l}}*f(x-z)|{\rm d}z
\end{align*}
by 
using (\ref{eq:211031-5}).
Since
\begin{align*}
m_{l, A r, B r}(y+z)
&=
(1+2^l|y+z|)^{Ar}e^{|y+z|Br}\\
&\le
(1+2^l|z|)^{Ar}
(1+2^l|y|)^{Ar}
e^{|y|Br+|z|Br}\\
&\lesssim
2^{Ar\max(0,l-\min(l,j,k))}
m_{l, A r, B r}(y)
\end{align*}
for all $z \in [-2^{2-\min(j,k,l)},2^{2-\min(j,k,l)}]^n$,
thanks to 
Lemmas \ref{lem3sw} and \ref{lem:211029-111},
we obtain
\begin{align*}
\lefteqn{
|\Phi_{2^{-k}}*\phi^*_{2^{-j}}*\psi^*_{2^{-l}}*\phi^*_{2^{-l}}*f(x)|
}\\
&\lesssim
2^{n\min(j,k,l)+2L|l-j|}\\
&\quad\quad \times
\int_{[-2^{2-\min(j,k,l)},2^{2-\min(j,k,l)}]^n}
\left(\sum_{l'=l}^\infty
2^{l' n+(l-l')Lr}\int_{{\mathbb R}^n}
\frac{|\phi^*_{2^{-l'}}*f(x-z-y)|^r}{m_{l, A r, B r}(y)}{\rm d}y
\right)^{\frac1r}
{\rm d}z
\\
&\lesssim
2^{n\min(j,k,l)+2L|l-j|+A\max(0,l-\min(l,j,k))}\\
&\quad\quad \times
\int_{[-2^{2-\min(j,k,l)},2^{2-\min(j,k,l)}]^n}
\sum_{l'=l}^\infty 2^{2(l-l')A}
\left(2^{l' n}\int_{{\mathbb R}^n}
\frac{|\phi^*_{2^{-l'}}*f(x-y)|^r}{m_{l, A r, B r}(y)}{\rm d}y
\right)^{\frac1r}
{\rm d}z.
\end{align*}
Due to the fact that $L \gg A \gg 1$,
we have
\begin{align*}
\lefteqn{
|\Phi_{2^{-k}}*\phi^*_{2^{-j}}*\psi^*_{2^{-l}}*\phi^*_{2^{-l}}*f(x)|
}\\
&\lesssim
\sum_{l'=l}^\infty
2^{L|l-j|+2A(l-l')}
\left(2^{l' n}\int_{{\mathbb R}^n}
\frac{|\phi^*_{2^{-l'}}*f(x-y)|^r}{m_{l, A r, B r}(y)}{\rm d}y\right)^{\frac1r}\\
&\lesssim
\sum_{l'=l}^\infty
2^{-L|l-j|+2A(l-l')}
\left(2^{l' n}\int_{{\mathbb R}^n}
\frac{|\phi^*_{2^{-l'}}*f(x-y)|^r}{m_{l, A r, B r}(y)}{\rm d}y\right)^{\frac1r}.
\end{align*}
Since $2^{-Ar(l-l')}m_{l,A r,B r} \ge m_{l',A r,B r}$,
\begin{align*}
\lefteqn{
|\Phi_{2^{-k}}*\phi^*_{2^{-j}}*\psi^*_{2^{-l}}*\phi^*_{2^{-l}}*f(x)|
}\\
&\lesssim
\sum_{l'=l}^\infty
2^{-L|l-j|+A(l-l')}
\left(2^{l' n}\int_{{\mathbb R}^n}
\frac{|\phi^*_{2^{-l'}}*f(x-y)|^r}{m_{l', A r, B r}(y)}{\rm d}y\right)^{\frac1r}.
\end{align*}
Hence, we have
\begin{align*}
\lefteqn{
\sum_{l=1}^\infty
|\Phi_{2^{-k}}*\phi^*_{2^{-j}}*\psi^*_{2^{-l}}*\phi^*_{2^{-l}}*f(x)|
}\\
&\lesssim
\sum_{l=1}^\infty
\sum_{l'=l}^\infty
2^{-L|l-j|+A(l-l')}
\left(2^{l' n}\int_{{\mathbb R}^n}
\frac{|\phi^*_{2^{-l'}}*f(x-y)|^r}{m_{l', A r, B r}(y)}{\rm d}y\right)^{\frac1r}\\
&\le
\sum_{l'=1}^\infty
\sum_{l=-\infty}^{l'}
2^{-L|l-j|+A(l-l')}
\left(2^{l' n}\int_{{\mathbb R}^n}
\frac{|\phi^*_{2^{-l'}}*f(x-y)|^r}{m_{l', A r, B r}(y)}{\rm d}y\right)^{\frac1r},
\end{align*}
Since
\[
\sum_{l=-\infty}^{l'}
2^{-L|l-j|+A(l-l')}
\sim
2^{-L|l'-j|}
\]
thanks to the fact that $L>A>0$,
we have
\begin{equation}\label{eq:220221-1}
\sum_{l=1}^\infty
|\Phi_{2^{-k}}*\phi^*_{2^{-j}}*\psi^*_{2^{-l}}*\phi^*_{2^{-l}}*f(x)|
\lesssim
\sum_{l'=1}^\infty
2^{-L|l'-j|}
\left(2^{l' n}\int_{{\mathbb R}^n}
\frac{|\phi^*_{2^{-l'}}*f(x-y)|^r}{m_{l', A r, B r}(y)}{\rm d}y\right)^{\frac1r}.
\end{equation}
Therefore,
from
(\ref{eq:220221-2})
and (\ref{eq:220221-1}),
we have
\begin{align}\label{eq:211201-4}
&|\Phi_{2^{-k}}*\phi^*_{2^{-j}}*f(x)|\\
&\lesssim
2^{-2Lj}
\left(
\int_{{\mathbb R}^n}\frac{|\phi*f(x-y)|^r}{m_{0,A r,B r}(y)}{\rm d}y
\right)^{\frac1r}
+
2^{-2Lj}\left\{
\sum_{l=1}^\infty
\left(2^{l n}
\int_{{\mathbb R}^n}\frac{|\phi_{2^{-l}}^**f(x-y)|^r}{m_{l,A r,B r}(y)}{\rm d}y
\right)^{\frac2r}\right\}^{\frac12}\nonumber \\
&\quad+
\sum_{l'=1}^\infty
2^{-L|l'-j|}
\left(2^{l' n}\int_{{\mathbb R}^n}
\frac{|\phi^*_{2^{-l'}}*f(x-y)|^r}{m_{l', A r, B r}(y)}{\rm d}y\right)^{\frac1r}.\nonumber
\end{align}
We observe
\begin{align*}
\lefteqn{
\sum_{j=1}^\infty
\left\{
\sum_{l'=1}^\infty
2^{-L|l'-j|}
\left(2^{l' n}\int_{{\mathbb R}^n}
\frac{|\phi^*_{2^{-l'}}*f(x-y)|^r}{m_{l', A r, B r}(y)}{\rm d}y
\right)^{\frac1r}
\right\}^2
}\\
&\le
\sum_{j=1}^\infty
\left[
\sum_{l'=-\infty}^\infty
2^{-L|l'-j|}
\times
\sum_{l'=1}^\infty
2^{-L|l'-j|}
\left(2^{l' n}\int_{{\mathbb R}^n}
\frac{|\phi^*_{2^{-l'}}*f(x-y)|^r}{m_{l', A r, B r}(y)}{\rm d}y
\right)^{\frac2r}
\right]\\
&\lesssim
\sum_{j=1}^\infty
\sum_{l'=1}^\infty
2^{-L|l'-j|}
\left(2^{l' n}\int_{{\mathbb R}^n}
\frac{|\phi^*_{2^{-l'}}*f(x-y)|^r}{m_{l', A r, B r}(y)}{\rm d}y
\right)^{\frac2r}\\
&=
\sum_{l'=1}^\infty
\sum_{j=1}^\infty
2^{-L|l'-j|}
\left(2^{l' n}\int_{{\mathbb R}^n}
\frac{|\phi^*_{2^{-l'}}*f(x-y)|^r}{m_{l', A r, B r}(y)}{\rm d}y
\right)^{\frac2r}.
\end{align*}Note that
\[
\sum_{j=1}^\infty 2^{-L|l'-j|}\lesssim
\sum_{j=-\infty}^\infty 2^{-L|l'-j|} \sim 1.
\]
Therefore,
\begin{align}\label{eq:220221-5}
&\sum_{j=1}^\infty
\left\{
\sum_{l'=l}^\infty
2^{-L|l'-j|}
\left(2^{l' n}\int_{{\mathbb R}^n}
\frac{|\phi^*_{2^{-l'}}*f(x-y)|^r}{m_{l', A r, B r}(y)}{\rm d}y
\right)^{\frac1r}
\right\}^2\\
&\quad \lesssim
\sum_{l'=1}^\infty
\left(2^{l' n}\int_{{\mathbb R}^n}
\frac{|\phi^*_{2^{-l'}}*f(x-y)|^r}{m_{l', A r, B r}(y)}{\rm d}y
\right)^{\frac2r}.\nonumber
\end{align}
Likewise we can prove
\begin{align}\label{eq:211201-5}
&|\Phi_{2^{-k}}*\phi*f(x)|\\
&\lesssim
\left(
\int_{{\mathbb R}^n}\frac{|\phi*f(x-y)|^r}{m_{0,A r,B r}(y)}{\rm d}y
\right)^{\frac1r}
+
\left(
\sum_{l=1}^\infty
\left(2^{l n}\int_{{\mathbb R}^n}
\frac{|\phi^*_{2^{-l}}*f(x-y)|^r}{m_{l, A r, B r}(y)}{\rm d}y\right)^{\frac2r}
\right)^{\frac12}\nonumber
\end{align}for all $k \in {\mathbb N}_0$ with the implicit constant
independent of $k$.
If we take the supremum over $k\in {\mathbb N}_0$
and the $\ell^2$-norm for $j$, 
we obtain
\begin{align*}
\lefteqn{
\sup_{k \in {\mathbb N}_0}
|\Phi_{2^{-k}}*\phi*f(x)|
+
\left(
\sum_{j=1}^\infty
\sup_{k \in {\mathbb N}_0}
|\Phi_{2^{-k}}*\phi^*_{2^{-j}}*f(x)|^2
\right)^{\frac12}
}\\
&\lesssim
\left(
\int_{{\mathbb R}^n}\frac{|\phi*f(x-y)|^r}{m_{0,A r,B r}(y)}{\rm d}y
\right)^{\frac1r}
+
\left(
\sum_{l=1}^\infty
\left(2^{l n}\int_{{\mathbb R}^n}
\frac{|\phi^*_{2^{-l}}*f(x-y)|^r}{m_{l, A r, B r}(y)}{\rm d}y\right)^{\frac2r}
\right)^{\frac12}
\end{align*}
from (\ref{eq:211201-4}), (\ref{eq:220221-5}) and (\ref{eq:211201-5}).
To complete, invoke Corollary \ref{cor:210923-5}.
\end{proof}

\subsection{Littlewood--Paley characterization of $h^{p(\cdot)}(w)$}
\label{s7312}

An important consequence of Theorem \ref{thm:7.2}
is the Littlewood--Paley characterization of $h^{p(\cdot)}(w)$.
We obtain it under a strong assumption of $L$.
\begin{proposition}\label{prop:210923-11}
Let $p(\cdot) \in {\mathcal P}_0
\cap {\rm L H}_0 \cap {\rm L H}_{\infty}$
and $w \in A_{\infty}^{\rm loc}$. 
Assume $L \gg 1$.
Let
$\phi,\phi^* \in {\mathcal D}({\mathbb R}^n)$
satisfy
$\phi^* \in {\mathcal P}_L^\perp({\mathbb R}^n)$,
$(\ref{eq:210923-1111})$
and
$(\ref{eq:210514-3})$.
Then
a distribution $f \in {\mathcal D}'({\mathbb R}^n)$
belongs to $h^{p(\cdot)}(w)$ if and only if
\[
\|\phi*f\|_{L^{p(\cdot)}(w)}+
\left\|
\left(
\sum_{j=1}^\infty |\phi^*_{2^{-j}}*f|^2
\right)^{\frac12}
\right\|_{L^{p(\cdot)}(w)}<\infty
\]
In this case, we have
\begin{equation}\label{eq:210923-1112}
\|f\|_{h^{p(\cdot)}(w)}
\sim
\|\phi*f\|_{L^{p(\cdot)}(w)}+
\left\|
\left(
\sum_{j=1}^\infty |\phi^*_{2^{-j}}*f|^2
\right)^{\frac12}
\right\|_{L^{p(\cdot)}(w)}.
\end{equation}
\end{proposition}
\begin{proof}
Assuming that 
$f \in h^{p(\cdot)}(w)$,
we first show that
\[
\|\phi*f\|_{L^{p(\cdot)}(w)}+
\left\|
\left(
\sum_{j=1}^\infty |\phi^*_{2^{-j}}*f|^2
\right)^{\frac12}
\right\|_{L^{p(\cdot)}(w)}
\lesssim
\|f\|_{h^{p(\cdot)}(w)}.
\]
The definition of the grand maximal function
${\mathcal M}^0_{N_{p(\cdot), w}} f$ easily gives that
$$
\|\phi*f\|_{L^{p(\cdot)}(w)}
\lesssim
\|{\mathcal M}^0_{N_{p(\cdot), w}} f\|_{L^{p(\cdot)}(w)}
\lesssim
\|f\|_{h^{p(\cdot)}(w)}.
$$
Therefore,
we must show that
\[
\left\|
\left(
\sum_{j=1}^\infty |\phi^*_{2^{-j}}*f|^2
\right)^{\frac12}
\right\|_{L^{p(\cdot)}(w)}
\lesssim
\|f\|_{h^{p(\cdot)}(w)}.
\]
By the monotone convergence theorem,
the matters are reduced to showing that
\[
\left\|
\left(
\sum_{j=1}^N |\phi^*_{2^{-j}}*f|^2
\right)^{\frac12}
\right\|_{L^{p(\cdot)}(w)}
\lesssim
\|f\|_{h^{p(\cdot)}(w)}
\]with the implicit constant independent of $N$.
By the Khinchine inequality,
we have only to show that
\[
\left\|
\sum_{j=1}^N a_j\phi^*_{2^{-j}}*f
\right\|_{L^{p(\cdot)}(w)}
\lesssim
\|f\|_{h^{p(\cdot)}(w)}
\]
for any sequences $\{a_j\}_{j=1}^N \subset \{-1,1\}^N$.
However, this is a direct consequence of Theorem \ref{thm:6.4}.

Let us move on to the proof of
\[
\|f\|_{h^{p(\cdot)}(w)}
\lesssim
\|\phi*f\|_{L^{p(\cdot)}(w)}+
\left\|
\left(
\sum_{j=1}^\infty |\phi^*_{2^{-j}}*f|^2
\right)^{\frac12}
\right\|_{L^{p(\cdot)}(w)}.
\]
Choose
$\psi, \psi^* \in C^\infty_{\rm c}({\mathbb R}^n)$ 
so that
$(\ref{eq:210514-1})$--$(\ref{eq:210514-5})$ hold.
Consider the operator
\[
\{g_j\}_{j=0}^\infty \in
h^{p(\cdot)}(w, \ell^2({\mathbb N}_0))
\mapsto
\psi*g_0+
\sum_{j=1}^\infty
\psi^*_{2^{-j}}*g_j\in
h^{p(\cdot)}(w).
\]
This operator is $h^{p(\cdot)}(w, \ell^2({\mathbb N}_0))$-$h^{p(\cdot)}(w)$ bounded thanks to Theorem \ref{thm:7.2}.
As a result,
\[
\left\|
\psi*f_0+
\sum_{j=1}^\infty
\psi^*_{2^{-j}}*g_j
\right\|_{h^{p(\cdot)}(w)}
\lesssim
\left\|
\{g_j\}_{j=0}^\infty
\right\|_{h^{p(\cdot)}(w,\ell^2)}.
\]
The right-hand side must be written out fully.
Use
(\ref{eq:210514-5}) and
Lemma \ref{lem:210923-6}.
If we let
$g_0=\phi*f$,
$g_j=\phi^*_{2^{-j}}*f$ 
$(j \ge 1)$,
then
\[
\left\|f
\right\|_{h^{p(\cdot)}(w)}
\lesssim
\left\|
\{g_j\}_{j=0}^\infty
\right\|_{h^{p(\cdot)}(w,\ell^2)}\lesssim
\|\phi*f\|_{L^{p(\cdot)}(w)}+
\left\|
\left(
\sum_{j=1}^\infty |\phi^*_{2^{-j}}*f|^2
\right)^{\frac12}
\right\|_{L^{p(\cdot)}(w)},
\]since
\[
f=\psi*\phi*f+\sum_{j=1}^\infty
\psi^*_{2^{-j}}*\phi_{2^{-j}}*f.
\]
Thus, the proof is complete.
\end{proof}

Let us relax the assumption
on $L$
in Proposition \ref{prop:210923-11}.
\begin{theorem}\label{thm:7.4}
Let $f \in {\mathcal D}'({\mathbb R}^n)$.
In Proposition \ref{prop:210923-11},
the same conclusion holds
if $L=0$.
\end{theorem}

\begin{proof}
We start with the set-up.
Assume $L^\dagger \gg s_0 \gg1$.
Let
$\zeta,\zeta^* \in C^\infty_{\rm c}({\mathbb R}^n)$
satisfy
$\zeta^* \in {\mathcal P}_{L^\dagger}^\perp({\mathbb R}^n)$,
\[
|\zeta|,
|\zeta^*| \le \chi_{[-1,1]^n}.
\]
and
\[
\zeta^*=\zeta-\frac{1}{2^n}\zeta\left(\frac{\cdot}{2}\right).
\]
Let
$\phi,\phi^* \in C^\infty_{\rm c}({\mathbb R}^n)$
satisfy
$\phi^* \in {\mathcal P}_0^\perp({\mathbb R}^n)$,
$(\ref{eq:210923-1111})$
and
$(\ref{eq:210514-3})$.

It suffices to show that
\begin{align}
\lefteqn{
\|\phi*f\|_{L^{p(\cdot)}(w)}+
\left\|
\left(
\sum_{j=1}^\infty |\phi^*_{2^{-j}}*f|^2
\right)^{\frac12}
\right\|_{L^{p(\cdot)}(w)}
}\nonumber\\
\label{eq:210923-1113}
&\sim
\|\zeta*f\|_{L^{p(\cdot)}(w)}+
\left\|
\left(
\sum_{j=1}^\infty |\zeta^*_{2^{-j}}*f|^2
\right)^{\frac12}
\right\|_{L^{p(\cdot)}(w)},
\end{align}
since we proved
in Proposition
\ref{prop:210923-11}
that
$f \in h^{p(\cdot)}(w)$
holds
if and only if
\[
\|\zeta*f\|_{L^{p(\cdot)}(w)}+
\left\|
\left(
\sum_{j=1}^\infty |\zeta^*_{2^{-j}}*f|^2
\right)^{\frac12}
\right\|_{L^{p(\cdot)}(w)}<\infty
\]
and that the norm equivalence
(\ref{eq:210923-1112}) holds.

We content ourselves with proving
\[
\left\|
\left(
\sum_{j=1}^\infty |\zeta^*_{2^{-j}}*f|^2
\right)^{\frac12}
\right\|_{L^{p(\cdot)}(w)}
\lesssim
\|\phi*f\|_{L^{p(\cdot)}(w)}+
\left\|
\left(
\sum_{j=1}^\infty |\phi^*_{2^{-j}}*f|^2
\right)^{\frac12}
\right\|_{L^{p(\cdot)}(w)},
\]
since the remaining estimates
needed to establish
(\ref{eq:210923-1113})
can be proved similarly.

Let $r>0$ be a constant which is slightly less than $\frac{p_-}{q_w}$.
Fix $x \in {\mathbb R}^n$.
We assume that $A<L^\dagger$ and $B r>8n+6\log D$.
Let
$j$ be fixed.
Then we have
\begin{align*}
|\zeta^*_{2^{-j}}*f(x)|
&\lesssim
\left(
\sum_{l=j}^\infty
2^{L^\dagger(j-l)r}
2^{l n}
\int_{{\mathbb R}^n}
\frac{|\phi^*_{2^{-l}}*f(x-y)|^r}{m_{l,A r,B r}(y)}dy
\right)^{\frac1r}\\
&\lesssim
\sum_{l=j}^\infty
2^{A(j-l)r}
\left(2^{l n}
\int_{{\mathbb R}^n}
\frac{|\phi^*_{2^{-l}}*f(x-y)|^r}{m_{l,A r,B r}(y)}dy
\right)^{\frac1r}\\
\end{align*}
thanks to Lemma \ref{lem:211029-211}.
Consequently,
\[
\left(
\sum_{j=1}^\infty |\zeta^*_{2^{-j}}*f(x)|^2
\right)^{\frac12}
\lesssim
\left(
\sum_{j=1}^\infty 
\left(
2^{l n}
\int_{{\mathbb R}^n}
\frac{|\phi^*_{2^{-l}}*f(x-y)|^r}{m_{l,A r,B r}(y)}dy
\right)^{\frac2r}
\right)^{\frac12}.
\]
It remains to use Corollary \ref{cor:210923-5}.
\end{proof}

\section{Wavelet characterization}
\label{s80}
As a further application of Theorem \ref{thm:7.4},
we consider the wavelet expansion.

Choose
compactly supported 
$C^r$-functions for large enough  $r \in {\mathbb N}$
\begin{equation}\label{eq:210511-1}
\varphi
\mbox{ and }
\psi^l \quad (l=1,2, \ldots , 2^n-1)
\end{equation}
so that the following conditions are satisfied:
\begin{enumerate}
\item[$(1)$]
For any $J\in \mathbb{Z}$, 
the system
\[
\left\{ \varphi_{J,k}, \, \psi^l_{j,k} \, : \, 
k\in \mathbb{Z}^n , \, j\ge J, \, l=1,2, \ldots , 2^n-1 \right\}
\]
is an orthonormal basis of $L^2({\mathbb R}^n)$.
Here, given a 
function $F$ defined on $\mathbb{R}^n$, we write 
\[
F_{j,k}\equiv 2^{\frac{j n}{2}}F(2^j \cdot -k)
\]
for $j \in {\mathbb Z}$ and $k\in {\mathbb Z}^n$.
\item[$(2)$]
Fix a large integer
$L \in {\mathbb N}$ for now.
We have
\begin{equation}\label{eq:211029-22}
\psi^l
\in{\mathcal P}_{L}^\perp({\mathbb R}^n)
\quad
(l=1,2, \ldots , 2^n-1).
\end{equation}
In addition, 
they are real-valued 
and compactly supported with
\begin{equation}\label{eq:200111-1}
\mathrm{supp}(\varphi)=\mathrm{supp}(\psi^l)=[0,2N-1]^n
\end{equation}
for some $N\in \mathbb{N}$.
See \cite{LM} for example.
\end{enumerate}

We also define
$\displaystyle 
\chi_{j,k} \equiv 
2^{\frac{j n}{2}}\chi_{Q_{j,k}}
$
and
$\displaystyle 
\chi^*_{j,k} \equiv 
2^{\frac{j n}{2}}\chi_{Q^*_{j,k}}
$
for $j \in {\mathbb Z}$ and $k=(k_1,k_2,\ldots,k_n) \in {\mathbb Z}^n$,
where $Q_{j,k}$ and $Q^*_{j,k}$ are the dyadic cube and its expansion
which are
given by
\begin{equation}
\label{define-Qjk}
Q_{j,k}\equiv \prod_{m=1}^n
\left[ 2^{-j}k_m , 2^{-j}(k_m+1) \right]
\end{equation}
 and 
\begin{equation}\label{define-Qjk*}
Q^*_{j,k}\equiv \prod_{m=1}^n
\left[ 2^{-j}k_m , 2^{-j}(k_m+2N-1) \right],
\end{equation} respectively.
Then using
the $L^2$-inner product
$\langle \cdot,\cdot \rangle$,
for $f \in L^1_{\rm loc}$,
we define two square functions
$V f$,
$W_sf$ by
\begin{align*}
Vf&\equiv V^{\varphi}f\equiv \left( \sum_{k\in \mathbb{Z}^n} \left| 
\langle f,\varphi_{J,k} \rangle \chi_{J,k} \right|^2 \right)^{\frac12},\\
W f&\equiv W^{\psi^l}f\equiv \left( 
\sum_{l=1}^{2^n-1}
\sum_{j=J}^\infty
\sum_{k\in \mathbb{Z}^n} \left| 
\langle f, \psi_{j,k}^l \rangle \chi_{j,k} \right|^2 \right)^{\frac12}.
\end{align*}
Here, $J$ is a fixed integer.
In 1994, Lemari$\acute{\rm{e}}$-Rieusset
commented that the class Muckenhoupt 
has a lot to do with
the wavelet characterization
\cite{Lemarie}.
We remark that
Kopaliani
considered the wavelet characterization
of $L^{p(t)}({\mathbb R})$
in 2008
(see \cite{Kop}).

Based on these works, we will prove the following theorem:
\begin{theorem}\label{thm:8.1}
Let $p(\cdot) \in {\mathcal P}_0
\cap {\rm L H}_0 \cap {\rm L H}_{\infty}$
and $w \in A_{\infty}^{\rm loc}$. 
Assume that 
\begin{equation}\label{eq:220123-1}
L \ge \max\left(-1,\left[n\left(\frac{q_w}{\min(1,p_-)}-1\right)\right]\right)
\end{equation}
in $(\ref{eq:211029-22})$.
\begin{enumerate}
\item
Let $f \in h^{p(\cdot)}(w)$.
Then
\[
\|f\|_{h^{p(\cdot)}(w)} \sim 
\|V f\|_{L^{p(\cdot)}(w)}+\|W f\|_{L^{p(\cdot)}(w)}.
\]
\item
If 
$f \in L^\infty_{\rm c}({\mathbb R}^n)$
satisfies
$V f+W f \in L^{p(\cdot)}(w)$,
then
$f \in h^{p(\cdot)}(w)$.
\end{enumerate}
\end{theorem}
The proof of Theorem \ref{thm:8.1} consists of several steps.
We start with a setup.
Let
$\phi \in C^\infty_{\rm c}({\mathbb R}^n) \setminus {\mathcal P}_0({\mathbb R}^n)^\perp$
and
$\phi^* \in C^\infty_{\rm c}({\mathbb R}^n)$
satisfy
\[
{\rm supp}(\phi) \subset [-1,1]^n, \quad
\phi^*=\phi-\frac1{2^n}\phi\left(\frac{\cdot}{2}\right).
\]
Choose $\psi,\psi^*$ according to Lemma \ref{lem:210712-1}.
In the light of the construction in \cite{Rychkov01},
we can arrange 
$\phi$,
$\phi^*$,
$\psi$ and $\psi^*$
so that they are
even functions
satisfying
\[\phi*\psi+\sum_{j=1}^\infty \phi^*_{2^{-j}}*\psi^*_{2^{-j}}=\delta
\]in ${\mathcal D}'({\mathbb R}^n)$
and that
$\phi^*, \psi^* \in {\mathcal P}_{n+L}({\mathbb R}^n)^\perp$
where $L$ is in (\ref{eq:220123-1}).
We must justify the definition of the couplings
$\langle f, \psi^l_{j,k} \rangle$
and
$\langle f,\varphi_{J,k} \rangle$.
To this end, we will prove the following estimate:
\begin{lemma}\label{lem:211023-1}
For all 
$f \in L^{p_++q_w}(w) \cap h^{p(\cdot)}(w)$,
\begin{equation}\label{eq:211023-2}
\|f\|_{h^{p(\cdot)}(w)} \gtrsim 
\|V f\|_{L^{p(\cdot)}(w)}+\|W f\|_{L^{p(\cdot)}(w)}.
\end{equation}
\end{lemma}
Before the proof, we offer a word
about Lemma \ref{lem:211023-1}.
Fu and Yang obtained a similar estimate
in \cite[Theorem 1.9]{FuYa18}.
However,
we cannot use \cite[Theorem 1.9]{FuYa18} directly
due to the presence of $w \in A^{\rm loc}_\infty$.
As such, we must establish an estimate from scratch.
\begin{proof}
It suffices to prove that
\begin{equation}\label{eq:210924-12}
V f+W f
\lesssim
\sup_{z \in {\mathbb R}^n}
\frac{|\phi_{2^{-J}}*f(\cdot-z)|}{m_{J,A,B}(z)}+
\left(
\sum_{j'=J+1}^\infty
\sup_{z \in {\mathbb R}^n}
\frac{|\phi^*_{2^{-j'}}*f(\cdot-z)|^2}{m_{j',A,B}(z)^2}
\right)^{\frac12}.
\end{equation}
Once this estimate is shown, we obtain the conclusion as follows.
Fix $r>0$ slightly less than $\frac{\min(1,p_-)}{q_w}$.
Then $(n/r,n+L+1) \ne \emptyset$, 
so we can take $A\in(n/r,n+L+1) \cap {\mathbb N}$.
For the first term, 
by Lemma \ref{lem:211029-111} and H\"older's inequality, we have
\begin{align*}
\sup_{z \in {\mathbb R}^n}
\frac{|\phi_{2^{-J}}*f(\cdot-z)|}{m_{J,A,B}(z)}
&\lesssim
\left(
\sum_{k=J}^{\infty}
2^{k n+(J-k)(n+L+1) r}
\int_{{\mathbb R}^n}
\frac{|\phi^*_{2^{-k}}*f(x-y)|^r}{m_{J,A r,B r}(y)}{\rm d}y
\right)^{\frac1r}\\
&\lesssim
\left(
\sum_{k=J}^\infty
\left(
2^{kn}
\int_{{\mathbb R}^n}
\frac{|\phi^*_{2^{-k}}*f(x-y)|^r}{m_{k,A r,B r}(y)}{\rm d}y
\right)^{\frac2r}
\right)^{\frac12}.
\end{align*}
For the second term, 
since $A>n/r$,
again by Lemma \ref{lem:211029-111} and H\"older's inequality, 
we have
\begin{align*}
\sup_{z \in {\mathbb R}^n}
\frac{|\phi^*_{2^{-j'}}*f(\cdot-z)|^2}{m_{j',A,B}(z)^2}
&\lesssim
\left(
\sum_{k=j'}^{\infty}
2^{k n+(j'-k)(n+L+1) r}
\int_{{\mathbb R}^n}
\frac{|\phi^*_{2^{-k}}*f(x-y)|^r}{m_{j',A r,B r}(y)}{\rm d}y
\right)^{\frac2r}\\
&\lesssim
\sum_{k=j'}^{\infty}
2^{2(j'-k)(n+L+1-\varepsilon_0)}
\left(2^{kn}
\int_{{\mathbb R}^n}
\frac{|\phi^*_{2^{-k}}*f(x-y)|^r}{m_{j',A r,B r}(y)}{\rm d}y
\right)^{\frac2r}.
\end{align*}
Here, we choose $0<\varepsilon_0<n+L+1-A$.
 Using the estimate $m_{j',A r,B r} \ge 2^{Ar(j'-k)}m_{k,A r,B r}$,
 we have
\begin{align*}
\lefteqn{
\sum_{j'=J+1}^\infty
\sup_{z \in {\mathbb R}^n}
\frac{|\phi^*_{2^{-j'}}*f(\cdot-z)|^2}{m_{j',A,B}(z)^2}
}\\
&\lesssim
\sum_{k=J}^\infty
\sum_{j'=J}^{k}
2^{2(j'-k)(n+L+1-\varepsilon_0)}
\left(
2^{kn}
2^{-Ar(j'-k)}
\int_{{\mathbb R}^n}
\frac{|\phi^*_{2^{-k}}*f(x-y)|^r}{m_{k,A r,B r}(y)}{\rm d}y
\right)^{\frac2r}\\
&\lesssim
\sum_{k=J}^\infty
\left(
\sum_{j'=-\infty}^{k}
2^{2(j'-k)(n+L+1-\varepsilon_0-A)}
\right)
\left(
2^{kn}
\int_{{\mathbb R}^n}
\frac{|\phi^*_{2^{-k}}*f(x-y)|^r}{m_{k,A r,B r}(y)}{\rm d}y
\right)^{\frac2r}\\
&\sim
\sum_{k=J}^\infty
\left(
2^{kn}
\int_{{\mathbb R}^n}
\frac{|\phi^*_{2^{-k}}*f(x-y)|^r}{m_{k,A r,B r}(y)}{\rm d}y
\right)^{\frac2r}\\
\end{align*}
Hence,
\begin{align*}
V f+W f
\lesssim
\left(
\sum_{k=J}^\infty
\left(
2^{kn}
\int_{{\mathbb R}^n}
\frac{|\phi^*_{2^{-k}}*f(x-y)|^r}{m_{k,A r,B r}(y)}{\rm d}y
\right)^{\frac2r}
\right)^{\frac12}
\end{align*}
Finally, we can resort 
to the 
vector-valued boundedness of the operators (Corollary \ref{cor:210923-5}).

So, we move on to
estimate (\ref{eq:210924-12}).
However, since we can handle $V f$ similar to $W f$,
we content ourselves with the proof of
\begin{equation}\label{eq:210924-14}
W f
\lesssim
\sup_{z \in {\mathbb R}^n}
\frac{|\phi_{2^{-J}}*f(\cdot-z)|}{m_{J,A,B}(z)}+
\left(
\sum_{j'=J+1}^\infty
\sup_{z \in {\mathbb R}^n}
\frac{|\phi^*_{2^{-j'}}*f(\cdot-z)|^2}{m_{j',A,B}(z)^2}
\right)^{\frac12}.
\end{equation}
instead of (\ref{eq:210924-12}).

Fix $r>0$ slightly less than $\frac{\min(1,p_-)}{q_w}$.
Then $(n/r,n+L+1) \ne \emptyset$, so we can take $A\in(n/r,n+L+1)$.
Let $j \in {\mathbb N}$ be fixed.
Since $\psi,\psi^*$ are radial,
\begin{align}
\label{eq:210924-11111}
\left| 
\langle f, \psi_{j,k}^l \rangle \chi_{j,k} \right|
&\le
\left| 
\langle \phi_{2^{-J}}*\psi_{2^{-J}}*f, \psi_{j,k}^l \rangle \chi_{j,k} \right|+
\sum_{j'=J+1}^\infty
\left| 
\langle \phi^*_{2^{-j'}}*\psi^*_{2^{-j'}}*f, \psi_{j,k}^l \rangle \chi_{j,k} \right|\\
\nonumber
&=
\left| 
\langle \phi_{2^{-J}}*f, \psi_{2^{-J}}*\psi_{j,k}^l \rangle \chi_{j,k} \right|+
\sum_{j'=J+1}^\infty
\left| 
\langle \phi^*_{2^{-j'}}*f, \psi^*_{2^{-j'}}*\psi_{j,k}^l \rangle \chi_{j,k} \right|.
\end{align}

By the moment condition,
\begin{align}
\label{eq:211014-1}
2^{\frac{j n}{2}}
|\psi_{2^{-J}}*\psi_{j,k}^l|
&\lesssim
2^{J n-(j-J)(n+L+1)}
\chi_{2^{4+j-J}Q^*_{j,k}}, \\
2^{\frac{j n}{2}}
|\psi^*_{2^{-j'}}*\psi_{j,k}^l|
&\lesssim
2^{\min(j,j')n-n\max(j'-j,0)-|j-j'|(L+1)}
\chi_{16Q^*_{j,k} \cup 2^{4+j-j'}Q^*_{j,k}}.
\nonumber
\end{align}
By inserting (\ref{eq:211014-1}) into (\ref{eq:210924-11111}),
we obtain
\begin{align*}
&\left| \langle f, \psi_{j,k}^l \rangle \chi_{j,k} \right|\\
&\quad\lesssim
2^{-(j-J)(n+L+1)}
\left(\int_{2^{4+j-J}Q^*_{j,k}}|\phi_{2^{-J}}*f(y)|{\rm d}y\right)\chi_{Q_{j,k}}\\
&\quad\quad+
\sum_{j'=J+1}^\infty
2^{\min(j,j')n-n\max(j'-j,0)-|j-j'|(L+1)}
\left(
\int_{16Q^*_{j,k} \cup 2^{4+j-j'}Q^*_{j,k}}|\phi^*_{2^{-j'}}*f(y)|{\rm d}y\right)\chi_{Q_{j,k}}.
\end{align*}
Let
$x \in Q_{j,k}$.
Using the function $m_{J,A,B}$,
we estimate
\[
\left(\int_{2^{4+j-J}Q^*_{j,k}}|\phi_{2^{-J}}*f(y)|{\rm d}y\right)
\lesssim
2^{-J n}
\sup_{z \in {\mathbb R}^n}
\frac{|\phi_{2^{-J}}*f(x-z)|}{m_{J,A,B}(z)}.
\]
Likewise,
for $x \in Q_{j,k}$,
\begin{align*}
\int_{16Q^*_{j,k} \cup 2^{4+j-j'}Q^*_{j,k}}
|\phi^*_{2^{-j'}}*f(y)|{\rm d}y
&\lesssim
2^{-\min(j,j')n+A\max(j'-j,0)}
\sup_{z \in {\mathbb R}^n}
\frac{|\phi^*_{2^{-j'}}*f(x-z)|}{m_{j',A,B}(z)}.
\end{align*}
Consequently,
\begin{align*}
\lefteqn{
\left(
\sum_{k \in {\mathbb Z}^n}
\left| 
\langle f, \psi_{j,k}^l \rangle \chi_{j,k} \right|^2
\right)^{\frac12}
}\\
&\lesssim
2^{-(j-J) (n+L+1)}
\sup_{z \in {\mathbb R}^n}
\frac{|\phi_{2^{-J}}*f(\cdot-z)|}{m_{J,A,B}(z)}
+
\sum_{j'=J+1}^\infty
2^{-|j'-j|(L+1+n-A)}
\sup_{z \in {\mathbb R}^n}
\frac{|\phi^*_{2^{-j'}}*f(\cdot-z)|}{m_{j',A,B}(z)}.
\end{align*}
Recall that $A$ satisfies
$L+1+n>A$.
Taking the $\ell^2$-norm over $j=J, J+1,J+2,\ldots$ and $\ell=1,2, \ldots, 2^n-1$,
 we obtain
 \begin{align*}
 \lefteqn{
\left(
\sum_{\ell=1}^{2^n-1}
\sum_{j=J}^\infty
\sum_{k \in {\mathbb Z}^n}
\left| 
\langle f, \psi_{j,k}^l \rangle \chi_{j,k} \right|^2
\right)^{\frac12}
}\\
&\lesssim
\left(
\sum_{j=J}^\infty
2^{-2(j-J) (n+L+1)}
\right)^{\frac12}
\sup_{z \in {\mathbb R}^n}
\frac{|\phi_{2^{-J}}*f(\cdot-z)|}{m_{J,A,B}(z)}\\
&\quad\quad+
\left(
\sum_{j=J}^\infty
\left(
\sum_{j'=J}^\infty
2^{-|j'-j|(L+1+n-A)}
\sup_{z \in {\mathbb R}^n}
\frac{|\phi^*_{2^{-j'}}*f(\cdot-z)|}{m_{j',A,B}(z)}
\right)^2
\right)^{\frac12}\\
&\lesssim
\sup_{z \in {\mathbb R}^n}
\frac{|\phi_{2^{-J}}*f(\cdot-z)|}{m_{J,A,B}(z)}
+
\left(
\sum_{j'=J}^\infty
\sup_{z \in {\mathbb R}^n}
\frac{|\phi^*_{2^{-j'}}*f(\cdot-z)|^2}{m_{j',A,B}(z)^2}
\right)^{\frac12}.
 \end{align*}
Thus, the proof is complete.
\end{proof}
Lemma \ref{lem:211023-1} has an important consequence.
First,
since $L^\infty_{\rm c}({\mathbb R}^n) \cap h^{p(\cdot)}(w)$
is dense in $h^{p(\cdot)}(w)$
(see Lemma \ref{lem:190408-74}),
we can extend couplings
$\langle f, \psi^l_{j,k} \rangle$
and
$\langle f,\varphi_{J,k} \rangle$,
which are initially defined for
$f \in L^\infty_{\rm c}({\mathbb R}^n) \cap h^{p(\cdot)}(w)$,
to bounded linear functionals
from
$L^\infty_{\rm c}({\mathbb R}^n) \cap h^{p(\cdot)}(w)$.
We still write
$\langle f, \psi^l_{j,k} \rangle$
and
$\langle f,\varphi_{J,k} \rangle$
for these extended functionals.
By the Fatou lemma,
we have
(\ref{eq:211023-2})
for all
$f \in h^{p(\cdot)}(w)$,
\begin{corollary}\label{cor:211023-1}
The conclusion of Lemma \ref{lem:211023-1} remains valid
for all
$f \in h^{p(\cdot)}(w)$.
\end{corollary}

With Corollary \ref{cor:211023-1} in mind, we complete the proof
of Theorem \ref{thm:8.1}.
Let us prove
\[
\|V f+W f\|_{L^{p(\cdot)}(w)}
\sim
\|\phi*f\|_{L^{p(\cdot)}(w)}
+
\left\|
\left(
\sum_{j=1}^\infty |\phi^*_{2^{-j}}*f|^2
\right)^{\frac12}\right\|_{L^{p(\cdot)}(w)}.
\]In view of Corollary \ref{cor:211023-1},
it remains to establish
\begin{equation}\label{eq:210924-11}
\|V f+W f\|_{L^{p(\cdot)}(w)}
\gtrsim
\|\phi*f\|_{L^{p(\cdot)}(w)}
+
\left\|
\left(
\sum_{j=1}^\infty |\phi^*_{2^{-j}}*f|^2
\right)^{\frac12}
\right\|_{L^{p(\cdot)}(w)}
\end{equation}
for all $f \in L^{p_++q_w}(w)$.
As before,
since the remaining estimates
needed for
(\ref{eq:210924-11})
are easier to prove, we content ourselves
with the proof of
\begin{equation}\label{eq:210924-13}
\|V f+W f\|_{L^{p(\cdot)}(w)}
\gtrsim
\left\|\left(
\sum_{j=1}^\infty |\phi^*_{2^{-j}}*f|^2
\right)^{\frac12}\right\|_{L^{p(\cdot)}(w)}
\end{equation}
instead of (\ref{eq:210924-11}).

We prove
(\ref{eq:210924-13}).
Let $j \in {\mathbb N} \cap [J, \infty)$ be fixed.
Since
$f \in L^{p_++q_w}(w)$,
we can use the wavelet decomposition
obtained
in \cite{INNS-wavelet}.
and estimate
each term of the decomposition:
\begin{align*}
\phi^*_{2^{-j}}*f
&=
\sum_{k \in {\mathbb Z}^n}
\langle f,\varphi_{J,k} \rangle
\phi^*_{2^{-j}}*\varphi_{J,k}
+
\sum_{l=1}^{2^n-1}
\sum_{j'=J}^\infty
\sum_{k \in {\mathbb Z}^n}
\langle f, \psi^l_{j',k} \rangle
\phi^*_{2^{-j}}*\psi^l_{j',k}.
\end{align*}
We have
\[
2^{\frac{J n}{2}}|\phi^*_{2^{-j}}*\varphi_{J,k}|
\lesssim
2^{J n-(j-J)(n+L+1)}
\chi_{16Q^*_{J,k}}
\]
and
\[
2^{\frac{j'n}{2}}|\phi^*_{2^{-j}}*\psi^l_{j',k}|
\lesssim
\begin{cases}
2^{j'n-(j-j')(n+L+1)}
\chi_{16Q^*_{j',k} \cup 2^{4+j'-j}Q^*_{j',k}}
&(j' \le j),\\
2^{j n-(j'-j)(L+1)}
\chi_{16Q^*_{j',k} \cup 2^{4+j'-j}Q^*_{j',k}}
&(j' \ge j).
\end{cases}
\]
As a result,
\begin{align*}
\sum_{k \in {\mathbb Z}^n}
|\langle f,\varphi_{J,k} \rangle
\phi^*_{2^{-j}}*\varphi_{J,k}|
&\lesssim
\sum_{k \in {\mathbb Z}^n}
2^{\frac{J n}{2}-(j-J)(n+L+1)}
|\langle f,\varphi_{J,k} \rangle|
\chi_{16Q^*_{J,k}}
\end{align*}
and
\begin{align*}
\sum_{j'=J}^\infty
\sum_{k \in {\mathbb Z}^n}
|\langle f, \psi^l_{j',k} \rangle\phi^*_{2^{-j}}*\psi^l_{j',k}|
&\lesssim\sum_{j'=J}^\infty
\sum_{k \in {\mathbb Z}^n}
2^{\frac{j'n}{2}-|j-j'|(n+L+1)}
|
\langle f, \psi^l_{j',k} \rangle|
\chi_{16Q^*_{j',k} \cup 2^{4+j'-j}Q^*_{j',k}}.
\end{align*}
Choose $r\in(0,1)$ so that
\[
n+L+1-\frac{n}{r}>0, \quad
\frac{p_-}{r}>q_w
=\inf\{u \in [1,\infty)\,:\,w \in A_u^{\rm loc}\}.
\]
Then
\begin{align*}
\sum_{j=J}^\infty
\left(
\sum_{k \in {\mathbb Z}^n}
|\langle f,\varphi_{J,k} \rangle
\phi^*_{2^{-j}}*\varphi_{J,k}|
\right)^2
&\lesssim
\left(\sum_{j=J}^\infty
2^{-2(j-J)(n+L+1)}
\right)
\left(
\sum_{k \in {\mathbb Z}^n}
2^{\frac{J n}{2}}
|\langle f,\varphi_{J,k} \rangle|
\chi_{16Q^*_{J,k}}
\right)^2\\
&\lesssim
\left(
\sum_{k \in {\mathbb Z}^n}
|\langle f,\varphi_{J,k} \rangle|
M^{\rm loc}[(\chi_{{J,k}})^r]^{\frac1r}
\right)^2
\end{align*}
and
\begin{align*}
\lefteqn{
\sum_{j=J}^\infty
\left(
\sum_{j'=J}^\infty
\sum_{k \in {\mathbb Z}^n}
|\langle f, \psi^l_{j',k} \rangle\phi^*_{2^{-j}}*\psi^l_{j',k}|
\right)^2}\\
&\lesssim
\sum_{j=J}^\infty
\left(\sum_{j'=J}^\infty
\sum_{k \in {\mathbb Z}^n}
2^{\frac{j'n}{2}-|j-j'|(n+L+1)}
|\langle f, \psi^l_{j',k} \rangle|
\chi_{16Q^*_{j',k} \cup 2^{4+j'-j}Q^*_{j',k}}
\right)^2\\
&\lesssim
\sum_{j=J}^\infty
\left(\sum_{j'=J}^\infty
\sum_{k \in {\mathbb Z}^n}
2^{\frac{j'n}{2}-|j-j'|(n+L+1-n/r)}
|\langle f, \psi^l_{j',k} \rangle|
(M^{\rm loc}\chi_{Q_{j',k}})^{\frac1r}
\right)^2.
\end{align*}
By Proposition \ref{prop:190408-10011}
\begin{align*}
\lefteqn{
\left\|\left\{
\sum_{j=J}^\infty
\left(
\sum_{j'=J}^\infty
\sum_{k \in {\mathbb Z}^n}
|\langle f, \psi^l_{j',k} \rangle\phi^*_{2^{-j}}*\psi^l_{j',k}|
\right)^2\right\}^{\frac12}\right\|_{L^{p(\cdot)}(w)}}\\
&\lesssim
\left\|\left\{
\sum_{j=J}^\infty
\left(\sum_{j'=J}^\infty
\sum_{k \in {\mathbb Z}^n}
2^{\frac{j'n}{2}-|j-j'|(n+L+1-n/r)}
|\langle f, \psi^l_{j',k} \rangle|
(M^{\rm loc}\chi_{Q_{j',k}})^{\frac1r}
\right)^2\right\}^{\frac12}\right\|_{L^{p(\cdot)}(w)}\\
&\lesssim
\left\|\left\{
\sum_{j=J}^\infty
\left(\sum_{j'=J}^\infty
\sum_{k \in {\mathbb Z}^n}
2^{\frac{j'n}{2}-|j-j'|(n+L+1-n/r)}
|\langle f, \psi^l_{j',k} \rangle|
\chi_{Q_{j',k}}
\right)^2\right\}^{\frac12}\right\|_{L^{p(\cdot)}(w)}.
\end{align*}
By H\"{o}lder's inequality,
\begin{align*}
\lefteqn{
\left\|\left\{
\sum_{j=J}^\infty
\left(
\sum_{j'=J}^\infty
\sum_{k \in {\mathbb Z}^n}
|\langle f, \psi^l_{j',k} \rangle\phi^*_{2^{-j}}*\psi^l_{j',k}|
\right)^2\right\}^{\frac12}\right\|_{L^{p(\cdot)}(w)}}\\
&\lesssim
\left\|\left\{
\sum_{j=J}^\infty
\sum_{j'=J}^\infty
\sum_{k \in {\mathbb Z}^n}
\left(2^{\frac{j'n}{2}-\frac12|j-j'|(n+L+1-n/r)}
|\langle f, \psi^l_{j',k} \rangle|
\chi_{Q_{j',k}}
\right)^2\right\}^{\frac12}\right\|_{L^{p(\cdot)}(w)}\\
&\lesssim
\left\|\left\{
\sum_{j'=J}^\infty
\sum_{k \in {\mathbb Z}^n}
\left(
|\langle f, \psi^l_{j',k} \rangle|
\chi_{{j',k}}
\right)^2\right\}^{\frac12}\right\|_{L^{p(\cdot)}(w)},
\end{align*}
as required.


\section{Examples and relations to other function spaces}
\label{s9}

In this section we give some examples of the weighted local Hardy spaces with variable exponents and weights.
One of the significant example is the Dirac Delta. 
We consider the condition to belong to $h^{p(\cdot)}(w)$ in Section \ref{s91}.
Next we provide the examples of weights.  
We handle the power weights in Section \ref{s92} 
and the exponential weights in Section \ref{s93}, respectively.
Finally, Section \ref{s94} and \ref{s95} is devoted to consider the relation to other function spaces.

\subsection{Dirac Delta}\label{s91}

Let $w \in A_\infty^{\rm loc}$.
Also let
$p(\cdot) \in {\mathcal P}_0 \cap {\rm L H}_0 \cap {\rm L H}_\infty$.
Let $\delta$ be the Dirac Delta.
Then
${\mathcal M}_{N_{p(\cdot)},w}^0\delta(x) \sim |x|^{-n}$
near the origin and
${\mathcal M}_{N_{p(\cdot)},w}^0\delta$ is 
bounded away from the origin and supported
on a bounded set.
Therefore,
if
$p(\cdot)$ and $w$ satisfy
\begin{equation}\label{eq:220424-11}
\int_{B(1)}|x|^{-n p(x)}w(x){\rm d}x<\infty,
\end{equation}
then
$\delta \in h^{p(\cdot)}(w)$.

\begin{example}
The following couples satisfy 
(\ref{eq:220424-11}) and falls within the scope
of this paper.
\begin{enumerate}
\item
$w(x)=1$,
$p(x)=\max(2^{-1},\min(1,|x|))$, $x \in {\mathbb R}^n$.
\item
$w(x)=\frac{|x|^{n+1}}{1+|x|^{2n+1}}$,
$p(x)=2$, $x \in {\mathbb R}^n$.
\item
$w(x)=|x|^{n+1}\exp(|x|)$,
$p(x)=2$, $x \in {\mathbb R}^n$.
\end{enumerate}
\end{example}

\subsection{Case of power weights}\label{s92}

Let $\mu \in {\mathbb R}$ and define
$w_\mu(x) \equiv (1+|x|)^{\mu}$.
Also let
$p(\cdot) \in {\mathcal P}_0 \cap {\rm L H}_0 \cap {\rm L H}_\infty$.
\begin{proposition}\label{prop:220424-1}
${\mathcal S}({\mathbb R}^n) \subset h^{p(\cdot)}(w_\mu)$.
\end{proposition}

\begin{proof}
This is a consequence of the atomic decomposition.
We can decompose any $f \in {\mathcal S}({\mathbb R}^n)$
into the sum of $(p(\cdot),q,-1)_{w_\mu}$-atoms:
\[
f=\sum_{m \in {\mathbb Z}^n}\lambda_m a_m,
\]
where $\lambda_m={\rm O}((1+|m|)^{-N})$ for any $N \in {\mathbb N}$
and each $a_m$ is a  $(p(\cdot),q,-1)_{w_\mu}$-atom
supported on $Q(m,1/2)$.
\end{proof}
\begin{proposition}\label{prop:220424-2}
Any element $f \in h^{p(\cdot)}(w_\mu)$,
which is initially defined as an element in ${\mathcal D}'({\mathbb R}^n)$
can be extended uniquely to the continuous functional
over ${\mathcal S}({\mathbb R}^n)$,
that is,
$h^{p(\cdot)}(w_\mu) \subset {\mathcal S}'({\mathbb R}^n)$.
\end{proposition}

\begin{proof}
In fact,
if we argue as in \cite[Proposition 3.1]{Tang12}
and write $\check{\varphi}=\varphi(-\cdot)$ for $\varphi \in {\mathcal S}({\mathbb R}^n)$,
we obtain
\begin{align*}
|\langle f,\varphi(\cdot-x) \rangle|
&=
|f*\check{\varphi}(x)|\\
&\lesssim
\frac{1}{\|\chi_{B(x,1)}\|_{L^{p(\cdot)}(w_\mu)}}
\|\check{\varphi}\|_{{\mathcal D}_{N_{p(\cdot)},w_\mu}}
\|f\|_{h^{p(\cdot)}(w_\mu)}
\end{align*}
for any $\varphi \in {\mathcal D}_{N_{p(\cdot)},w_\mu}({\mathbb R}^n)$.
Notice that
$\|\chi_{B(x,1)}\|_{L^{p(\cdot)}(w_\mu)} \gtrsim (1+|x|)^K$
for some $K>0$.

By the use of the partition of unity, any
$\varphi \in {\mathcal S}({\mathbb R}^n)$ has the following decomposition:
\[
\varphi=\sum_{m \in {\mathbb Z}^n}
a_m \varphi_m(\cdot-m),
\]
where
each $\varphi_m \in {\mathcal D}_N({\mathbb R}^n)$ depends
linearly on $\varphi$ and
$|a_m| \lesssim (1+|m|)^{-N}$
for any $N \in {\mathbb N}$.
Therefore, we can define
\[
\langle f,\varphi \rangle
=\sum_{m \in {\mathbb Z}^n}
a_m \langle f,\varphi_m(\cdot-m) \rangle
\]
for $f \in h^{p(\cdot)}(w_\mu)$ and $\varphi \in {\mathcal S}({\mathbb R}^n)$,
where the convergence takes place absolutely.
Thus,
$h^{p(\cdot)}(w_\mu) \subset {\mathcal S}'({\mathbb R}^n)$.
\end{proof}

\begin{proposition}
Let $\kappa \in {\mathbb R}$.
Then
$f \mapsto (1+|\cdot|^2)^{\frac{\kappa}{2}}f$
is an isomorphism
from $h^{p(\cdot)}(w_\mu)$ to $h^{p(\cdot)}(w_{\mu-\kappa})$.
\end{proposition}

\begin{proof}
Simply observe
${\mathcal M}^0_{N_{p(\cdot)},w_\mu}[(1+|\cdot|^2)^{\frac{\kappa}{2}}f]
\lesssim w_\kappa{\mathcal M}^0_{N_{p(\cdot)},w_\mu}f$.
\end{proof}

\subsection{Case of exponential weights}\label{s93}

We work in ${\mathbb R}$.
Let $\mu \in {\mathbb R}$ and define
$w^{(\mu)}(x) \equiv \exp(\mu x)$
for $x \in {\mathbb R}$.
Also let
$p(\cdot) \in {\mathcal P}_0 \cap {\rm L H}_0 \cap {\rm L H}_\infty$.
\begin{proposition}
Let $\kappa \in {\mathbb R}$.
The mapping
$f \mapsto w^{(\kappa)}f$
is an isomorphism from
$h^{p(\cdot)}(w^{(\mu)})$ to $ h^{p(\cdot)}(w^{(\mu-\kappa)})$.
\end{proposition}

\begin{proof}
Simply observe
${\mathcal M}^0_{N_{p(\cdot)},w^{(\mu)}}(w^{(\kappa)} f)
\sim w^{(\kappa)} {\mathcal M}^0_{N_{p(\cdot)},w^{(\mu)}}f$.
\end{proof}
Similar phenomena can be observed
if $|\mu| \ll 1$.
We omit further details.

\subsection{Periodic case}\label{s94}

Although the exponent $p(\cdot)$ must be constant
in this subsection, it seems useful
to discuss periodic function spaces.
Let $0<p<\infty$.
Let 
$L^p({\mathbb T}^n)$
be the set of all $p$-locally integrable functions
$f$ with period $1$ for which
\[
\|f\|_{L^p({\mathbb T}^n)}=
\left(
\int_{[0,1]^n}|f(x)|^p{\rm d}x
\right)^{\frac1p}<\infty.
\]
Similarly, 
the periodic local Hardy space
$h^p({\mathbb T}^n)$ is the set of all periodic distributions
$f \in {\mathcal D}'({\mathbb T}^n) \hookrightarrow {\mathcal S}'({\mathbb R}^n)$
for which
\[
\sup_{0<t \le 1}
\sup_{\varphi \in {\mathcal D}_N}|\varphi_t*f| \in
L^p({\mathbb T}^n).
\]
The norm is given by
\[
\|f\|_{h^{p}({\mathbb T}^n)}=
\left\|
\sup_{0<t \le 1}
\sup_{\varphi \in {\mathcal D}_N}|\varphi_t*f|
\right\|_{L^p({\mathbb T}^n)}.
\]
If a variable exponent $p(\cdot)$ is periodic and satisfies
the global $\log$-H\"{o}lder condition,
then $p(\cdot)$ must be constant.
Thus, we assume that $p(\cdot)$ is a constant here.
\begin{lemma}\label{lem:200429-1}
For any $0< p \le \infty$,
$L^p({\mathbb T}^n) \hookrightarrow L^p(w_{-n-1})$
and
\[
\|f\|_{L^p({\mathbb T}^n)} \sim \|f\|_{L^p(w_{-n-1})}.
\]
In particular,
$h^p({\mathbb T}^n) \hookrightarrow h^p(w_{-n-1})$.
\end{lemma}

\begin{proof}
Note that 
\[
w_{-n-1}(x)=(1+|x|)^{-n-1} \sim M\chi_{Q_{0,0}}(x)^{\frac{n+1}{n}}
\]
for $x \in {\mathbb R}^n$.
Hence simply use
$\sum\limits_{m \in {\mathbb Z}^n}(1+|m|)^{-n-1}<\infty$.
\end{proof}

\subsection{Weighted uniformly local Lebesgue spaces with variable exponents}
\label{s95}
Let $p(\cdot) \in {\mathcal P}
\cap {\rm L H}_0 \cap {\rm L H}_{\infty}
$
and $w \in A_{p(\cdot)}^{\rm loc}$.
Then the weighted uniformly local
Lebesgue space $L^{p(\cdot)}_{\rm uloc}(w)$ with a variable exponent
is defined to be all
$f \in L^1_{\rm loc}$ for which the norm
$\|f\|_{L^{p(\cdot)}_{\rm uloc}(w)}=\sup\limits_{m \in {\mathbb Z}^n}
\|\chi_{Q_{0,m}}f\|_{L^{p(\cdot)}(w)}$
is finite.
This is a natural extension of the
uniformly local
Lebesgue space $L^p_{\rm uloc}$,
which considers a substitute of $L^\infty$.
If we replace the supremum by the $\ell^r$-norm,
then 
the weighted amalgam space $(\ell^r,L^{p(\cdot)}_{\rm uloc}(w))$
with
a variable exponent
is obtained
as an extension of the amalgam space $(\ell^r,L^p)$
considered in 
\cite{BDD78,CMP99,FoSt85,Holland75,KNTYY07}.
Although our results are applicable to amalgam spaces,
to simplify the argument, we consider uniformly local Lebesgue spaces
with variable exponents.

For $w \in A_{p(\cdot)}^{\rm loc}$,
we write $w_m(x)=w(x)(1+|x-m|)^{-p_+(1+n)}$.
Then by the triangle inequality, we can check that
\begin{equation}\label{eq:200307-1}
\|f\|_{L^{p(\cdot)}_{\rm uloc}(w)} \sim \sup\limits_{m \in {\mathbb Z}^n}
\|f\|_{L^{p(\cdot)}(w_m)}.
\end{equation}
Therefore,
if we define
the weighted uniformly locally integrable local Hardy spaces
$h^{p(\cdot)}_{\rm uloc}(w)$
with variable exponent $p(\cdot)$ and weight $w$
to be the set of all distributions $f \in {\mathcal D}'({\mathbb R}^n)$
for which
\[
\|f\|_{h^{p(\cdot)}_{\rm uloc}(w)}
:=
\|{\mathcal M}_{N_{p(\cdot)},w}^0f\|_{L^{p(\cdot)}_{\rm uloc}(w)}
\]
is finite,
then we can apply the results obtained in this paper
to $h^{p(\cdot)}_{\rm uloc}(w)$.
For example, as in this paper, we can obtain the Litttlewood--Paley characterization.

\section{Appendix--Proof of Proposition \ref{prop:211027-111}}
\label{s88}

Let $w$ be a weight.
It is known that
$w \in A^{\rm loc}_{p(\cdot)}$ if and only if $M^{\rm loc}$ is bounded
on $L^{p(\cdot)}(w)$.
In this section, we characterize
the class $A^{\rm loc}_{p(\cdot)}$
motivated by reference \cite{DiHapre}.
As a corollary of this characterization,
which is stated in Proposition \ref{prop:211027-111},
we show that 
$A^{\rm loc}_{p(\cdot)}$ is monotone in $p(\cdot)$.
That is,
if $p(\cdot),q(\cdot) \in {\mathcal P} \cap {\rm LH}_0 \cap {\rm LH}_\infty$
satisfy $p(\cdot) \le q(\cdot)$,
then $A^{\rm loc}_{p(\cdot)} \subset A^{\rm loc}_{q(\cdot)}$.
Similar to Section \ref{subsection:Dyadic grids},
the matters are reduced to the maximal operator
generated by dyadic grids and
let
${\mathfrak D}={\mathcal D}_{{(1,1,\ldots,1)}}$.
We can handle ${\mathcal D}_{\bf a}$
for other values of ${\bf a} \in \{0,1,2\}^n$.
Define
$M^{\mathfrak D}f$
as the maximal function
of $f \in L^0({\mathbb R}^n)$
with respect to 
${\mathfrak D}$.
That is,
\[
M^{\mathfrak D}f(x)=\sup_{Q \in {\mathfrak D}}
\frac{\chi_Q(x)}{|Q|}\int_Q|f(y)|{\rm d}y
\quad (x \in {\mathbb R}^n).
\]
For $R_0>0$,
let 
$M^{\mathfrak D}_{\le R_0}f$
be the maximal function
of $f \in L^0({\mathbb R}^n)$,
where the supremum is taken over all cubes 
$Q \in {\mathfrak D}$
with $\ell(Q) \le R_0$,
while
$M^{\mathfrak D}_{\ge R_0}f$
stands for
the maximal function
of $f \in L^0({\mathbb R}^n)$
with respect to 
${\mathfrak D}$,
where the supremum is taken over all cubes 
$Q \in {\mathfrak D}$
with $\ell(Q) \ge R_0$.
Thus,
\begin{align*}
M^{\mathfrak D}_{\le R_0}f(x)&=\sup_{Q \in {\mathfrak D}, \ell(Q) \le R_0}
\frac{\chi_Q(x)}{|Q|}\int_Q|f(y)|{\rm d}y\\
M^{\mathfrak D}_{\ge R_0}f(x)&=\sup_{Q \in {\mathfrak D}, \ell(Q) \ge R_0}
\frac{\chi_Q(x)}{|Q|}\int_Q|f(y)|{\rm d}y.
\end{align*}

Let
$p(\cdot)$ be a variable exponent.
We define the index $p_E$ by
\begin{equation}\label{eq:211029-121}
\frac{1}{p_E}=\frac{1}{|E|}\int_E\frac{{\rm d}x}{p(x)}
\end{equation}
for all measurable sets $E$ with $|E|>0$.
We also define the norm $\|\cdot\|_{L^{p(\cdot)}(E)}$ by
\[
\|f\|_{L^{p(\cdot)}(E)}=\|\chi_{E}f\|_{L^{p(\cdot)}}
\]
for all measurable functions $f$.
Although
it is an abuse of the notation,
we write $L^{p(\cdot)}(w)$ and $L^{p(\cdot)}(E)$
for a weight $w$ and a measurable set $E$.

\begin{definition}
Let
$p(\cdot) \in {\rm L H}_0 \cap {\rm L H}_\infty \cap {\mathcal P}$.
Define
$\tilde{A}_{p(\cdot)}^{\mathfrak D}$
as the class of weights
$w$ satisfying
\[
[w]_{\tilde{A}_{p(\cdot)}^{\mathfrak D}}
\equiv 
\sup_{Q \in {\mathfrak D}}
|Q|^{-p_Q}
\|w\|_{L^1(Q)}
\|w^{-1}\|_{L^{p'(\cdot)/p(\cdot)}(Q)}<\infty.
\]
\end{definition}

Recall \cite{NS-local} shows that
$M^{\mathfrak D}$ is bounded on $L^{p(\cdot)}(w)$
if and only if $w \in A_{p(\cdot)}^{\mathfrak D}$.
Here, we show that $\tilde{A}_{p(\cdot)}^{\mathfrak D}$
enjoys the same property.
\begin{theorem}\label{thm:9.3}
The maximal operator
$M^{\mathfrak D}$ is bounded on $L^{p(\cdot)}(w)$
if and only if $w \in \tilde{A}_{p(\cdot)}^{\mathfrak D}$.
This is equivalent to
$\tilde{A}_{p(\cdot)}^{\mathfrak D}=A_{p(\cdot)}^{\mathfrak D}$.
\end{theorem}

The following property 
has been frequently used in this paper.
\begin{corollary}\label{cor:9.4}
If $q(\cdot) \ge p(\cdot)$,
then $A_{q(\cdot)}^{\mathfrak D} \supset A_{p(\cdot)}^{\mathfrak D}$.
\end{corollary}

We remark that
Proposition \ref{prop:211027-111} follows from the corresponding assertion
to the generalized dyadic grid ${\mathfrak D}$.
Corollary \ref{cor:9.4}
remains true for other grids.
Due to Lemma \ref{lem:210803-1} below
as well as Theorem \ref{thm:9.3},
$A_{q(\cdot)}^{\mathfrak D}= 
\tilde{A}_{q(\cdot)}^{\mathfrak D} \supset \tilde{A}_{p(\cdot)}^{\mathfrak D}= A_{p(\cdot)}^{\mathfrak D}$, which proves 
Corollary \ref{cor:9.4}.
Thus, along with the technique of constructing a weight
in $A_{p(\cdot)}^{\mathfrak D}$ from $A_{p(\cdot)}^{\rm loc}$,
this relation of weights
proves Proposition \ref{prop:211027-111}.
Another corollary of 
Corollary \ref{cor:9.4} and
(\ref{eq:220205-71})
is the monotonicity
of the class of $A_{p(\cdot)}$ 
considered in \cite{CFN-2012}.

\begin{corollary}\label{cor:9.4a}
If $q(\cdot) \ge p(\cdot)$,
then $A_{q(\cdot)} \supset A_{p(\cdot)}$.
\end{corollary}
\subsection{Sufficiency in Theorem \ref{thm:9.3}}

We also need the local versions of $A_{p(\cdot)}^{\mathfrak D}$:
For a measurable subset $E$ of ${\mathbb R}^n$,
we define
\[
{\mathfrak D}(E)\equiv \{Q \in {\mathfrak D}\,:\,Q \subset E\}.
\]
\begin{definition}
Let $E$ be a subset of ${\mathbb R}^n$.
\begin{enumerate}
\item
Define
$\tilde{A}_{p(\cdot)}^{\mathfrak D}(E)$
as the class of weights
$w$ satisfying
\[
[w]_{\tilde{A}_{p(\cdot)}^{\mathfrak D}(E)}
\equiv 
\sup_{Q \in {\mathfrak D}(E)}
|Q|^{-p_Q}
\|w\|_{L^1(Q)}
\|w^{-1}\|_{L^{p'(\cdot)/p(\cdot)}(Q)}<\infty.
\]
\item
Define
$
A_p^{\mathfrak D}(E)
$
for $E \in {\mathfrak D}$ and $1<p<\infty$
analogously.
\end{enumerate}
\end{definition}

The following lemma is easy to prove:
\begin{lemma}{\rm \cite[Lemma 3.1]{DiHapre}}
\label{lem:210803-1}
Let
$p(\cdot),q(\cdot)\in {\mathcal P} \cap {\rm LH}_0 \cap {\rm LH}_\infty$.
Assume
$p(\cdot) \le q(\cdot)$ everywhere.
Then
there exists a constant $C_0>0$,
which depends on
$p_\pm$,
$q_\pm$,
$c_*(p(\cdot))$,
$c_*(q(\cdot))$,
$c^*(p(\cdot))$
and
$c^*(q(\cdot))$,
 such that
 {
\[
[w]_{\tilde{A}_{q(\cdot)}^{\mathfrak D}}
\le 
C_0
[w]_{\tilde{A}_{p(\cdot)}^{\mathfrak D}}.
\]
}
\end{lemma}

\begin{proof}
Fix $Q \in {\mathfrak D}$.
Let $\alpha(\cdot)$ satisfy
\[
\frac{1}{\alpha(\cdot)}
=
\frac{q(\cdot)}{q'(\cdot)}
-
\frac{p(\cdot)}{p'(\cdot)}
=
q(\cdot)-p(\cdot) \ge0.
\]
Since $\alpha(\cdot) \in {\rm L H}_0 \cap {\rm L H}_\infty$,
we have
\[
\|\chi_Q\|_{L^{\alpha(\cdot)}} \sim |Q|^{\frac{1}{\alpha_Q}}
=
|Q|^{\frac{1}{|Q|}\int_Q q(x){\rm d}x-\frac{1}{|Q|}\int_Q p(x){\rm d}x}
\sim
|Q|^{q_Q-p_Q} \sim 1.
\]
Here, the third equivalence follows from \cite[Lemma 2.1]{DiHapre}
(see also Lemma \ref{lem:190613-3}).
Thus,
from the H\"{o}lder inequality for Lebesgue spaces with variable exponents,
we have
\[
[w]_{\tilde{A}_{q(\cdot)}^{\mathfrak D}}
\le C_0
[w]_{\tilde{A}_{p(\cdot)}^{\mathfrak D}}.
\]
Thus, the proof is complete.
\end{proof}

We have a local counterpart.
\begin{corollary}\label{cor:211025-1}
Let
$p(\cdot),q(\cdot)\in {\mathcal P} \cap {\rm LH}_0 \cap {\rm LH}_\infty$
and
let $R \in {\mathfrak D}$.
Assume
$p(\cdot) \le q(\cdot)$ everywhere.
Then
there exists a constant $C_0>0$,
which depends on
$p_\pm(R)$,
$q_\pm(R)$,
$c_*(p(\cdot)|R)$,
$c_*(q(\cdot)|R)$,
$c^*(p(\cdot)|R)$
and
$c^*(q(\cdot)|R)$
 such that
{$
[w]_{\tilde{A}_{q(\cdot)}^{\mathfrak D}(R)}
\le C_0
[w]_{\tilde{A}_{p(\cdot)}^{\mathfrak D}(R)}
$
}
for all $R \in {\mathfrak D}$.
\end{corollary}

Although
the following estimate is crude, it is important.
\begin{lemma}{\rm \cite[Lemma 3.3]{DiHapre}}
\label{lem:210803-3}
Let
$p(\cdot)\in {\mathcal P} \cap {\rm LH}_0 \cap {\rm LH}_\infty$.
If
$w \in \tilde{A}_{p(\cdot)}^{\mathfrak D}$,
then
\[
w(Q) \gtrsim \min\left(1,\frac{|Q|}{|S|}\right)^{p_+}w(S)
\]
for all cubes $Q,S \in {\mathfrak D}$ with $Q \cap S\ne\emptyset$.
\end{lemma}

\begin{proof}
When $Q,S \in {\mathfrak D}$ with $Q \cap S\ne\emptyset$,
there are four cases;
\begin{enumerate}
\item $Q \subset S$, 
\item $Q \supset S$, 
\item $|Q|\le |S|$ and there is a unique cube $\widetilde{S} \supset S$ such that 
$|\widetilde{S}|=2^n|S|$ and $Q \subset \widetilde{S}$,
\item
$|S|\le |Q|$ and there is a unique cube $\widetilde{Q} \supset Q$ such that 
$|\widetilde{Q}|=2^n|Q|$ and $S \subset \widetilde{Q}$.
\end{enumerate}

First, we assume (1) and (3).
We know that {$w \in \tilde{A}_{p_+}^{\mathfrak D}$}
thanks to Lemma \ref{lem:210803-1}.
Since
\[
M^{\mathfrak D}\chi_Q \gtrsim 
\frac{|Q|}{|S|}\chi_{\widetilde{S}},
\]
we have
\begin{align*}
w(Q)
=
\int_{{\mathbb R}^n}\chi_Q(z)^{p_+}w(z){\rm d}z
&\gtrsim
\int_{{\mathbb R}^n}M^{\mathfrak D}\chi_Q(z)^{p_+}w(z){\rm d}z\\
&\gtrsim
\int_{{\mathbb R}^n}\chi_{\widetilde{S}}(z)
\min\left(1,\frac{|Q|}{|S|}\right)^{p_+}w(z){\rm d}z\\
&=\min\left(1,\frac{|Q|}{|S|}\right)^{p_+}w(\widetilde{S})
\ge \min\left(1,\frac{|Q|}{|S|}\right)^{p_+}w(S).
\end{align*}
If (2) and (4) hold,
then this is clear since $w$ is a non-negative function.
When we consider the case (4), note that using the above argument, we can show that 
$w(Q) \gtrsim w(\tilde{Q})$.
Thus, the proof is complete.
\end{proof}

Since we assume that $w$ is (globally) in
$\tilde{A}_{p(\cdot)}^{\mathfrak D}$,
$w$ has
at most polynomial growth.
Here and below,
we let
$Q^\dagger_k \in {\mathfrak D}$
be the unique cube in ${\mathfrak D}_k$ containing $0$ {for $k\in {\mathbb Z}$}.
It is noteworthy that
$\{Q_{-k}^\dagger\}_{k \in {\mathbb Z}}$ is an increasing family of cubes.
\begin{corollary}\label{cor:210804-1}
Let $w \in \tilde{A}_{p(\cdot)}^{\mathfrak D}$ and $Q \in {\mathfrak D}$.
Assume that $Q \ne Q^\dagger_{-k}$ for any $k \in {\mathbb N}$.
Then
$
w(Q) \lesssim (1+|x|)^{p_+n}
$
for all $x \in Q$.
\end{corollary}

\begin{proof}
Fix $Q \in {\mathfrak D}$.
Let $k_0 \in {\mathbb Z}$ be the largest integer such that
$Q \subset Q^\dagger_{k_0}$.
If $k_0 \ge -1$,
then
\[
w(Q) \le w(Q^\dagger_{k_0}) \le w(Q^\dagger_{-1}) \lesssim 1 \lesssim (1+|x|)^{p_+n}.
\]
If $k_0 \le -2$,
then $Q \subsetneq Q^\dagger_{k_0}$
by assumption.
Note that
$Q$ and $Q^\dagger_{k_0+2}$ do not intersect; otherwise
$Q \subset Q^\dagger_{k_0+2}$ or $Q \supsetneq Q^\dagger_{k_0+2}$.
The first possibility never occurs in view of the maximality of $k_0$.
Meanwhile, since $Q \subsetneq Q^\dagger_{k_0}$ 
and $Q, Q^\dagger_{k_0} \in {\mathfrak D}$,  
$|Q^\dagger_{k_0}| \ge 4^n|Q|$ 
and $|Q^\dagger_{k_0+2}|=\frac1{4^n}|Q^\dagger_{k_0}| \ge |Q|$.
Hence, the latter case is also impossible.
Thus,
$|x| \sim \ell(Q^\dagger_{k_0})$ for all $x \in Q (\subset {\mathbb R}^n \setminus Q^\dagger_{k_0+2})$.
Consequently,
thanks to Lemma \ref{lem:210803-3}
\[
w(Q) \le w(Q^\dagger_{k_0}) \lesssim 
\left(\frac{|Q^\dagger_{k_0}|}{|Q^\dagger_0|}\right)^{p_+}w(Q^\dagger_0) \lesssim
(1+|x|)^{p_+n},
\]
as required.
\end{proof}

We have various quantities equivalent
to $\|\chi_Q\|_{L^{p(\cdot)}(w)}$
if the cube
$Q \in {\mathfrak D}$ is small.
\begin{lemma}\label{lem:210804-1}
Let $w \in \tilde{A}_{p(\cdot)}^{\mathfrak D}$
and $Q \in {\mathfrak D}_k$ with $k \ge -1$.
Then
\begin{equation}\label{eq:211101-1}
\|\chi_Q\|_{L^{p(\cdot)}(w)}
\sim
w(Q)^{\frac{1}{p_+(Q)}}
\sim
w(Q)^{\frac{1}{p_-(Q)}}
\sim
w(Q)^{\frac{1}{p(x)}}
\sim
w(Q)^{\frac{1}{p_Q}}
\end{equation}
for all $x \in Q$.
\end{lemma}

\begin{proof}
We concentrate on the proof of
$w(Q)^{\frac{1}{p_+(Q)}}
\sim
w(Q)^{\frac{1}{p_-(Q)}}$;
other equivalences are clear, since
other quantities are between
$
w(Q)^{\frac{1}{p_+(Q)}}
$
and
$
w(Q)^{\frac{1}{p_-(Q)}}
$.

First,
assume that $Q \cap Q_{-1}^\dagger=\emptyset$.
In this case, we choose the largest integer
$k_0 \in {\mathbb Z}$ such that
$Q \subset Q_{k_0}^\dagger$.
We have ${k_0} \le -2$.
Then
$Q$ and $Q_{k_0+2}^\dagger$
do not intersect.
Otherwise
we have either
$Q \subset Q_{k_0+2}^\dagger$
or
$Q_{k_0+2} \subsetneq Q \subsetneq Q_{k_0}^\dagger$.
As in the proof for Corollary \ref{cor:210804-1},
neither of these cases occurs.

Since $w \in \tilde{A}_{p(\cdot)}^{\mathfrak D}$
and $|x| \sim \ell(Q_{k_0}^\dagger) \sim \ell(Q_{k_0+2}^\dagger)\gtrsim 1$ 
for all $x \in Q$,
\begin{equation}\label{eq:220420-1}
\left(\frac{r}{1+|x|}\right)^{n p_+} w(Q_{-1}^\dagger)
\le
\left(\frac{r}{1+|x|}\right)^{n p_+}
w(Q_{k_0}^\dagger)
\lesssim
w(Q)
\le
w(Q_{k_0}^\dagger)
\lesssim
(1+|x|)^{n p_+} w(Q_{-1}^\dagger)
\end{equation}
for all $x \in Q$
thanks to Lemma \ref{lem:210803-3},
where $r=-\log_2\ell(Q)$.
If we use the
global/local $\log$-H\"{o}lder conditions,
then 
$
w(Q)^{\frac{1}{p_+(Q)}}
\sim
w(Q)^{\frac{1}{p_-(Q)}}$
and hence
(\ref{eq:211101-1}).

Let us deal with the case $Q \cap Q^\dagger_{-1} \ne \emptyset$.
Then we have $Q \subset Q^\dagger_{-2}$.
In fact, if $Q=Q_k^\dagger (k \ge -1)$, this claim is clear since $\{Q_j^\dagger\}_{j}$ is decreasing.
Otherwise, let $Q \in {\mathfrak D}_k \setminus \{Q_{k}^\dagger\} (k \ge 0)$.
If $k$ is even, then $Q \subset Q_0^\dagger$ thanks to the construction of the dyadic grids.
In particular $Q \subset Q_{-2}^\dagger$.
Similarly, if $k$ is odd, then 
$Q \subset Q_{-1}^\dagger \subset Q_{-2}^\dagger$.
Hence
\[
w(Q^\dagger_{-2}) \ge
w(Q) \gtrsim \left(\frac{|Q|}{|Q^\dagger_{-2}|}\right)^{p_+}w(Q^\dagger_{-2})
\]
thanks to Lemma \ref{lem:210803-3}.
Thus,
\[
w(Q)^{\left|\frac{1}{p_+(Q)}-\frac{1}{p_-(Q)}\right|}
\gtrsim
w(Q^\dagger_{-2})^{\left|\frac{1}{p_+(Q)}-\frac{1}{p_-(Q)}\right|}
\left(\frac{|Q^\dagger_{-2}|}{|Q|}\right)^{-p_+\left|\frac{1}{p_+(Q)}-\frac{1}{p_-(Q)}\right|}
\gtrsim 1
\]
and
\[
w(Q)^{\left|\frac{1}{p_+(Q)}-\frac{1}{p_-(Q)}\right|}
\le
w(Q^\dagger_{-2})^{\left|\frac{1}{p_+(Q)}-\frac{1}{p_-(Q)}\right|}
\lesssim 1.
\]
Thus,
$
w(Q)^{\frac{1}{p_+(Q)}}
\sim
w(Q)^{\frac{1}{p_-(Q)}}
$.
\end{proof}
Using a similar argument to \cite[Proposition 3.8]{DiHapre},
we have the following equivalence:
\begin{lemma}\label{lem:210903-1}
If 
$p(\cdot) \in {\mathcal P} \cap {\rm L H}_0 \cap {\rm L H}_\infty$
and
$w \in \tilde{A}_{p(\cdot)}^{\mathfrak D}$,
then,
for all $Q \in {\mathfrak D}$,
\[
|Q|^{-p_Q}w(Q)\left\|w^{-1}\right\|_{L^{\frac{p'(\cdot)}{p(\cdot)}}(Q)}
\sim
\frac{w(Q)}{|Q|}\left(\frac{\sigma(Q)}{|Q|}\right)^{p_Q-1}.
\]
\end{lemma}

\begin{proof}
Let
$w \in \tilde{A}_{p(\cdot)}^{\mathfrak D}$ and suppose that
$Q \in \bigcup\limits_{k=0}^\infty {\mathfrak D}_k$.
By the definition of $\|w\|_{\tilde{A}_{p(\cdot)}^{\mathfrak D}}$,
we have
\[
\frac{w(Q)}{|Q|^{p_Q}}\|w^{-1}\|_{L^{\frac{p'(\cdot)}{p(\cdot)}}(Q)}
\le 
\|w\|_{\tilde{A}_{p(\cdot)}^{\mathfrak D}}.
\]
Due to Lemma \ref{lem:210804-1},
$w(Q)^{\frac{1}{p_Q}} \sim \|\chi_Q\|_{L^{p(\cdot)}(w)} 
= \|w^{\frac{1}{p(\cdot)}}\|_{L^{p(\cdot)}(Q)}$.
By virtue of the H\"{o}lder inequality, we have
\begin{align*}
|Q|=
\int_Q w(y)^{\frac{1}{p(y)}}w(y)^{-\frac{1}{p(y)}}{\rm d}y
&\le
2\|w^{\frac{1}{p(\cdot)}}\|_{L^{p(\cdot)}(Q)}
\|w^{-\frac{1}{p(\cdot)}}\|_{L^{p'(\cdot)}(Q)}\\
&\lesssim
\|w(Q)^{\frac{1}{p_Q}}w^{-\frac{1}{p(\cdot)}}\|_{L^{p'(\cdot)}(Q)}.
\end{align*}
This means that
\[
\int_Q 
\left(\frac{w(Q)^{\frac{1}{p_Q}}}{|Q|}\right)^{p'(y)}
w(y)^{-\frac{p'(y)}{p(y)}}{\rm d}y \gtrsim 1.
\]
Again, using Lemma \ref{lem:210804-1}, we have
$
w(Q)^{\frac{1}{p_Q}} \sim w(Q)^{\frac{1}{p(y)}}$
for all $y \in Q$.
Since $|Q| \le 1$,
we obtain
\[
\int_Q 
\left(\frac{w(Q)}{|Q|^{p_Q}}\right)^{\frac{p'(y)}{p(y)}}
w(y)^{-\frac{p'(y)}{p(y)}}{\rm d}y \gtrsim 1.
\]
Therefore,
\[
\frac{w(Q)}{|Q|^{p_Q}}
\left\|w^{-1}\right\|_{L^{\frac{p'(\cdot)}{p(\cdot)}}(Q)}
\gtrsim 1.
\]
From the definition of the quantity 
$[w]_{\tilde{A}_{p(\cdot)}^{\mathfrak D}}$,
we conclude
\[
\frac{w(Q)}{|Q|^{p_Q}}
\left\|w^{-1}\right\|_{L^{\frac{p'(\cdot)}{p(\cdot)}}(Q)}
\sim 1.
\]
Hence
\[
\left\|w^{-1}\right\|_{L^{\frac{p'(\cdot)}{p(\cdot)}}(Q)}
\sim
\frac{|Q|^{p_Q}}{w(Q)}.
\]
This implies that by Lemma \ref{lem:210804-1}
\[
\sigma(Q)^{p_Q-1}
\sim
\frac{|Q|^{p_Q}}{w(Q)}.
\]
Thus, if $|Q| \le 1$, then we have the desired property.

This means that $\sigma \in A_\infty^{{\rm loc},{\mathfrak D}}$.
Furthermore, if $R$ satisfies $|R|=1$, then by Remark \ref{rem:200807-1}, 
and
\[
\min(\|\chi_{R}\|_{L^{\frac{1}{p(\cdot)-1}}(\sigma)}^{\frac{1}{p_+(R)-1}}, 
\|\chi_{R}\|_{L^{\frac{1}{p(\cdot)-1}}(\sigma)}^{\frac{1}{p_-(R)-1}})
\le
\sigma(R)
\le
\max(\|\chi_{R}\|_{L^{\frac{1}{p(\cdot)-1}}(\sigma)}^{\frac{1}{p_+(R)-1}}, 
\|\chi_{R}\|_{L^{\frac{1}{p(\cdot)-1}}(\sigma)}^{\frac{1}{p_-(R)-1}}),
\]
we have
\[
\|\chi_R\|_{L^{\frac{p'(\cdot)}{p(\cdot)}}(\sigma)} \sim \sigma(R)^{p_\infty-1}.
\]
By the localization property (Lemma \ref{lem:190627-11}),
we have
\[
\|w^{-1}\|_{L^{\frac{p'(\cdot)}{p(\cdot)}}(Q)}
=
\|\chi_Q\|_{L^{\frac{p'(\cdot)}{p(\cdot)}}(\sigma)}
\sim
\sigma(Q)^{p_\infty-1}
\]
for all cubes $Q$ with $|Q| \ge 1$.
Since we know that
$|Q|^{\frac{1}{p_\infty}-\frac{1}{p_Q}} \sim 
w(Q)^{\frac{1}{p_\infty}-\frac{1}{p_Q}} \sim 1$,
we see that
$\sigma(Q)^{\frac{1}{p_\infty}-\frac{1}{p_Q}} \sim 1$.
Thus, the proof is complete.

\end{proof}

Under these preparations, we establish 
the boundedness of the local maximal operator.

\begin{lemma}{\rm \cite[Lemma 5.1]{DiHapre}}
\label{lem:210803-2}
Let
$w \in \tilde{A}_{p(\cdot)}^{\mathfrak D}$.
Then there exists $r_0=r_0([w]_{\tilde{A}_{p(\cdot)}^{\mathfrak D}},p(\cdot)) \in (0,1)$ such that
$-\log_2 r_0$ is an integer and that
\[
\|\chi_{Q}M^{\mathfrak D}_{\le r_0}[\chi_{Q} f]\|_{L^{p(\cdot)}(w)}
\lesssim 
\|\chi_{3Q}f\|_{L^{p(\cdot)}(w)}
\]
for all $f \in L^{p(\cdot)}(w)$ and $Q \in {\mathfrak D}_j$
with $\ell(Q)<r_0\le\frac12$.
Here $j=1-\log_2r_0$.
\end{lemma}

\begin{proof}
For now, let $r_0 \in (0,1)$ be small enough. 
We will specify it shortly.
Then, there exists 
$j \in {\mathbb N} 
$ 
such that $2^{-j} < r_0 \le 2^{-j+1}$.
Fix $Q \in {\mathfrak D}_j$.
Note that $\ell(Q)<r_0$ by the definition of $j$.
Let
$C_0$ be the constant from Lemma \ref{lem:210803-1}.
Write
$c_1\equiv C_0
[w]_{\tilde{A}_{p(\cdot)}^{\mathfrak D}}$.
Then there exist $c_2>0$ and
$\varepsilon \in (0,1)$,
which is
independent of $S\in{\mathfrak D}$ such that
$[\sigma]_{A_q(S)} \le c_1$ implies
$[\sigma]_{A_{q-\varepsilon}(S)} \le c_2$
for all
$\sigma \in A_q(S)$ and
$q \in [p_-,p_+
+1]$
by the openness property established
by 
Hyt\"{o}nen
and
P\'{e}rez
\cite{HyPe13}
(see also Lemma \ref{Lemma-Muckenhoupt}).

Next
using the $\log$-H\"{o}lder continuity, we can choose
$r_0<\frac12n^{-1/2}$
so that
$p_+(3S)-\varepsilon<p_-(3S)$
for all $S \in {\mathfrak D}_{-\log_2 r_0}$
and
$j\equiv-\log_2 r_0+1 \in {\mathbb N}$.
By virtue of Lemma \ref{lem:210803-1}, 
\[
[w]_{A_{p_+(3Q)}^{\mathfrak D}(3Q)} 
\le
C_0
[w]_{\tilde{A}_{p(\cdot)}^{\mathfrak D}(3Q)} 
\le
C_0 [w]_{\tilde{A}_{p(\cdot)}^{\mathfrak D}} = c_1.
\]
By the property of $c_2$,
we have
\[
[w]_{A_{p_-(3Q)}^{\mathfrak D}(3Q)} 
\le
C_0
[w]_{A_{p_+(3Q)-\varepsilon}^{\mathfrak D}(3Q)}
\le c_2C_0.
\]
Let
$f \in L^{p(\cdot)}(w)$
with
$\|\chi_{3Q}f\|_{L^{p(\cdot)}(w)} \le1$.
Set
$g\equiv\chi_{Q}f$ and 
\begin{equation}\label{eq:211029-31}
q(\cdot)\equiv \frac{p(\cdot)}{p_-({3Q})}.
\end{equation}
Fix
$x \in Q$ and
choose a cube $\displaystyle R \in \bigcup_{k=j}^{\infty}{\mathfrak D}_k$ with
$x \in R$. 
Note that $\ell(R)\le 2^{-j}<r_0$.
Then
for all $\beta>0$,
\begin{align*}
\left(
\frac{1}{|R|}\int_R |g(y)|{\rm d}y
\right)^{q(x)}
&\le
\left(
\frac{1}{|R|}\int_R |g(y)|^{q_-(R)}{\rm d}y
\right)^{\frac{q(x)}{q_-(R)}}
=
\left(
\frac{1}{\beta|R|}\int_R \beta|g(y)|^{q_-(R)}{\rm d}y
\right)^{\frac{q(x)}{q_-(R)}}
\end{align*}
by H\"{o}lder's inequality.
By the inequality $t \le 1+t^{\frac{q(\cdot)}{q_-(R)}}$, we have
\begin{align*}
\left(
\frac{1}{|R|}\int_R |g(y)|{\rm d}y
\right)^{q(x)}
&\le
\left(
\frac{1}{\beta|R|}\int_R 
\left(1+\beta^{\frac{q(y)}{q_-(R)}}|g(y)|^{q(y)}\right){\rm d}y
\right)^{\frac{q(x)}{q_-(R)}}\\
&=
\left(\frac{1}{\beta}+
\frac{1}{|R|}\int_R \beta^{\frac{q(y)}{q_-(R)}-1}|g(y)|^{q(y)}{\rm d}y
\right)^{\frac{q(x)}{q_-(R)}}.
\end{align*}
We choose
$\beta\equiv \max(1,w(3Q)^{\frac{1}{p_-(3Q)}})$.
We suppose that $0 \notin 5Q$.
Decompose equally $3Q$ into $3^n$ cubes
$Q_1,Q_2,\ldots,Q_{3^n}$.
By Corollary \ref{cor:210804-1}, we have
\begin{equation*}
w(Q_k) 
\lesssim (1+|y|)^{p_+n}
\end{equation*}
for all $y \in Q_k$ and $k=1,2, \ldots, 3^n$.
Note that for all $y \in Q_k$ and $z \in 3Q$, $|y| \sim |z|$.
Thus, 
\begin{equation}\label{eq:220129-1}
w(3Q) 
=
\sum_{k=1}^{3^n} w(Q_k)
\lesssim
(1+|y|)^{p_+n}
\sim
(1+|z|)^{p_+n}
\end{equation}
for all $z \in 3Q$.
Meanwhile, if $5Q \ni 0$, then 
$w(3Q) \lesssim w([-10,10]^n) \lesssim 1\lesssim
(1+|z|)^{p_+n}$
for all $z \in 3Q$.
Since $Q \in {\mathfrak D}_j$ for $j \in {\mathbb N}$, $R \subset 3Q$. 
Hence, 
estimate (\ref{eq:220129-1}) still holds for all $y \in R$.

Thus,
since $q(\cdot) \in {\rm L H}_\infty$,
$$
\beta^{\frac{q(y)}{q_-(R)}-1}
=\max(1,w(3Q)^{\frac{1}{p_-(3Q)}})^{\frac{q(y)}{q_-(R)}-1}
\lesssim ((1+|y|)^{p_+n})^{q(y)-q_-(R)}
\lesssim 1
$$
for all $y \in R$,
where the implicit constant depends on $p(\cdot)$.
Since 
$q(x) \ge q_-(R)$
and 
$\beta \ge 1$,
we obtain
\begin{align}
\lefteqn{
\left(
\frac{1}{|R|}\int_R |g(y)|{\rm d}y
\right)^{q(x)}
}
\label{eq:220121-2}\\
\nonumber
&\lesssim
\frac{1}{\beta}+
\left(
\frac{1}{|R|}\int_R |g(y)|^{q(y)}{\rm d}y
\right)^{\frac{q(x)}{q_-(R)}}\\
\nonumber
&=
\min(1,w(3Q)^{-\frac{1}{p_-(3Q)}})+
\left(
\frac{1}{|R|}\int_R |g(y)|^{q(y)}{\rm d}y
\right)^{\frac{q(x)}{q_-(R)}-1}
\times
\frac{1}{|R|}\int_R |g(y)|^{q(y)}{\rm d}y.
\end{align}
Thanks to the Young inequality,
the definition of the quantity
$A_{p_-(3Q)}^{\mathfrak D}(3Q)$
and the fact that $\|\chi_{3Q}f\|_{L^{p(\cdot)}(w)} \le 1$,
we have
\begin{align*}
\int_R |g(y)|^{q(y)}{\rm d}y
&\le
\int_R |g(y)|^{p(y)}w(y){\rm d}y
+
\int_R w(y)^{-\frac{1}{p_-(3Q)-1}}{\rm d}x\\
&=
\int_R |g(y)|^{p(y)}w(y){\rm d}y
+
\left\{\left(\int_R w(y)^{-\frac{1}{p_-(3Q)-1}}{\rm d}x\right)^{p_-(3Q)-1}\right\}^{\frac{1}{p_-(3Q)-1}}\\
&\lesssim
1+
([w]_{A_{p_-(3Q)}^{\mathfrak D}(3Q)}|R|^{p_-(3Q)}w(R)^{-1})^{-\frac{1}{p_-(3Q)-1}}.
\end{align*}
Recall that $\ell(R) \le \frac12$.
Since $\frac{q(x)}{q_-(R)}-1 \ge 0$,
thanks to the $\log$-H\"{o}lder continuity
of $q(\cdot)$ and Lemma \ref{lem:210804-1},
we have
\[
\left(
\frac{1}{|R|}\int_R |g(y)|^{q(y)}{\rm d}y
\right)^{\frac{q(x)}{q_-(R)}-1}\lesssim
\left(
1+
([w]_{A_{p_-(3Q)}^{\mathfrak D}(3Q)}|R|^{p_-(3Q)}w(R)^{-1})^{-\frac{1}{p_-(3Q)-1}}
\right)^{\frac{q(x)}{q_-(R)}-1}\sim 1.
\]
From
(\ref{eq:220121-2}), we obtain
\[
\left(
\frac{1}{|R|}\int_R |g(y)|{\rm d}y
\right)^{q(x)}
\lesssim
\min(1,w(3Q)^{-\frac{1}{p_-(3Q)}})+
\frac{1}{|R|}\int_R |g(y)|^{q(y)}{\rm d}y.
\]
Recall that $R$ is a cube with $\ell(R) \le r_0 \le 1/2$.
Therefore from the definition of $M^{\mathfrak D}_{\le r_0}$,
we deduce
\begin{equation}\label{eq:220121-3}
M^{\mathfrak D}_{\le r_0}g(x)^{q(x)}
\lesssim
M^{\mathfrak D}_{\le r_0}[|g|^{q(\cdot)}](x)
+
\min(1,w(3Q)^{-\frac{1}{p_-(3Q)}})
\end{equation}
Recall that $q(\cdot)$ is given by $(\ref{eq:211029-31})$.
Inserting the definition of $q(\cdot)$ into
(\ref{eq:220121-3}),
we obtain
\[
M^{\mathfrak D}_{\le r_0}g(x)^{p(x)}
\lesssim
(M^{\mathfrak D}_{\le r_0}[|g|^{q(\cdot)}](x))^{p_-(3Q)}
+
\frac{1}{\max(1,w(3Q))}.
\]
Since $w \in A_{p_-(3Q)}(3Q)$,
integrating the above inequality for the measure $w(x){\rm d}x$
over $Q$
gives
\begin{align*}
\int_{Q} 
M^{\mathfrak D}_{\le r_0}g(x)^{p(x)} w(x){\rm d}x
&\lesssim
\int_{Q}
\left[
(M^{\mathfrak D}_{\le r_0}[|g|^{q(\cdot)}](x))^{p_-(3Q)}
+
\frac{1}{\max(1,w(3Q))}
\right] 
w(x){\rm d}x\\
&\lesssim
\int_{Q}
|f(x)|^{q(\cdot)p_-(3Q)}
w(x) {\rm d}x
+1
\sim 1.
\end{align*}
Hence, we have 
$\|\chi_QM^{\mathfrak D}_{\le r_0}g\|_{L^{p(\cdot)}(w)}
=
\|\chi_QM^{\mathfrak D}_{\le r_0}[\chi_Qf]\|_{L^{p(\cdot)}(w)} 
\le 1$, as desired.
\end{proof}

By the localization argument
(see Lemma \ref{lem:190627-11}), 
we can prove $M^{\mathfrak D}_{\le r_0}$ is bounded on $L^{p(\cdot)}(w)$.
\begin{lemma}{\rm \cite[Lemma 5.3]{DiHapre}}
\label{lem:210803-22}
Let
$w \in \tilde{A}_{p(\cdot)}^{\mathfrak D}$.
Then there exists $r_0=r_0([w]_{\tilde{A}_{p(\cdot)}^{\mathfrak D}},p(\cdot)) \in (0,1)$ such that
\[
\|M^{\mathfrak D}_{\le r_0}f\|_{L^{p(\cdot)}(w)}
\lesssim 
\|f\|_{L^{p(\cdot)}(w)}
\]
for all $f \in L^{p(\cdot)}(w)$.
\end{lemma}

Let
$w \in \tilde{A}_{p(\cdot)}^{\mathfrak D}$.
Define
\[
E_k f\equiv \sum_{Q \in {\mathfrak D}_k}\chi_Qm_Q(f)
\quad
(k \in {\mathbb Z})
\]
for $f \in L^1_{\rm loc}({\mathbb R}^n)$.
We harvest a corollary
of Lemmas \ref{lem:210803-2} and \ref{lem:210803-22}.
\begin{corollary}\label{cor:210823-1}
Let
$w \in \tilde{A}_{p(\cdot)}^{\mathfrak D}$.
If $k \gg 1$,
then $E_k$ is bounded on $L^{p(\cdot)}(w)$.
\end{corollary}

\begin{proof}
Simply observe that
$|E_k f| \le M^{\mathfrak D}_{\le r_0}f$
for $f \in L^1_{\rm loc}({\mathbb R}^n)$. 
\end{proof}

We obtain another corollary
of Lemmas \ref{lem:210803-2} and \ref{lem:210803-22}.
\begin{corollary}\label{cor:220121-1}
Let
$w \in \tilde{A}_{p(\cdot)}^{\mathfrak D}$
and 
$r_0 \in (0,1)$ be the same as in Lemma \ref{lem:210803-2}.
Then
\begin{equation}\label{eq:220121-4}
\|M^{\mathfrak D}_{\ge r_0}[\chi_{R}f]\|_{L^{p(\cdot)}(w)}
\lesssim
\|f\|_{L^{p(\cdot)}(w)}
\end{equation}
for all $f \in L^{p(\cdot)}(w)$ and $R \in {\mathfrak D}$ with $\ell(R)=r_0$
with the implicit constant dependent on $r_0$.
\end{corollary}

\begin{proof}


Let $k\equiv1-\log_2r_0$.
Fix $R \in {\mathfrak D}$ with $\ell(R)=r_0$.

Let $x \in Q_k^\dagger$. Then,
a geometric observation shows that
there exists
the smallest cube $Q_{i}^\dagger$ 
such that $|\ell(Q_{i})|\ge r_0, Q^\dagger_k, R \subset Q^\dagger_{i}$ 
and $Q^\dagger_k \cap Q =\emptyset$ 
for $\displaystyle Q \in \bigcup_{j=-\infty}^{i}
({\mathfrak D}_j \setminus \{Q^\dagger_j\})$.
Thus,
\[
\chi_{Q_k^\dagger}(x)
M^{\mathfrak D}_{\ge r_0}[\chi_{R}f](x)
=
\frac{\chi_{Q_k^\dagger}(x)}{|Q^\dagger_{i}|} 
\int_{Q^\dagger_{i}} \chi_{R}(y)|f(y)| {\rm d}y.
\] 
From this pointwise estimate, we have
\begin{align*}
\|\chi_{Q_k^\dagger}
M^{\mathfrak D}_{\ge r_0}[\chi_{R}f]
\|_{L^{p(\cdot)}(w)}
&\le
\left\|\frac{\chi_{Q_k^\dagger}}{|R|} 
\int_{Q^\dagger_{i}} \chi_{R}(y)|f(y)| {\rm d}y\right\|_{L^{p(\cdot)}(w)}\\
&=
\frac{\|\chi_{Q_k^\dagger}\|_{L^{p(\cdot)}(w)}}{\|\chi_{R}\|_{L^{p(\cdot)}(w)}}
\left\|\frac{\chi_{R}}{|R|} 
\int_{Q^\dagger_{i}} \chi_{R}(y)|f(y)| {\rm d}y\right\|_{L^{p(\cdot)}(w)}\\
&\lesssim
\|
M^{\mathfrak D}_{\le r_0}f\|_{L^{p(\cdot)}(w)}\\
&\lesssim
\|f\|_{L^{p(\cdot)}(w)}.
\end{align*}
Thus, we obtain 
\begin{equation}\label{eq:220128-1}
\left\|
\chi_{Q_k^\dagger}
M^{\mathfrak D}_{\ge r_0}[\chi_{Q_k^\dagger}f]
\right\|_{L^{p(\cdot)}(w)}
\le
\left\|
M^{\mathfrak D}_{\le r_0}f
\right\|_{L^{p(\cdot)}(w)}
\lesssim
\|f\|_{L^{p(\cdot)}(w)}
\end{equation}
by Lemma \ref{lem:210803-22}.

Let $x \in {\mathbb R}^n \setminus Q_k^\dagger$.
Then, 
a geometric observation shows that
there exists the largest number $\ell < k$ such that
$x \in Q_{\ell}^\dagger$.
Then, by the maximality of $\ell$, we have $|x| \sim |Q_{\ell}^\dagger|$.
Since $Q_{k}^\dagger \subset Q_{\ell}^\dagger$, we obtain
\begin{align*}
\lefteqn{
\chi_{{\mathbb R}^n \setminus Q_k^\dagger}(x)
M^{\mathfrak D}_{\ge r_0}[\chi_{Q_k^\dagger}f](x)
}\\
&=
\frac{\chi_{Q_{\ell}^\dagger \setminus Q_k^\dagger}(x)}{|Q_{\ell}^\dagger|} 
\int_{Q_{\ell}^\dagger} \chi_{Q_k^\dagger}(y)|f(y)| {\rm d}y
\sim
\frac{\chi_{Q_{\ell}^\dagger \setminus Q_k^\dagger}(x)}{|x|^n} 
\int_{Q_k^\dagger}|f(y)| {\rm d}y
\end{align*}
Then, by H\"older's inequality and Lemma \ref{lem:210804-1}, we have
\begin{align*}
\int_{Q_{k}^\dagger}|f(y)|{\rm d}y
&\lesssim 
\|\chi_{Q_{k}^\dagger}f\|_{L^{p(\cdot)}(w)}
\|w^{-1}\chi_{Q_{k}^\dagger}\|_{L^{p'(\cdot)}(w)}\\
&\le
\|\chi_{Q_{k}^\dagger}\|_{L^{p'(\cdot)}(\sigma)}
\|f\|_{L^{p(\cdot)}(w)}
\lesssim
\sigma(Q_{k}^\dagger)^{\frac{1}{p'_{Q_{k}^\dagger}}}\|f\|_{L^{p(\cdot)}(w)},
\end{align*}
where
$\sigma\equiv w^{-{\frac{1}{p(\cdot)-1}}}$ stands
for the dual weight.
We obtain
\begin{equation*}
\chi_{{\mathbb R}^n \setminus Q_k^\dagger}(x)
M^{\mathfrak D}_{\ge r_0}[\chi_{Q_k^\dagger}f](x)
 \lesssim
 \chi_{{\mathbb R}^n \setminus Q_k^\dagger}(x)
|x|^{-n} \sigma(Q_{k}^\dagger)^{\frac{1}{p'_{Q_{k}^\dagger}}}
\|f\|_{L^{p(\cdot)}(w)}.
\end{equation*}
Thus,
\begin{equation} \label{eq:220128-2}
\left\|
\chi_{{\mathbb R}^n \setminus Q_k^\dagger}
M^{\mathfrak D}_{\ge r_0}[\chi_{Q_k^\dagger}f]
\right\|_{L^{p(\cdot)}(w)}
\lesssim
\||\cdot|^{-n}\chi_{{\mathbb R}^n \setminus Q_k^\dagger}\|_{L^{p(\cdot)}(w)}
\|f\|_{L^{p(\cdot)}(w)}.
\end{equation}



So we will estimate $\||\cdot|^{-n}\chi_{{\mathbb R}^n \setminus Q_k^\dagger}\|_{L^{p(\cdot)}(w)}$ 
using the modular.
Let
$C_0$ be the constant from Lemma \ref{lem:210803-1}.
Write
$c_1\equiv C_0{}^2
[w]_{\tilde{A}_{p(\cdot)}^{\mathfrak D}}$.
Then there exist $c_2>0$ and
$\varepsilon \in (0,1)$
independent of the set $E$ such that
$[\sigma]_{A_q^{\mathfrak D}(E)} \le c_1$ implies
$[\sigma]_{A_{q-\varepsilon}^{\mathfrak D}(E)} \le c_2$
for all
$\sigma \in A_q^{\mathfrak D}(E)$ and
$q \in [p_-,p_+
+1]$
for all sets $E$
again by the openness property established
by 
Hyt\"{o}nen
and
P\'{e}rez
\cite{HyPe13}.
Let $k' \ll -1$ be an integer
so that
$p_+({\mathbb R}^n \setminus Q^\dagger_{k'})-\frac13\varepsilon \le p_\infty$.
Then
$w \in A_{p_\infty+\frac13\varepsilon}({\mathbb R}^n \setminus Q^\dagger_{k'})$
with
$$
[w]_{A_{p_\infty+\frac13\varepsilon}^{\mathfrak D}({\mathbb R}^n \setminus Q^\dagger_{k'})} \le 
C_0[w]_{A_{p_+({\mathbb R}^n \setminus Q^\dagger_{k'})}^{\mathfrak D}({\mathbb R}^n \setminus Q^\dagger_{k'})} \le 
C_0{}^2
[w]_{\tilde{A}_{p(\cdot)}^{\mathfrak D}}
=c_1
$$
from Lemma \ref{lem:210803-1}.
As a result,
\[
[w]_{A_{p_\infty-\frac23\varepsilon}^{\mathfrak D}({\mathbb R}^n \setminus Q^\dagger_{k'})}
\le c_2,
\]
or equivalently,
$w \in A_{p_\infty-\frac23\varepsilon}^{\mathfrak D}({\mathbb R}^n \setminus Q^\dagger_{k'})$.
Thus, we have
\begin{align*}
\int_{{\mathbb R}^n \setminus Q_{k}^\dagger}
\frac{w(x){\rm d}x}{|x|^{n p(x)}}
&\lesssim
\int_{Q_{k'}^\dagger \setminus Q_{k}^\dagger}
\frac{w(x){\rm d}x}{(1+|x|)^{n p(x)}}
+
\int_{{\mathbb R}^n \setminus Q_{k'}^\dagger}
\frac{w(x){\rm d}x}{|x|^{n p(x)}}\\
&\lesssim
\int_{Q_{k'}^\dagger 
}
\frac{w(x){\rm d}x}{(1+|x|)^{n p(x)}}
+
\int_{{\mathbb R}^n \setminus Q_{k'}^\dagger}
\frac{w(x){\rm d}x}{|x|^{n\left(p_\infty-\frac23\varepsilon\right)}}\\
&\lesssim
\int_{Q_{k'}^\dagger 
}
w(x){\rm d}x
+
\int_{{\mathbb R}^n \setminus Q_{k'}^\dagger}
\left[
M^{\mathfrak D}\chi_{[-1,1]^n}(x)
\right]^{p_\infty-\frac23\varepsilon}w(x){\rm d}x\\
&\lesssim
\int_{Q_{k'}^\dagger 
}
w(x){\rm d}x
+
\int_{[-1,1]^n}w(x){\rm d}x 
\lesssim 1.
\end{align*}
Consequently,
$\||\cdot|^{-n}\chi_{{\mathbb R}^n \setminus Q_k^\dagger}\|_{L^{p(\cdot)}(w)} \lesssim 1$.
By
combining estimates (\ref{eq:220128-1}) and (\ref{eq:220128-2}),
 we conclude that (\ref{eq:220121-4}) holds.

\end{proof}

From this corollary, we can obtain the boundedness property 
of $M^{\mathfrak D}_{\ge r_0}$
for the function supported on the cube
with comparative ease.
Next, consider  
this operator 
for the function supported on the outside of the cube.
However, this case is very complicated.
We must prepare some lemmas. 

\begin{lemma}\label{lem:220201-1}
Let $w \in \tilde{A}_{p(\cdot)}^{\mathfrak D}$.
Then
\[
\sup_{R \in {\mathfrak D}, |R \setminus Q_{k'}^\dagger|>0}
|R\setminus Q_{k'}^\dagger|^{-p_\infty}\|w\|_{L^1(R\setminus Q_{k'}^\dagger)}
\|w^{-1}\|_{L^{\frac{1}{p_\infty-1}}(R\setminus Q_{k'}^\dagger)}<\infty
\]
for some $k' \le -1$.
\end{lemma}

For the proof,
we invoke the following fact:
from
\cite[Corollary 3.7]{DiHapre}
(see also Remark \ref{rem:220420-1} below),
$\|\chi_Q\|_{L^{p(\cdot)}(w)}
\sim w(Q)^{\frac{1}{p_Q}}$
for all cubes $Q$ 
as long as $\omega \in A_\infty$ and $p(\cdot) \in 
{\mathcal P}_0 \cap {\rm L H}_0 \cap {\rm L H}_\infty$.
\begin{proof}
Since $w \in \tilde{A}_{p(\cdot)}^{\mathfrak D}$, 
the weight $w$ satisfies the condition  
\[
\sup_{R \in {\mathfrak D}}
|R|^{-p_R}\|w\|_{L^1(R)}\|w^{-1}\|_{L^{\frac{p'(\cdot)}{p(\cdot)}}(R)}
<\infty.
\]
Since
\[
|R|^{-p_R}\sim
|R\setminus Q_{k'}^\dagger|^{-p_{R}}\sim
|R\setminus Q_{k'}^\dagger|^{-p_{R\setminus Q_{k'}^\dagger}}
\]
for any cube $R \in {\mathfrak D}$ such that $R \subset Q_{k'}^\dagger$ fails,
we have
\[
\sup_{R \in {\mathfrak D}, |R \setminus Q_{k'}^\dagger|>0}
|R\setminus Q_{k'}^\dagger|^{-p_{R\setminus Q_{k'}^\dagger}}\|w\|_{L^1(R\setminus Q_{k'}^\dagger)}
\|w^{-1}\|_{L^{\frac{p'(\cdot)}{p(\cdot)}}(R\setminus Q_{k'}^\dagger)}<\infty.
\]
Let $\tau\in (0,1)$ be small enough.
By the H\"{o}lder inequality for variable exponent Lebesgue spaces,
\[
\|w^{-1}\|_{L^{\frac{1}{p_\infty-1+\tau}}(R\setminus Q_{k'}^\dagger)}
\le
2
\|w^{-1}\|_{L^{\frac{1}{p(\cdot)-1}}(R\setminus Q_{k'}^\dagger)}
\|\chi_{R\setminus Q_{k'}^\dagger}\|_{L^{\frac{1}{p_\infty+\tau-p(\cdot)}}}.
\]
By \cite[Corollary 3.7 and Lemma 2.1]{DiHapre}
(see also Lemma \ref{lem:190613-3} and Remark \ref{rem:220420-1} below), 
$\|\chi_{R\setminus Q_{k'}^\dagger}\|_{L^{\frac{1}{p_\infty+\tau-p(\cdot)}}}
\sim
|R\setminus Q_{k'}^\dagger|^{p_\infty+\tau-p_{R\setminus Q_{k'}^\dagger}}$.
Thus,
\begin{align*}
\lefteqn{
\sup_{R \in {\mathfrak D}, |R \setminus Q_{k'}^\dagger|>0}
|R\setminus Q_{k'}^\dagger|^{-p_\infty-\tau}\|w\|_{L^1(R\setminus Q_{k'}^\dagger)}
\|w^{-1}\|_{L^{\frac{1}{p_\infty-1+\tau}}(R\setminus Q_{k'}^\dagger)}
}\\
&\lesssim
\sup_{R \in {\mathfrak D}, |R \setminus Q_{k'}^\dagger|>0}
|R\setminus Q_{k'}^\dagger|^{-p_{R\setminus Q_{k'}^\dagger}}\|w\|_{L^1(R\setminus Q_{k'}^\dagger)}
\|w^{-1}\|_{L^{\frac{p'(\cdot)}{p(\cdot)}}(R\setminus Q_{k'}^\dagger)}
<\infty,
\end{align*}
or equivalently,
\[
\sup_{R \in {\mathfrak D}, |R \setminus Q_{k'}^\dagger|>0}
|R\setminus Q_{k'}^\dagger|^{-\frac{p_\infty+\tau}{p_\infty+\tau-1}}
\|w^{\frac{1}{p_\infty-1+\tau}}\|_{L^{p_\infty+\tau-1}(R\setminus Q_{k'}^\dagger)}
\|w^{-{\frac{1}{p_\infty-1+\tau}}}\|_{L^1(R\setminus Q_{k'}^\dagger)}<\infty.
\]
This means that
$w^{-{\frac{1}{p_\infty-1+\tau}}} \in A_{1+\frac{1}{p_\infty+\tau-1}} \subset A_{1+\frac{1}{p_\infty-1}}$.
By virtue of Lemma \ref{Lemma-Muckenhoupt} 
with $\varepsilon=(4^{n+6}[w]_{A^{\mathfrak D}_{\infty, R^{\times}}})^{-1}$, 
\[
\sup_{R \in {\mathfrak D}, |R \setminus Q_{k'}^\dagger|>0}
|R\setminus Q_{k'}^\dagger|^{-\frac{p_\infty+\tau}{p_\infty+\tau-1}-\frac{1}{1+\varepsilon}}
\|w^{\frac{1}{p_\infty-1+\tau}}\|_{L^{p_\infty+\tau-1}(R\setminus Q_{k'}^\dagger)}
(\|w^{-{\frac{1+\varepsilon}{p_\infty-1+\tau}}}\|_{L^1(R\setminus Q_{k'}^\dagger)})^{\frac{1}{1+\varepsilon}}<\infty.
\]
Since
\[
1+\frac{p_\infty-1+\tau}{1+\varepsilon}<p_\infty,
\]
as long as $\tau$ is small enough,
we obtain 
\[
\sup_{R \in {\mathfrak D}, |R \setminus Q_{k'}^\dagger|>0}
|R\setminus Q_{k'}^\dagger|^{-p_\infty}\|w\|_{L^1(R\setminus Q_{k'}^\dagger)}
\|w^{-1}\|_{L^{\frac{1}{p_\infty-1}}(R\setminus Q_{k'}^\dagger)}<\infty
\]
by H\"{o}lder's inequality.
\end{proof}

Assuming $f=E_k f$,
we obtain some growth information of $f$.
\begin{lemma}\label{lem:220202-2}
Let $r_0$ be the same
as in Lemma \ref{lem:210803-2}.
Suppose that $f \in L^{p(\cdot)}(w)$ satisfies $f=E_k f$
for some $k \in {\mathbb Z}$
such that $2^{-k}<r_0$.
Let
$w \in \tilde{A}_{p(\cdot)}^{\mathfrak D}$.
Then
\begin{equation}\label{eq:220103-1}
|f(x)|=|E_k f(x)| \lesssim
M^{\mathfrak D}_{\ge r_0}f(x) \lesssim
(1+|x|)^{\frac{p_+n}{p_-}}\|f\|_{L^{p(\cdot)}(w)}
\end{equation}
In particular, 
we have
\begin{equation*}
\|f\|_{L^{p_\infty}(w)} \lesssim \|f\|_{L^{p(\cdot)}(w)}.
\end{equation*}
\end{lemma}

\begin{proof}

Fix $x \in {\mathbb R}^n$ and a cube $Q
\in {\mathfrak D}$ satisfying $x \in Q$ and $\ell(Q) \ge r_0$.
Then
by \cite[Corollary 3.7]{DiHapre}
and Lemma \ref{lem:210903-1}
(see also Remark \ref{rem:220420-1} below),
\begin{align*}
\frac{1}{|Q|}\int_Q|f(y)|{\rm d}y
&\le 
\frac{2}{|Q|}
\|\chi_Qf\|_{L^{p(\cdot)}(w)}
\|w^{-1}\chi_Q\|_{L^{p'(\cdot)}(w)}\\
&\le
\frac{2}{|Q|}
\|\chi_Q\|_{L^{p'(\cdot)}(\sigma)}\|f\|_{L^{p(\cdot)}(w)}\\
&\le
\frac{2\sigma(Q)^{\frac{1}{p'_Q}}}{|Q|}\|f\|_{L^{p(\cdot)}(w)}\\
&\lesssim
\left(\frac{[w]_{{A}_{p_-}^{\mathfrak D}}}{w(Q)}\right)^{\frac{1}{p_-}}\|f\|_{L^{p(\cdot)}(w)},
\end{align*}
where
$\sigma\equiv w^{-{\frac{1}{p(\cdot)-1}}}$ stands
for the dual weight.
Thus,
\begin{equation}\label{eq:210903-11}
\frac{1}{|Q|}\int_Q|f(y)|{\rm d}y
\lesssim
\left(\frac{[w]_{{A}_{p_-}^{\mathfrak D}}}{w(Q)}\right)^{\frac{1}{p_-}}\|f\|_{L^{p(\cdot)}(w)}.
\end{equation}


If $0 \notin 2Q$, then there exists the largest number $\ell \le -1$ such that 
$Q \subset Q_{\ell}^\dagger$.
By the geometric observation, we have $|x| \sim \ell(Q_{\ell}^\dagger)$.
By Lemma \ref{lem:210803-3}, since 
\[
w(Q) \gtrsim 
w(Q^\dagger_{\ell})\min\left(1,\frac{|Q|}{|Q_{\ell}^\dagger|}\right)^{p_+},
\] 
we have
\[
\frac{1}{w(Q)} \lesssim 
\frac{1}{w(Q_{\ell}^\dagger)} \left(1+\frac{|Q^\dagger_{\ell}|}{|Q|}\right)^{p_+}
\lesssim
\frac{1}{w(Q_{\ell}^\dagger)} \left(1+|x|\right)^{p_+n}
\le
\frac{1}{w(Q_{0}^\dagger)} \left(1+|x|\right)^{p_+n}
\]

Meanwhile, if $0 \in 2Q$,
then $64Q \supset [-r_0,r_0]^n$.
By the doubling condition, we have
\begin{equation}\label{eq:210903-13}
\frac{1}{w(Q)} \lesssim \frac{1}{w(64Q)} \le \frac{1}{w([-r_0,r_0]^n)} \lesssim1.
\end{equation}
Inserting this estimate into (\ref{eq:210903-11}),
we obtain
\begin{equation}\label{eq:210903-12}
M^{\mathfrak D}_{\ge r_0}f(x) \lesssim
(1+|x|)^{\frac{p_+n}{p_-}}\|f\|_{L^{p(\cdot)}(w)}.
\end{equation}
Then we have
\begin{equation*}
|f(x)|=|E_k f(x)| \lesssim
M^{\mathfrak D}_{\ge r_0}f(x) \lesssim
(1+|x|)^{\frac{p_+n}{p_-}}\|f\|_{L^{p(\cdot)}(w)}.
\end{equation*}
Moreover, $C\|f\|_{L^{p(\cdot)}(w)}^{-1}f$ satisfies 
the assumption of Lemma \ref{lem:210806-11} for some constant $C>0$.
By Lemma \ref{lem:210806-11} (i),
we have
\begin{equation}\label{eq:220104-1}
\|f\|_{L^{p_\infty}(w)} \lesssim \|f\|_{L^{p(\cdot)}(w)}.
\end{equation}
\end{proof}

\begin{lemma}\label{lem:210803-200}
Let $r_0$ be the same as in Lemma \ref{lem:210803-2}.
Let 
$w \in \tilde{A}_{p(\cdot)}^{\mathfrak D}$.
Then
$M^{\mathfrak D}_{\ge r_0}$
is bounded
on $L^{p(\cdot)}(w)$.
\end{lemma}

\begin{proof}
Let $c_1, c_2, $ and $k' \ll -1$ be
the same constants in
the proof of Corollary \ref{cor:220121-1}.
Then, the integer $k'$ satisfies 
$p_+({{\mathbb R}^n \setminus Q^\dagger_{k'}})-\varepsilon 
\le 
p_+({{\mathbb R}^n \setminus Q^\dagger_{k'}})-\varepsilon/3
\le
p_\infty$.
Thus we have
$w \in A_{p_\infty+\varepsilon}({\mathbb R}^n \setminus Q^\dagger_{k'})$
with
$$
[w]_{A_{p_\infty+\varepsilon}^{\mathfrak D}({\mathbb R}^n \setminus Q^\dagger_{k'})} \le 
C_0
[w]_{A_{p_+({{\mathbb R}^n \setminus Q^\dagger_{k'}})}^{\mathfrak D}({\mathbb R}^n \setminus Q^\dagger_{k'})} \le 
C_0{}^2
[w]_{\tilde{A}_{p(\cdot)}^{\mathfrak D}}
=c_1
$$
from Lemma \ref{lem:210803-1}.
By the property of $\varepsilon>0$,
we have
$
[w]_{A_{p_\infty}^{\mathfrak D}({\mathbb R}^n \setminus Q^\dagger_{k'})} 
\le c_2
$.

Let $k \gg 1$
have the same parity as $k'$.
Since
$M^{\mathfrak D}_{\ge r_0}f=M^{\mathfrak D}_{\ge r_0} \circ E_kf$,
we can assume $f=E_k f=\chi_{{\mathbb R}^n \setminus Q^\dagger_{k'}}E_k f$
thanks to Corollary \ref{cor:220121-1}.
Let us establish
\[
\|\chi_{Q^\dagger_{k'}}M^{\mathfrak D}_{\ge r_0}f\|_{L^{p(\cdot)}(w)}
+
\|\chi_{{\mathbb R}^n \setminus Q^\dagger_{k'}}M^{\mathfrak D}_{\ge r_0}f\|_{L^{p(\cdot)}(w)}
\lesssim\|f\|_{L^{p(\cdot)}(w)}.
\]

First, let $x \in Q_{k'}^\dagger$.
Since $\ell(Q_{k'}^\dagger) \ge 2>r_0$, 
all cubes $Q \in {\mathfrak D}$ satisfying $x \in Q$ and $\ell(Q) \ge r_0$
must either
include the cube $Q_{k'}^\dagger$
or be a cube in $\bigcup\limits_{j=k'}^{-\log_2 r_0}{\mathfrak D}_j$.
Thus, we can write such cubes $Q$ as $Q_{\ell}^\dagger$ for $\ell \le k'-1$. 
By virtue of these observations, we have
\[
M^{\mathfrak D}_{\ge r_0}f(x)
\lesssim
\sup_{\ell \in {\mathbb Z}, \ell \le k'-1} \frac{1}{|Q_{\ell}^\dagger|} 
\int_{Q_{\ell}^\dagger \setminus Q_{k'}^\dagger} |f(y)| {\rm d}y.
\]
Then, using the H\"older inequality, Lemma \ref{lem:220201-1},
Lemma \ref{lem:220202-2}
and the fact $Q_{k'}^\dagger \subset Q_{\ell}^\dagger$, we have
\begin{align}
\label{eq:211029-42}
\lefteqn{
\|\chi_{Q^\dagger_{k'}}M^{\mathfrak D}_{\ge r_0}f\|_{L^{p(\cdot)}(w)}
}\\
\nonumber
&\lesssim
\sup_{\ell \in {\mathbb Z}, \ell \le k'-1}m_{Q_{\ell}^\dagger \setminus Q_{k'}^\dagger}(|f|) 
\times \|\chi_{Q_{k'}^\dagger}\|_{L^{p(\cdot)}(w)}\\ \nonumber
&\sim
\sup_{\ell \in {\mathbb Z}, \ell \le k'-1}m_{Q_{\ell}^\dagger \setminus Q_{k'}^\dagger}(|f|) 
\times \|\chi_{Q_{k'-1}^\dagger \setminus Q_{k'}^\dagger}\|_{L^{p(\cdot)}(w)} \\
&\le
\sup_{\ell \in {\mathbb Z}, \ell \le k'-1}
\frac{2}{|Q_{\ell}^\dagger|}
\|f\|_{L^{p_\infty}(w)}
\|w^{-1}\chi_{Q_{\ell}^\dagger \setminus Q_{k'}^\dagger}\|_{L^{p'_\infty}(w)}
\|\chi_{Q_{\ell}^\dagger \setminus Q_{k'}^\dagger}\|_{L^{p_\infty}(w)} \nonumber\\
&\lesssim
\sup_{\ell \in {\mathbb Z}, \ell \le k'-1}
\frac{1}{|Q_{\ell}^\dagger|}\|\chi_{Q_{\ell}^\dagger \setminus Q_{k'}^\dagger}\|_{L^{p_\infty}(w)}
\|\chi_{Q_{\ell}^\dagger \setminus Q_{k'}^\dagger}\|_{L^{p_\infty'}(\sigma)}\|f\|_{L^{p_\infty}(w)} \nonumber\\
&\lesssim
\|f\|_{L^{p(\cdot)}(w)}. \nonumber
\end{align}
Next, we define
\[
{\mathfrak N}g(x)
\equiv
\sup_{R \in {\mathfrak D}, R \cap Q^\dagger_{k'}=\emptyset \mbox{ or } Q^\dagger_{k'} \subsetneq R}
\frac{\chi_{R \setminus Q^\dagger_{k'}}(x)}{|R \setminus Q^\dagger_{k'}|}
\int_{R \setminus Q^\dagger_{k'}}|f(y)|{\rm d}y.
\]
By Lemma \ref{lem:220201-1}, $w$ satisfies the condition
\[
\sup_{R \in {\mathfrak D}, R \cap Q^\dagger_{k'}=\emptyset \mbox{ or } Q^\dagger_{k'} \subsetneq R}
|R \setminus Q^\dagger_{k'}|^{-p_\infty}
\|w\|_{L^1(R \setminus Q^\dagger_{k'})}
\|w^{-1}\|_{L^{p_\infty'/p_\infty}(R \setminus Q^\dagger_{k'})}<\infty.
\]
Therefore, we have
\[
\|{\mathfrak N}g\|_{L^{p_\infty}(w)}
\lesssim
\|g\|_{L^{p_\infty}(w)}
\]
thanks to \cite[Theorem 1.1]{Jawerth86} and \cite[Theorem B]{Lerner08}.
Here, we can verify that $M^{\mathfrak D}_{\ge r_0}f(z) \le {\mathfrak N}f(z)$ for $z \in {\mathbb R}^n \setminus Q^\dagger_{k'}$.
Since
${\mathfrak N}$ is bounded on $L^{p_\infty}(w)$,
we deduce
\begin{equation}\label{eq:220104-2}
\|\chi_{{\mathbb R}^n \setminus Q^\dagger_{k'}}M^{\mathfrak D}_{\ge r_0}f\|_{L^{p_\infty}(w)}
\lesssim
\|\chi_{{\mathbb R}^n \setminus Q^\dagger_{k'}}{\mathfrak N}f\|_{L^{p_\infty}(w)}\\\lesssim
\|f\|_{L^{p_\infty}(w)}
\lesssim 
\|f\|_{L^{p(\cdot)}(w)}
\end{equation}
from Lemma \ref{lem:220202-2}.
Meanwhile, thanks to (\ref{eq:210903-12}), 
$C\|f\|_{L^{p(\cdot)}(w)}^{-1}\chi_{{\mathbb R}^n \setminus Q^\dagger_{k'}}M^{\mathfrak D}_{\ge r_0}f$ 
satisfies 
the assumption of Lemma \ref{lem:210806-11} for some constant $C>0$.
By (\ref{eq:220104-2}) and  Lemma \ref{lem:210806-11} (2),
we obtain
\[
\|\chi_{{\mathbb R}^n \setminus Q^\dagger_{k'}}M^{\mathfrak D}_{\ge r_0}f\|_{L^{p(\cdot)}(w)}
\lesssim
\|\chi_{{\mathbb R}^n \setminus Q^\dagger_{k'}}M^{\mathfrak D}_{\ge r_0}f\|_{L^{p_\infty}(w)}.
\]
By combining 
(\ref{eq:220104-2}) and (\ref{eq:220104-3}),
we obtain 
\begin{equation}\label{eq:220104-3}
\|\chi_{{\mathbb R}^n \setminus Q^\dagger_{k'}}M^{\mathfrak D}_{\ge r_0}f\|_{L^{p(\cdot)}(w)}
\lesssim
\|f\|_{L^{p(\cdot)}(w)}.
\end{equation}


Thus,
the desired result
is given
from (\ref{eq:211029-42}) and (\ref{eq:220104-3}).
\end{proof}
By combining Lemmas \ref{lem:210803-2}
and \ref{lem:210803-200},
we conclude that $M^{\mathfrak D}$ is bounded
on $L^{p(\cdot)}(w)$,
which implies that
$\tilde{A}_{p(\cdot)}^{\mathfrak D} \subset A_{p(\cdot)}^{\mathfrak D}$.
\subsection{Necessity of Theorem \ref{thm:9.3}}
Now let us prove
$\tilde{A}_{p(\cdot)}^{\mathfrak D} \supset A_{p(\cdot)}^{\mathfrak D}$.

Consider the converse.
We suppose that we have a weight $w$
such that
$M^{\mathfrak D}$ is bounded on $L^{p(\cdot)}(w)$.

A weight $w$ is doubling if
$w(\tilde{\tilde{Q}}) \lesssim w(Q)$
for any $Q \in {\mathfrak D}$,
where $\tilde{\tilde{Q}} \in {\mathfrak D}$ is the dyadic grand parent of $Q$,
That is,
$\tilde{\tilde{Q}}$ is
a cube $R \in {\mathfrak D}$ with $|Q|=4^{-n}|R|$ and $Q \subset R$.
We will use the following observation:
\begin{lemma}\label{lem:220131-4}
Let a doubling weight $w$, $p(\cdot)\in {\mathcal P}$ and $C>0$ satisfy
\begin{equation}\label{eq:220131-11}
\sup_{\lambda>0}\lambda\|\chi_{(\lambda,\infty]}(M^{{\mathfrak D}} f)\|_{L^{p(\cdot)}(w)}\le C
\|f\|_{L^{p(\cdot)}(w)},
\end{equation}
or equivalently
$M^{\mathfrak D}$ is weak bounded on $L^{p(\cdot)}(w)$.
Then
\[
\|\chi_Q\|_{L^{p(\cdot)}(w)}
\gtrsim
\min\left(1,\frac{|Q|}{|R|}\right)
\|\chi_R\|_{L^{p(\cdot)}(w)}
\]
for all
$Q,R \in {\mathfrak D}$ satisfying $Q \cap R \ne \emptyset$.
\end{lemma}

\begin{proof}
We can assume $\ell(Q) \le \ell(R)$;
otherwise the conclusion is trivial by the doubling property of $w$.
If we denote 
by $\tilde{\tilde{R}} \in {\mathfrak D}$,
which is the dyadic grand parent of $R$,
then 
\[
M^{{\mathfrak D}}\chi_Q \ge \frac{|Q \cap \tilde{\tilde{R}}|}{|\tilde{\tilde{R}}|}\chi_{\tilde{\tilde{R}}}
=
\frac{|Q|}{4^n|R|}\chi_{\tilde{\tilde{R}}}
\ge
\frac{|Q|}{4^n|R|}\chi_{R}.
\]
Therefore,
\[
R \subset
\left\{x \in {\mathbb R}^n\,:\,
M^{{\mathfrak D}}\chi_Q(x) \ge
\frac{|Q|}{4^n|R|}\chi_{R}(x)\right\}.
\]
Hence
from (\ref{eq:220131-11}), we have
\[
\|\chi_R\|_{L^{p(\cdot)}(w)}
\le
\left\|\chi_{(\frac{|Q|}{5^n|R|},\infty)}(M^{{\mathfrak D}}\chi_Q)
\right\|_{L^{p(\cdot)}(w)}
\le
\frac{5^n C|R|}{|Q|}\|\chi_Q\|_{L^{p(\cdot)}(w)},
\]
which proves Lemma \ref{lem:220131-4}.
\end{proof}

We prove the weighted analogy to Lemma \ref{lem:190613-3}.
\begin{lemma}\label{lem:211029-115}
Let
$p(\cdot) \in {\rm L H}_0 \cap {\rm L H}_\infty$
and
$w$
be a variable exponent
and a weight such that
$M^{\mathfrak D}$ is weak bounded on $L^{p(\cdot)}(w)$.
Then
for all $Q \in {\mathfrak D}_k$ with $k \ge 0$,
\begin{equation}\label{eq:220121-71}
\|\chi_Q\|_{L^{p(\cdot)}(w)}
\sim
w(Q)^{\frac{1}{p_Q}}
\sim
w(Q)^{\frac{1}{p_-(Q)}}
\sim
w(Q)^{\frac{1}{p_+(Q)}}.
\end{equation}
\end{lemma}
Before the proof, a couple of remarks may be in order.
\begin{remark}\label{rem:220420-1}\
\begin{enumerate}
\item
Notice that 
$M^{\mathfrak D}$ is not assumed bounded on $L^{p(\cdot)}(w)$.
However, it is absolutely necessary to assume that
$M^{\mathfrak D}$ is weak bounded on $L^{p(\cdot)}(w)$
(see Lemma \ref{lem:210907-1} below).
\item
As in \cite[Lemma 3.4]{DiHapre},
the same conclusion holds
for the case of $w \in A_\infty^{\rm loc,{\mathfrak D}}$.
In fact, 
assuming $w \in A_u^{\rm loc,{\mathfrak D}}$,
we have
(\ref{eq:220420-1})
with $p_+$ replaced by $u$,
which corresponds to
\cite[(3.5)]{DiHapre}.
\item
As in \cite[Corollary 3.7]{DiHapre},
if $w \in A_\infty^{\mathfrak D}$,
(\ref{eq:220121-71}) remains valid for any $Q \in {\mathfrak D}$.
In fact, 
assuming $w \in A_u^{\rm loc,{\mathfrak D}}$,
we have
(\ref{eq:220420-1})
with $p_+$ replaced by $u$,
which corresponds to the key inequality 
in the proof of
\cite[Corollary 3.7]{DiHapre}.
\end{enumerate}
\end{remark}

\begin{proof}
Let $k \ge 0$.
Fix $Q \in {\mathfrak D}_k$.
Recall that
$Q^\dagger_k \in {\mathfrak D}$
is the unique cube in ${\mathfrak D}_k$ containing $0$.
Then, we can find the smallest cube $Q_{\ell}^\dagger (\ell \le k)$ 
such that $Q \subset Q_{\ell}^\dagger$.
Due to Lemma \ref{lem:220131-4},
\[
\|\chi_{Q_{\ell}^\dagger}\|_{L^{p(\cdot)}(w)} \ge
\|\chi_{Q}\|_{L^{p(\cdot)}(w)} \gtrsim
\frac{|Q|}{|Q^\dagger_{\ell}|}\|\chi_{Q_{\ell}^\dagger}\|_{L^{p(\cdot)}(w)}.
\]
By the $\log$-H\"{o}lder condition, we obtain 
\[
(\|\chi_{Q}\|_{L^{p(\cdot)}(w)})^{\left|\frac{1}{p_-(Q)}-\frac{1}{p_+(Q)}\right|}\sim 1.
\]
Due to Lemma \ref{lem:190410-1},
we have
\[
\min(\|\chi_{Q}\|_{L^{p(\cdot)}(w)}^{p_-(Q)},\|\chi_{Q}\|_{L^{p(\cdot)}(w)}^{p_+(Q)})
\sim
\max(\|\chi_{Q}\|_{L^{p(\cdot)}(w)}^{p_-(Q)},\|\chi_{Q}\|_{L^{p(\cdot)}(w)}^{p_+(Q)})
\sim
w(Q).
\]
Thus, we obtain (\ref{eq:220121-71}).
\end{proof}

Now, let us investigate how fast $w$ grows.
\begin{lemma}\label{lem:210806-12}
Let
$p(\cdot) \in {\rm L H}_0 \cap {\rm L H}_\infty$
be a variable exponent such that
$M^{\mathfrak D}$ is weak bounded on $L^{p(\cdot)}(w)$.
Then
$w$ has at most polynomial growth. More precisely,
\begin{equation}\label{eq:211029-91}
w(\{y \in {\mathbb R}^n\,:\,|y| \le |x|\}) \lesssim (1+|x|)^{p_+n}
\end{equation}
for all $x \in {\mathbb R}^n$.
\end{lemma}

\begin{proof}
By Remark \ref{rem:200807-1} and Lemma \ref{lem:220131-4},
\[
\min\{w(Q^\dagger_k)^{\frac{1}{p_+(Q)}},w(Q^\dagger_k)^{\frac{1}{p_-(Q)}}\}
\le
\|\chi_{Q^\dagger_k}\|_{L^{p(\cdot)}(w)}
\lesssim
\frac{|Q^\dagger_k|}{|Q^\dagger_0|}
\|\chi_{Q^\dagger_0}\|_{L^{p(\cdot)}(w)}
\lesssim
2^{-k n}.
\]
As a result,
$
w(Q^\dagger_k) \lesssim 2^{-k p_+n}
$
for all $k \le 0$
or equivalently,
(\ref{eq:211029-91}) holds.
\end{proof}
We obtain a crude conclusion assuming that
$M^{\mathfrak D}$ is weak bounded on $L^{p(\cdot)}(w)$.
\begin{lemma}[{\rm c.f. \cite[Lemma 6.3]{DiHapre}}]
\label{lem:220121-11}
Let
$p(\cdot) \in {\rm L H}_0 \cap {\rm L H}_\infty$
be a variable exponent such that
$M^{\mathfrak D}$ is weak bounded on $L^{p(\cdot)}(w)$.
Then
$w \in A_\infty^{{\mathfrak D},{\rm loc}}$. That is, there exists $q>1$ such that
\[
\sup_{Q \in {\mathfrak D}, |Q| \le 1}
\frac{1}{|Q|}\int_Q w(x){\rm d}x
\left(\frac{1}{|Q|}\int_Q w(x)^{-\frac{1}{q-1}}{\rm d}x\right)^{q-1} \lesssim 1
\]
\end{lemma}

\begin{proof}
Let $Q \in {\mathfrak D}$ with $|Q| \le 1$
and $E \subset Q$ be a measurable set.
Using
\[
\|\chi_E\|_{L^{p(\cdot)}(w)} \le \max\{w(E)^{\frac{1}{p_+(Q)}},w(E)^{\frac{1}{p_-(Q)}}\},
\quad
Q \subset \{M^{\mathfrak D}\chi_E \ge |Q|^{-1}|E|\},
\]
we will show that
\begin{equation}\label{eq:210907-1112}
\frac{w(E)}{w(Q)}
\gtrsim
\left(\frac{|E|}{|Q|}\right)^{p_+}.
\end{equation}
Once this is achieved,
we will have $q \in [1,\infty)$
such that $w \in A_q^{\mathfrak D}$.
Then,
due to Lemma \ref{lem:211029-115},
\begin{align*}
\max\{w(Q)^{\frac{1}{p_+(Q)}},w(Q)^{\frac{1}{p_-(Q)}}\}
&\sim
\|\chi_Q\|_{L^{p(\cdot)}(w)}\\
&\le
\|\chi_{\{M^{\mathfrak D}\chi_E \ge |Q|^{-1}|E|\}}\|_{L^{p(\cdot)}(w)}\\
&\lesssim
\frac{|Q|}{|E|}
\|\chi_E\|_{L^{p(\cdot)}(w)}\\
& \le 
\frac{|Q|}{|E|}\max\{w(E)^{\frac{1}{p_+(Q)}},w(E)^{\frac{1}{p_-(Q)}}\}.
\end{align*}
As a result
\[
\frac{w(E)}{w(Q)}
\gtrsim
\min\left\{
\left(\frac{|E|}{|Q|}\right)^{{p_-}(Q)},
\left(\frac{|E|}{|Q|}\right)^{{p_+}(Q)}
\right\}
=
\left(\frac{|E|}{|Q|}\right)^{{p_+}(Q)}
\ge
\left(\frac{|E|}{|Q|}\right)^{{p_+}}.
\]
Thus, the proof of
(\ref{eq:210907-1112})
is complete.
\end{proof}

\begin{corollary}[{\rm \cite[Corollary 6.6]{DiHapre}}]\label{cor:210907-1}
Let
$p(\cdot),q(\cdot) \in {\rm L H}_0 \cap {\rm L H}_\infty$
be variable exponents such that
$M^{\mathfrak D}$ is weak bounded on $L^{p(\cdot)}(w)$.
Then
\begin{equation} \label{eq:220121-33}
\|\chi_Q\|_{L^{q(\cdot)}(w)} \sim w(Q)^{\frac{1}{q_Q}}
\end{equation}
for all cubes $Q \in {\mathfrak D}$.
\end{corollary}

This corollary seems to be
the same as Lemma \ref{lem:211029-115}.
However, 
we
remark that the weak boundedness is assumed
on $L^{p(\cdot)}(w)$ and 
 the equivalence for a different exponent $q(\cdot)$
is obtained.

\begin{proof}
Simply combine
Lemmas
\ref{lem:210806-11}
and
\ref{lem:210806-12}.
In fact, as in \cite[Lemma 3.4 and Corollary 3.7]{DiHapre}
(see also Remark \ref{rem:220420-1}),
we have
\begin{equation}\label{eq:220121-33a}
\|\chi_{Q}\|_{L^{q(\cdot)}(w)}
\sim
w(Q)^{\frac{1}{q_Q}},
\end{equation}
where we use Lemma \ref{lem:220121-11} if $|Q| \le 1$
but we use Lemma \ref{lem:190627-11}
and (\ref{eq:220121-33a}) for cubes having volume $1$ if $|Q| \ge 1$
\end{proof}
As the following lemma shows,
the dual space inherits the boundedness
of the operator $M^{\mathfrak D}$ from the original space.
\begin{lemma}\label{lem:210907-1}
Let
$p(\cdot) \in {\rm L H}_0 \cap {\rm L H}_\infty$
be a variable exponent such that
$M^{\mathfrak D}$ is bounded on $L^{p(\cdot)}(w)$.
Then
$M^{\mathfrak D}$ is weak bounded on $L^{p'(\cdot)}(\sigma)$,
where $\sigma\equiv w^{-{\frac{1}{p(\cdot)-1}}}$ stands for the dual weight.
\end{lemma}

\begin{proof}
Let $\lambda>0$ be fixed.
Also, let $f \in L^{p'(\cdot)}(\sigma)$.
By the duality
$L^{p(\cdot)}(w)$-$L^{p'(\cdot)}(\sigma)$ and the Stein type dual inequality, we have
\begin{align*}
\|\lambda\chi_{[\lambda,\infty]}(M^{\mathfrak D} f)\|_{L^{p'(\cdot)}(\sigma)}
&\sim
\sup_{g \in L^{p(\cdot)}(w), \|g\|_{L^{p(\cdot)}(w)}=1}
\int_{{\mathbb R}^n}\lambda\chi_{[\lambda,\infty]}(M^{\mathfrak D} f(x))|g(x)|{\rm d}x\\
&\lesssim
\sup_{g \in L^{p(\cdot)}(w), \|g\|_{L^{p(\cdot)}(w)}=1}
\int_{{\mathbb R}^n}|f(x)|M^{\mathfrak D}g(x){\rm d}x.
\end{align*}
Finally,
use the $L^{p(\cdot)}(w)$-boundedness of $M^{\mathfrak D}$
and the H\"{o}lder inequality.
\end{proof}

If we reexamine the proof of
Lemma \ref{lem:210907-1},
then we see that
Lemma \ref{lem:210907-1} holds
for a wider class of function spaces.
We summarize our observation below.
\begin{remark}
A Banach lattice over ${\mathbb R}^n$
is a Banach space
$({\mathcal X}({\mathbb R}^n),\|\cdot\|_{\mathcal X})$ contained in $L^0({\mathbb R}^n)$
such that,
for all $g \in {\mathcal X}({\mathbb R}^n)$ and $f \in L^0({\mathbb R}^n)$,
the implication
$\lq\lq|f| \le |g| \Rightarrow$ $f \in {\mathcal X}({\mathbb R}^n)$ 
and $\|f\|_{\mathcal X} \le \|g\|_{\mathcal X}$''
holds. 
The dual lattice ${\mathcal X}'({\mathbb R}^n)$ of ${\mathcal X}({\mathbb R}^n)$
is given by the set of all
$g \in L^0({\mathbb R}^n)$
for which
\[
\|f\|_{{\mathcal X}'}=
\sup\{\|f g\|_{L^1}\,:\,g \in {\mathcal X}\}
\]
is finite.
According to \cite{BeSh-text-98},
${\mathcal X}'$ is a Banach lattice over ${\mathbb R}^n$.
Lemma \ref{lem:210907-1} is available for Banach lattices.
Namely,
if ${\mathcal X}$ is a Banach lattice over ${\mathbb R}^n$
$\tilde{M}^{{\mathcal D}(Q)}$ is bounded on ${\mathcal X}$.
Then
$\tilde{M}^{{\mathcal D}(Q)}$ is weak bounded on ${\mathcal X}'({\mathbb R}^n)$.
As the example
of ${\mathcal X}({\mathbb R}^n)=L^1({\mathbb R}^n)$
shows,
it can happen that
$\tilde{M}^{{\mathcal D}(Q)}$
is not bounded on ${\mathcal X}({\mathbb R}^n)$.
\end{remark}

We conclude the proof of necessity.
Thus, we suppose that there exists a constant $C>0$
such that
\[
\|M^{\mathfrak D} f\|_{L^{p(\cdot)}(w)}
\le C
\|f\|_{L^{p(\cdot)}(w)}.
\]
Fix a cube $Q \in {\mathfrak D}$.
Then we have
\[
M^{\mathfrak D}[\sigma \chi_Q](x) \ge \frac{\sigma(Q)}{|Q|}\chi_Q(x).
\]
As a result,
\begin{equation}\label{eq:210907-111}
\frac{\sigma(Q)}{|Q|}\|\chi_Q\|_{L^{p(\cdot)}(w)}
\le C
\|\chi_Q\|_{L^{p(\cdot)}(\sigma)}.
\end{equation}
Note that $w$ has at most polynomial growth thanks to
Lemma \ref{lem:210806-12}.
Additionally it should be observed that
$M^{\mathfrak D}$ is weak bounded on $L^{p'(\cdot)}(\sigma)$
thanks to Lemma \ref{lem:210907-1}.
Thus, we conclude from 
Lemma \ref{lem:210806-12}
that $\sigma$ has at most polynomial growth.
Thus, 
Corollary \ref{cor:210907-1}
can be applied to both $w$ and $\sigma$.
Due to Corollary \ref{cor:210907-1},
we have
\[
\|\chi_Q\|_{L^{p(\cdot)}(w)}\sim w(Q)^{\frac{1}{p_Q}}, \quad
\|\chi_Q\|_{L^{p(\cdot)}(\sigma)} \sim \sigma(Q)^{\frac{1}{p_Q}}.
\]
Inserting these estimates into (\ref{eq:210907-111}),
we obtain
\[
\frac{\sigma(Q)}{|Q|}w(Q)^{\frac{1}{p_Q}}
\le C
\sigma(Q)^{\frac{1}{p_Q}},
\]
or equivalently,
\[
|Q|^{-p_Q}\|w\|_{L^1(Q)}\sigma(Q)^{p_Q-1} \le C,
\]
where constant $C$ is independent of $Q$.
If we use Corollary \ref{cor:210907-1} once again, we conclude
\[
|Q|^{-p_Q}\|w\|_{L^1(Q)}\|\sigma\|_{L^{\frac{p(\cdot)}{p'(\cdot)}}(Q)} \le C,
\]
as required.

\section*{Acknowledgement}

\noindent
The authors thank  Professor Lars Diening
for the preprint \cite{DiHapre},
which motivated us to complete Section \ref{s88}.
Mitsuo Izuki was partially supported 
by Grand-in-Aid for Scientific Research (C), No.\,15K04928, 
for Japan Society for the Promotion of Science. 
Toru Nogayama 
was supported financially by Research Fellowships of the Japan Society for
the Promotion of Science for Young Scientists (20J10403 and 22J00614).
Takahiro Noi was partially supported by Grand-in-Aid for Young Scientists (B), No.\,17K14207, for Japan Society for the Promotion of Science. 
Yoshihiro Sawano was partially supported by Grand-in-Aid for Scientific Research (C), No.\,19K03546, for Japan Society for the Promotion of Science. 
This work was partly supported by Osaka City University Advanced Mathematical Institute (MEXT Joint Usage/Research Center on Mathematics and Theoretical Physics).

Mitsuo Izuki,\\
Faculty of Liberal Arts and Sciences, \\
Tokyo City University, \\
1-28-1, Tamadutsumi Setagaya-ku Tokyo 158-8557, Japan. \\ 
E-mail: izuki@tcu.ac.jp

\smallskip

Toru Nogayama (Corresponding author),\\
Graduate School of Science and Engineering,\\
Chuo University, 1-13-27 Kasuga, Bunkyo-Ku, Tokyo,
112-8551, Japan.\\
E-mail: toru.nogayama@gmail.com

\smallskip

Takahiro Noi,\\
Center for Basic Education and Integrated Learning,\\
Kanagawa Institute of Technology,\\
1030 Shimoogino Atsugi-city Kanagawa, 243-0292, Japan.\\
E-mail: taka.noi.hiro@gmail.com

\smallskip

Yoshihiro Sawano,\\
Graduate School of Science and Engineering,\\
Chuo University, 1-13-27 Kasuga, Bunkyo-Ku, Tokyo,
112-8551, Japan\\
+\\
People's Friendship University of Russia.\\
\\
E-mail: yoshihiro-sawano@celery.ocn.ne.jp
\end{document}